\pdfoutput=1
\documentclass[a4paper,11pt]{article}
\usepackage{amsmath,amssymb,amsthm,graphicx,inputenc}
\usepackage{hyperref,url,enumitem,bm,esint,color}
\usepackage[margin=30mm]{geometry}

\newcommand{\ov}[1]{\overline{#1}}
\newcommand{\bs}[1]{\boldsymbol{#1}}
\renewcommand{\hat}[1]{\widehat{#1}}
\newcommand{\eps}{\varepsilon}
\newcommand{\vp}{\varphi}
\newcommand{\E}{\mathcal{E}}
\newcommand{\BV}{\mathrm{BV}}
\newcommand{\pd}{\partial}
\newcommand{\0}{\bm{0}}
\newcommand{\bph}{{\boldsymbol{\varphi}}}
\renewcommand{\S}{\mathcal{S}}
\newcommand{\bth}{\bm{\vartheta}}
\newcommand{\bphi}{\bm{\phi}}
\newcommand{\bxi}{\bm{\xi}}
\newcommand{\bz}{{\boldsymbol{\zeta}}}
\newcommand{\bu}{{\overline{\uu}}}
\newcommand{\hu}{{\widehat{\uu}}}
\newcommand{\bv}{{\overline{\vv}}}
\newcommand{\hv}{{\widehat{\vv}}}

\newcommand{\hw}{{\widehat{\ww}}}
\newcommand{\bU}{\overline{\bm{U}}}
\newcommand{\hU}{\widehat{\bm{U}}}
\newcommand{\bp}{\overline{\bm{p}}}
\newcommand{\hq}{\widehat{\bm{q}}}
\newcommand{\bP}{\overline{\bm{P}}}
\newcommand{\hQ}{\widehat{\bm{Q}}}
\newcommand{\HD}{H^1_D(\Omega, \RRR^d)}
\newcommand{\ovHD}{\ov H^1_D(\Omega, \RRR^d)}
\newcommand{\hatHD}{\hat H^1_D(\Omega, \RRR^d)}
\newcommand{\Vad}{\mathcal{V}_{\mathrm{ad}}}
\newcommand{\Tad}{\mathcal{T}_{\mathrm{ad}}}

\newcommand{\dy}{\, \mathrm{d}y}
\newcommand{\dt}{\, \mathrm{d}t}
\newcommand{\ds}{\, \mathrm{d}s}
\newcommand{\dx}{\, \mathrm{dx}}
\newcommand{\dH}{\, \mathrm{d}\mathcal{H}^{d-1}}
\newcommand{\Uad}{\mathcal{U}_{\rm ad}}
\newcommand{\Jred}{J_{\rm red}}
\newcommand{\inn}[1]{\langle #1 \rangle}
\newcommand{\no}[1]{\| #1 \|}
\newcommand{\V}{\mathcal{V}}
\def\norma #1{\mathopen \| #1\mathclose \|}
\def\<#1>{\mathopen\langle #1\mathclose\rangle}
\def\iO #1{\int_\Omega #1 \dx}

\renewcommand{\tilde}{\widetilde}
\newcommand{\x}{\bm{x}}
\newcommand{\s}{\bm{s}}
\newcommand{\y}{\bm{y}}
\newcommand{\Id}{\mathrm{Id}}

\newcommand{\bigchi}{\ensuremath{\mathrm{\mathcal{X}}}}

\theoremstyle{plain}
\newtheorem{thm}{Theorem}[section]
\newtheorem{lem}{Lemma}[section]
\newtheorem{prop}{Proposition}[section]
\newtheorem{cor}{Corollary}[section]
\newtheorem{remark}{Remark}[section]

\newtheorem{defn}{Definition}[section]

\numberwithin{equation}{section}

\def\multibold #1{\def\arg{#1}%
  \ifx\arg\pto \let\next\relax
  \else
  \def\next{\expandafter
    \def\csname #1#1\endcsname{{\bf #1}}%
    \multibold}%
  \fi \next}

\def\pto{.}

\def\multimathbb #1{\def\arg{#1}%
  \ifx\arg\pto \let\next\relax
  \else
  \def\next{\expandafter
    \def\csname #1#1#1\endcsname{{\mathbb #1}}%
    \multimathbb}%
  \fi \next}
  
\def\multical #1{\def\arg{#1}%
  \ifx\arg\pto \let\next\relax
  \else
  \def\next{\expandafter
    \def\csname cal#1\endcsname{{\cal #1}}%
    \multical}%
  \fi \next}

\def\multimathop #1 {\def\arg{#1}%
  \ifx\arg\pto \let\next\relax
  \else
  \def\next{\expandafter
    \def\csname #1\endcsname{\mathop{\rm #1}\nolimits}%
    \multimathop}%
  \fi \next}

\multibold
qwertyuiopasdfghjklzxcvbnmQWERTYUIOPASDFGHJKLZXCVBNM.

\multimathbb
QWERTYUIOPASDFGHJKLZXCVBNM.

\multical
QWERTYUIOPASDFGHJKLZXCVBNM.

\multimathop
diag dist div dom mean meas sign supp .

\title{Phase field topology optimisation for 4D printing}
\author{Harald Garcke \footnotemark[1] \and Kei Fong Lam \footnotemark[2] \and Robert N\"urnberg \footnotemark[3] \and Andrea Signori \footnotemark[4]}
\date{ }

\begin{document}
\maketitle

\begin{abstract} 
\noindent
This work concerns a structural topology optimisation problem for 4D printing based on the phase field approach. The concept of 4D printing as a targeted evolution of 3D printed structures can be realised in a two-step process. One first fabricates a 3D object with multi-material active composites and apply external loads in the {\it programming stage}. Then, a change in an environmental stimulus and the removal of loads cause the object deform in the {\it programmed stage}. The dynamic transition between the original and deformed shapes is achieved with appropriate applications of the stimulus. 
The mathematical interest is to find an optimal distribution for the materials such that the 3D printed object achieves a targeted configuration in the programmed stage as best as possible.

Casting the problem as a PDE-constrained minimisation problem, we consider a vector-valued order parameter representing the volume fractions of the different materials in the composite as a control variable. We prove the existence of  optimal designs and formulate first order necessary conditions for minimisers. Moreover, by suitable asymptotic techniques, we relate our approach to a sharp interface description. Finally, the theoretical results are validated by several numerical simulations both in two and three space dimensions.
\end{abstract}

\noindent {\bf Keywords:}
4D printing, Printed active composites, Topology optimisation, Phase field, Linear elasticity, Optimal control
\vskip3mm
\noindent {\bf AMS (MOS) Subject Classification:} {
		49J20, 
		49K40, 
		49J50  

\renewcommand{\thefootnote}{\fnsymbol{footnote}}
\footnotetext[1]{Fakult{\"a}t f\"ur Mathematik, Universit{\"a}t Regensburg, 93040 Regensburg, Germany
({\texttt Harald.Garcke@mathematik.uni-regensburg.de}).}
\footnotetext[2]{Department of Mathematics, Hong Kong Baptist University, Kowloon Tong, Hong Kong ({\texttt akflam@math.hkbu.edu.hk}).}
\footnotetext[3]{Department of Mathematics, University of Trento, Trento, Italy
({\texttt robert.nurnberg@unitn.it}).}
\footnotetext[4]{Department of Mathematics, Politecnico di Milano, 20133 Milano, Italy ({\texttt andrea.signori@polimi.it}).}

\section{Introduction}
Four-dimensional (4D) printing \cite{Ali,Momeni,Tibbits} entails the combination of additive manufacturing (3D printing) and active material technologies to create printed composites capable of morphing into different configurations in response to various environmental stimuli. First designs of such composites consist of active material components, such as piezoelectric ceramics, hydrogels or shape memory polymers \cite{Ding}, in the form of fibers integrated within a passive elastomeric matrix \cite{Ge}. These multi-material active composites were originally difficult to manufacture, owing to the fragility of the materials involved \cite{Liu}. However, with 3D printing techniques it is nowadays feasible to fabricate these active composites to a high degree of precision, resulting in so-called {\it printed active composites} (PACs) \cite{Maute}. For an overview of other 4D printing strategies besides PACs in the construction of smart materials allowing direct stimuli-responsive transformations, we refer to \cite{Yuan}.

The shape shifting functionality of the active components enables the self-actuating and self-assembling potentials of PACs, allowing them to fold, bend, twist, expand and contract when a stimulus is applied, and return to their original configurations after the stimulus is removed. This property has led to the fabrication of intelligent active hinges and origami-like objects \cite{GeOrigami,Ge}, mesh structures \cite{Ding,Wang} and self-actuated deformable solids \cite{Sun} in the form of in the form of graspers and smart key-lock systems. We refer to the review article \cite{Momeni} and the references cited therein for more applications of 4D printing. The shape memory behaviour of the PACs can be programmed in a two-step cycle: The first (programming) step involves deforming the structure from its permanent shape to a metastable temporary shape, and the second (recovery) step involves applying an appropriate stimulus so that the structure regains its original shape. A typical stimulus is heat (in combination with light \cite{Kuksenok} or water \cite{Bakarich}), in which programmed PACs alter their shapes when the temperature rises above or drops below a critical value.

With the advances in the state-of-the art 3D printing technologies, the designs of PACs need not be limited to the conventional fibre-matrix architectures first considered in \cite{Ge}. In particular, the distribution of active and passive materials in the designs can take on more complicated geometries to better fulfil the intended functionalities of the PACs. This opens up the possibility of a computational design approach guided by a structural topology optimisation framework. In the context of 3D printing, see, e.g., the review \cite{Jiang}, this framework has been applied to explore optimising support structures to overhang regions \cite{Allaire,Langelaar,Mirzendehdel}, as well as self-support designs respecting the overhang angle constraints \cite{Cacace,GLNS,Leray,Liu:over}. For active materials and active composites, \cite{Howard,Pajot} studied how to pattern thin-film layers within a multi-layer structure with the aim of generating large shape changes via spatially varying eigenstrains within the microstructures, while \cite{Maute} aimed to optimise the microstructures of PACs matching various target shapes after a thermomechanical training and activation cycle. Later works incorporated nonlinear thermoelasticity \cite{Geiss,Sun}, thermo-mechanical cycles of shape memory polymers \cite{Bhattacharyya}, reversible deformations \cite{Lumpe}, as well as multi-material designs \cite{Wei} within the topology optimisation framework. 

In many of the aforementioned contributions, the topology optimisation is implemented numerically with the level-set method or the solid isotropic material with penalisation (SIMP) approach. In this work we employ an alternative approach based on the phase field methodology \cite{Bourdin}, which allows a straightforward extension to the multiphase setting \cite{BGFS,WangZ04} involving multiple (possibly distinct) types of active materials within the design.  In particular, this opens up the design to multiphase PACs that can memorise more than two shapes \cite{Ge:Mult,Li,Sun:Mult,Wan,Wu}. The phase field-based structural topology optimisation approach has been popularised in recent years by many authors, with applications in nonlinear elasticity \cite{Penzler}, stress constraints \cite{Burger}, compliance optimisation \cite{BlankGSSSV12,Take}, elastoplasticity \cite{Almi}, eigenfrequency maximisation \cite{GHK,Take}, graded-material design \cite{Carr}, shape optimisation in fluid flow \cite{GHechtNS,GHechtStokes,GHHKL} and more recently for 3D printing with overhang angle constraints \cite{GLNS}. 

Taking inspiration from the setting of Maute et al.~\cite{Maute}, we formulate a structural topology optimisation problem for a multiphase PAC with the objective of finding optimal distributions of active and passive materials so that the resulting composite matches targeted shapes as close as possible. An additional perimeter regularisation term, in the form of a multiphase Ginzburg--Landau functional, is added, and our contribution involves a mathematical analysis of the resulting multiphase structural topology optimisation problem with emphasis on the rigorous derivation of minimisers and optimality conditions. A sharp interface asymptotic analysis is performed to obtain a set of optimality conditions applicable in a level set-based shape optimisation framework. We perform numerical simulations in two and three spatial dimensions to show the optimal distributions of active and passive components in order to match with various target shapes for the PACs.

The rest of this paper is organised as follows: in Section \ref{SEC:PROB} we formulate the phase field structural optimisation problem to be studied, and present several preliminary mathematical results.  In Sections \ref{SEC:ANA} and \ref{SEC:OPT} we analyse the design optimisation problem and establish analytical results concerning minimisers and optimality conditions. The sharp interface limit is explored in Section \ref{SEC:Sharp} and, finally, in Section \ref{SEC:NUM} we present the numerical discretisation and several simulations of our approach.

\section{Problem formulation}\label{SEC:PROB}
Within a bounded domain $\Omega \subset \RRR^d$, $d \in \{2,3\}$, with Lipschitz boundary $\Gamma := \pd \Omega$, we assume there are $L$ types of linearly elastic materials, whose volume fractions are encoded with the help of a vectorial phase field variable $\bph = (\vp_1, \dots, \vp_L) : \Omega \to \Delta^L$, where $\Delta^L$ denotes the Gibbs simplex in $\RRR^L$:
\[
\Delta^L := \Big \{ \x=(x_1,...,x_L) \in \RRR^L \, : \, \sum_{i=1}^L x_i = 1, \, x_i \geq 0 \text{ for all } i \in \{1, \dots, L\} \Big \}.
\]
For our application to PACs, we take $\vp_L$ as the volume fraction of the passive elastic material, and $\vp_1, \dots, \vp_{L-1}$ as the volume fractions of (possibly different) active elastic materials. Note that in the two-phase case $L = 2$, we simply have $\bph = (\vp_1, \vp_2)$, and due to the relation $\vp_1 + \vp_2 = 1$ we may instead use the scalar difference function $\vp := \vp_1 - \vp_2$ to encode $\bph$ via the relation $\bph = (\frac{1}{2}(1+\vp), \frac{1}{2}(1-\vp))$.
This particular scenario will be employed later on, when dealing with the connection between the problem we are going to analyse and the corresponding sharp interface limit in Section \ref{SEC:Sharp}, as well as for the numerical simulations presented in Section \ref{SEC:NUM}.

The shape shifting mechanism considered in \cite{Maute,Zhang} involves two levels of temperature and one set of external loads, with one temperature $T_H$ higher than a critical transition temperature $T_g$ of the active materials (e.g., the glass transition temperature for shape memory polymers), and the other temperature $T_L$ lower than the critical temperature. The printed composite is first heated to $T_H$, and the shape memory cycle starts at $T_H$ and proceeds as follows: First, external loads are applied to deform the printed composite while the temperature remains at $T_H$, with the new configuration being known as the {\it programming stage} (or Stage 1). Then, the temperature is decreased while the loads are maintained on the printed composite, which are then removed once the temperature reached $T_L$. The resulting shape at $T_L$ is the desired shape and we denote it as the {\it programmed stage} (or Stage 2). Increasing the temperature to $T_H$ enables the printed composite to recover its original shape, and this ends the shape memory cycle, see Figure~\ref{fig:Maute} for the thermo-mechanical processing steps involving the two stages.

\begin{figure}[h]
\centering
\includegraphics[scale=0.38]{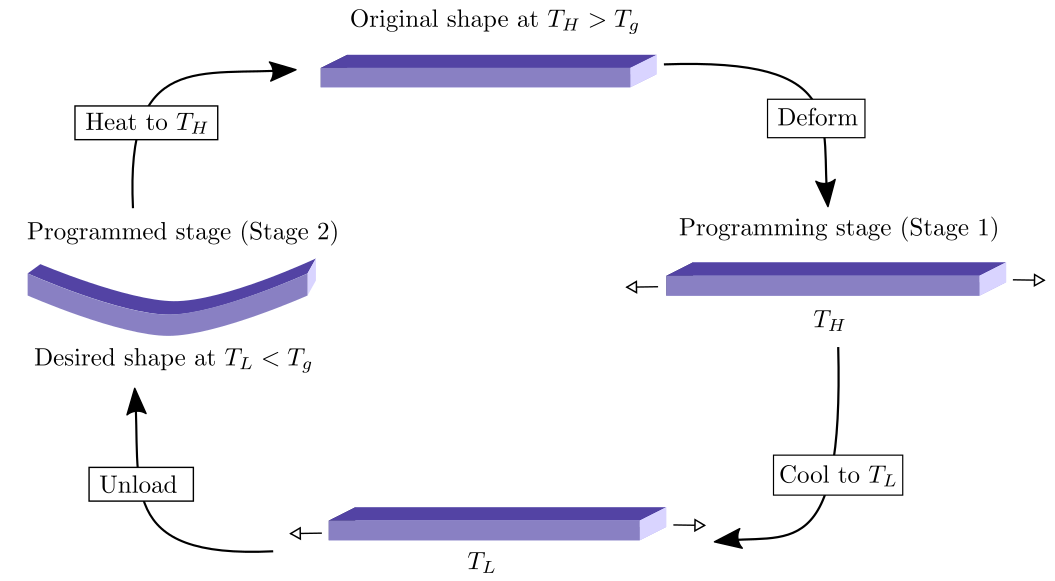}
\caption{Schematics of the the shape memory cycle from \cite{Maute} involving a programming stage (Stage 1) and a programmed stage (Stage 2).}
\label{fig:Maute}
\end{figure}

To capture the above behaviour, following \cite{Maute} we consider a model for each stage. In the programming stage (Stage 1), we consider an elastic displacement $\bu : \Omega \to \RRR^{d}$ and decompose the domain boundary $\Gamma$ into a partition $\Gamma = {\rm cl}(\ov{\Gamma}_D) \cup {\rm cl}(\ov{\Gamma}_N)$ with relative open subsets $\ov{\Gamma}_D$ and $\ov{\Gamma}_N$ such that $\ov{\Gamma}_D \cap \ov{\Gamma}_N = \emptyset$ and $\ov{\Gamma}_D \neq \emptyset$, where ${\rm cl}(A)$ denotes the closure of a set $A$, and we assign a prescribed displacement $\bU$ on $\ov{\Gamma}_D$ and surface loads $\ov{\gg}$ on $\ov{\Gamma}_N$.  Under a linearised elasticity setting, the balance of momentum yields the following system of equations for the displacement $\bu$:
\begin{subequations}\label{bu:sys}
\begin{alignat}{2}
\label{SYS:1} -\div \big (\ov{\CCC}(\bph) \E(\bu) \big ) & = \ov{\FF} && \quad \text{ in } \Omega, \\
\bu & = \bU && \quad  \text{ on } \ov{\Gamma}_D, \\
\big (\ov{\CCC}(\bph) \E(\bu) \big )\nn &= \ov{\gg} && \quad  \text{ on } \ov{\Gamma}_N,
\end{alignat}
\end{subequations}
with a phase-dependent elasticity tensor $\ov{\CCC}$, body force $\ov{\FF}$, outer unit normal $\nn$, and symmetrised gradient $\E(\bu)$.   One example of $\ov{\CCC}(\bph)$ is
\[
\ov{\CCC}(\bph(\x)) = \sum_{i=1}^L \ov{\CCC}_i \vp_i (\x) \quad \text{ for } \bph(\x) = (\vp_1(\x), \dots, \vp_L(\x)) \in \Delta^L, \quad \x \in \Omega,
\]
with constant tensors $\ov{\CCC}_i$, $1 \leq i \leq L$. 

After the change in temperature from $T_H$ to $T_L$ and after the programming loads in Stage 1 have been removed, the PAC experiences deformations due to residual stresses generated during the thermomechanical processing steps.  When the temperature falls below $T_g$, the active elastic materials undergo a phase transition from a soft rubbery state to a glassy state that has a higher Young's modulus. We introduce a new variable $\hu : \Omega \to \RRR^{d}$ to denote the displacement in the programmed stage (Stage 2), and as in \cite{Maute}, model the strains from the programming stage (Stage 1) as eigenstrains for $\hu$.  These eigenstrains are present only in the regions of active elastic materials, which we model with a fixity function $\chi: \RRR^L \to [0,\infty)$. The shape fixity for a shape memory material is the ratio (expressed as a percentage) between the strain in the stress-free state after the programming step and the maximum strain \cite{Abdullah}. For example, if the deformation elongates the material, the fixity quantifies the ability of the material to hold the temporary elongated length when the stress is removed.  It is clear from the definition that for a passive elastic material, the fixity is zero, and so we set that $\chi = 0$ in the region $\{\vp_L = 1\}$ of the passive elastic material. Decomposing the domain boundary $\Gamma$ into a possibly different partition $\Gamma = {\rm cl}(\hat{\Gamma}_D) \cup {\rm cl}(\hat{\Gamma}_N)$ with relative open subsets $\hat{\Gamma}_D$ and $\hat{\Gamma}_N$ such that $\hat{\Gamma}_D \cap \hat{\Gamma}_N = \emptyset$ and $\hat{\Gamma}_D \neq \emptyset$, where we assign a prescribed displacement $\hU$ on $\hat{\Gamma}_D$ and surface loads $\hat{\gg}$ on $\hat{\Gamma}_N$, the equations for the programmed stage (Stage 2) read as
\begin{subequations}\label{hu:sys}
\begin{alignat}{2}
\label{SYS:2} -\div \big (\hat{\CCC}(\bph) (\E(\hu) - \chi(\bph) \E(\bu)) \big )&  = \hat{\FF} && \text{ in } \Omega, \\
\hu & = \hU && \text{ on } \hat{\Gamma}_D, \\
\big ( \hat{\CCC}(\bph) (\E(\hu) - \chi(\bph) \E(\bu))\big ) \nn & = \hat{\gg} && \text{ on } \hat{\Gamma}_N,
\end{alignat}
\end{subequations}
with a phase-dependent elasticity tensor $\hat{\CCC}$ and body force $\hat{\FF}$.  In the above, the change in the elasticity tensor from $\ov{\CCC}$ in Stage 1 to $\hat{\CCC}$ in Stage 2 encodes the change in the elastic properties of the PAC when the temperature changes from $T_H$ to $T_L$. Similarly to \cite{Maute}, here we have neglected the strains arising from thermal expansion in \eqref{bu:sys} and \eqref{hu:sys}.

In the next section, under a suitable functional framework, we demonstrate that \eqref{bu:sys} and \eqref{hu:sys} are uniquely solvable, with the solution depending continuously on $\bph$.  Since $\bph$ controls the distribution of the passive and active elastic materials, it is natural to ask for specific material distributions that optimise certain cost functionals related to the design of PACs. Motivated from \cite{Maute}, we primarily focus on the following cost functional
\begin{align}\label{cost}
J(\bph, \hu) := \frac{1}{2} \int_{\Gamma^{\rm tar}}  \Big (W(\hu - \bm{u}^{\rm tar}) \Big ) \cdot (\hu - \bm{u}^{\rm tar}) \dH +  \gamma \iO{ \eps |\nabla \bph|^2 + \frac{1}{\eps} \Psi(\bph)},
\end{align}
where $\gamma > 0$ is a weighting factor, $\hu$ is a solution to \eqref{hu:sys} depending on $\bph$ (and also on $\bu$, a solution to \eqref{bu:sys}), $W \in \RRR^{d \times d}$ is a fixed weighting matrix, $\Gamma^{\rm tar}$ is a subset of the boundary $\hat{\Gamma}_N$, $\eps > 0$ is a fixed constant related to the thickness of the interfacial regions $\{0 < \vp_i <1\}$, $i\in\{1,...,L\}$, $\mathcal{H}^{d-1}$ indicates the standard $(d-1)$-dimensional Hausdorff measure, and $\Psi : \RRR^L \to \RRR$ is a non-negative multi-well potential that attains its minimum at the corners $\{\bm{e}_1, \dots, \bm{e}_L\}$ (the unit vectors in $\RRR^L$) of the Gibbs simplex $\Delta^L$. 
The first term in \eqref{cost} consists of a target shape matching term, where we like to match the displacement $\hu$ in Stage 2 with a prescribed deformation $\uu^{\rm tar}$ over the surface $\Gamma^{\rm tar} \subset \hat{\Gamma}_N$ by minimising the squared difference weighted by a matrix $W$.  The second term is the well-known Ginzburg--Landau functional in the multiphase setting that serves as a form of perimeter regularisation.  It provides some form of regularity to our design solutions and penalises designs that have large interfaces between the different phases of elastic materials.

For our problem we introduce the design space
\[
\Uad = \Big \{ \bph \in H^1(\Omega, \RRR^L) \, : \, \bph(\x) \in \Delta^L \text{ for a.e.~} \x \in \Omega \Big \},
\]
and our design problem can be formulated as the following
\begin{align*}
{\bf (P)} \quad & \min_{\bph \in \Uad} J(\bph, \hu) \text{ subject to } (\bph, \bu,\hu) \text{ is a solution to } \eqref{bu:sys}-\eqref{hu:sys}.
\end{align*}

\begin{remark}
For the existence theory for optimal designs to $(\bf P)$, it is also possible to consider a more general form of the cost functional:
\begin{align*}
J(\bph, \bu, \hu) & = \iO{h_\Omega(\x, \bu, \hu)} 
 + \int_{\ov{\Gamma}_N} \ov{h}({\bm s}, \bu) \dH + \int_{\hat{\Gamma}_N} \hat{h}(\s, \hu) \dH
\\
& \quad + \gamma \iO{\eps |\nabla \bph|^2 + \frac{1}{\eps} { \Psi(\bph)}},
\end{align*}
with Carath\'eodory functions $h_\Omega$, $\ov{h}$ and $\hat{h}$ satisfying (see \cite[Remark 5]{BGHR})
\begin{alignat*}{2}
 |h_\Omega(\x, \bm{v}, \bm{w})| & \leq a_1(\x) + b_1(\x) |\bm{v}|^p + b_2(\x) |\bm{w}|^p \quad &&\text{ for all } \bm{v}, \bm{w} \in \RRR^d, \text{ a.e.~} \x \in \Omega, \\
 |\ov{h}(\s, \bm{v})|  & \leq a_2(\s) + c_1(\s) |\bm{v}|^2  \quad  && \text{ for all } \bm{v} \in \RRR^d,\text{ a.e.~} \s \in \ov{\Gamma}_N, \\
 |\ov{h}(\s, \bm{w})| &  \leq a_3(\s) + c_2(\s) |\bm{w}|^2  \quad && \text{ for all } \bm{w} \in \RRR^d, \text{ a.e.~} \s \in \hat{\Gamma}_N,
\end{alignat*}
for any $2 \leq p < \infty$ if $d = 2$ and $2 \leq p < 6$ if $d = 3$, with functions $a_1 \in L^1(\Omega)$, $b_1, b_2 \in L^\infty(\Omega)$, $a_2 \in L^1(\ov{\Gamma}_N)$, $a_3 \in L^1(\hat{\Gamma}_N)$, $c_1 \in L^\infty(\ov{\Gamma}_N)$ and $c_2 \in L^\infty(\hat{\Gamma}_N)$. In our current setting we have $h_\Omega = \ov{h} = 0$ and $\hat{h}(\s, \hu) = \frac{1}{2} \bigchi_{\Gamma^{\rm tar}}(\s) W(\hu - \bm{u}^{\rm tar}) \cdot (\hu - \bm{u}^{\rm tar})$.  
\end{remark}

\begin{remark}
It is also possible to consider mass constraints for $\bph$ of the form
\begin{align*}
\frac 1 {|\Omega|} \iO{\bph} \leq \bm{\alpha} \quad \text{ or } \quad \frac 1 {|\Omega|} \iO{\bph} \geq  \bm{\beta} \quad \text{ or } \quad \bm{\beta} \leq \frac 1 {|\Omega|} \iO{\bph} \leq \bm{\alpha},
\end{align*}
for fixed vectors $\bm{\alpha}, \bm{\beta} \in \Delta^L$ (possibly also $\bm{\alpha}=\bm{\beta}$), where in the above the inequalities are taken component-wise. These are convex constraints and thus when included into the definition of $\Uad$, the design space remains a closed and convex set. Then, in the corresponding necessary optimality condition, associated Lagrange multipliers will appear, see \cite{BGFS,BGHR} for more details. 
\end{remark}

\paragraph{Notation.} For a Banach space $X$, we denote its topological dual by $X^*$, and the corresponding duality pairing by $\inn{\cdot,\cdot}_X$. For any $p \in [1,\infty]$ and $k >0$, the standard Lebesgue and Sobolev spaces over $\Omega$ are denoted by $L^p := L^p(\Omega)$ and $W^{k,p} := W^{k,p}(\Omega)$ with the corresponding norms 
$\no{\cdot}_{L^p(\Omega)}$ and $\no{\cdot}_{W^{k,p}(\Omega)}$.
In the special case $p = 2$, these become Hilbert spaces and we employ the notation $H^k := H^k(\Omega) = W^{k,2}(\Omega)$ with the corresponding norm $\no{\cdot}_{H^k(\Omega)}$.  For our subsequent analysis, we introduce the spaces
\begin{align*}
	\ovHD 
	: = \left \{
	\vv \in H^1(\Omega, \RRR^d) : \vv = \0 \quad \text{a.e. on $\ov \Gamma_D$} 
	\right\},
\end{align*}
and 
\begin{align*}
	\hatHD : = \left \{
	\vv \in H^1(\Omega, \RRR^d) : \vv = \0 \quad \text{a.e. on $\hat \Gamma_D$} 
	\right\}.
\end{align*}
For brevity, the corresponding norms are denoted by the same symbol $\norma{\cdot}_{H^1(\Omega)}$ if no confusion may arise.
Vectors, matrices, and vector- or matrix-valued functions will be denoted by bold symbols.  Furthermore, for a subset $\Gamma_N \subset \Gamma$, we consider the function space 
\[
H^{1/2}_{00}(\Gamma_N, \RRR^d ) := \left\{ \vv \in H^{1/2}(\Gamma_N, \RRR^d): \widetilde\vv \in  H^{1/2} (\Gamma, \RRR^d)\right\},
\]
where $\widetilde \vv$ denotes the trivial extension of $\vv$ to $\Gamma$, and we endow it with the norm
\begin{align*}
	\norma{\vv}_{H^{1/2}_{00}(\Gamma_N)} := \norma{\widetilde \vv}_{H^{1/2}(\Gamma)},
	\quad 
	\vv \in H^{1/2}_{00}(\Gamma_N, \RRR^d).
\end{align*}
We highlight that the above definition is not redundant as in general the trivial extension of a $H^{1/2}(\Gamma_N, \RRR^d)$ function does not belong to $H^{1/2}(\Gamma, \RRR^d)$. Besides, we remark that $H^{1/2}_{00}(\Gamma_N, \RRR^d)$ is a Hilbert space and 
\begin{align*}
	(H^{1/2}_{00}(\Gamma_N, \RRR^d), L^{2}(\Gamma_N, \RRR^d), H^{1/2}_{00}(\Gamma_N, \RRR^d)^*)
\end{align*}
forms a Hilbert triple (see, e.g., \cite{ML}).
\medskip

For the forthcoming analysis we make the following structural assumptions.
\begin{enumerate}[label={(\bf A\arabic{*})}, ref={\bf(A\arabic{*})}]
\item \label{ass:dom} 
The domain $\Omega \subset \RRR^d$, $d \in \{2,3\}$, is a bounded domain with $C^{1,1}$ or convex boundary $\Gamma = \pd \Omega$ that admits a partition 
$\Gamma = {\rm cl}(\ov\Gamma_D) \cup {\rm cl}(\ov\Gamma_N)$ with relative open subsets $\ov\Gamma_D$ and $\ov\Gamma_N$ such that $\ov \Gamma_D \cap \ov\Gamma_N = \emptyset$ and $\ov\Gamma_D \neq \emptyset$. Here, ${\rm cl}(A)$ denotes the closure of the set $A$. The same assumptions are made for $\hat \Gamma_D$ and $\hat \Gamma_N$.

\item \label{ass:elasticity}
	The elasticity tensor ${\ov \CCC}$ is assumed to be a tensor-valued function
	\begin{align*}
		{\ov \CCC} : \RRR^L \to \RRR^{d\times d \times d \times d },
	\end{align*}
	with ${\ov \CCC}_{ijkl} \in {C^{1,1}(\RRR^L,\RRR)}$, $i,j,k,l \in \{1, \dots, d\}$.
	Moreover, it fulfils the symmetry conditions 
	\begin{align*}
		{\ov \CCC}_{ijkl} = {\ov \CCC}_{klij} = {\ov \CCC}_{ijlk}= {\ov \CCC}_{jikl} \quad \text{for all $i,j,k,l \in \{1, \dots, d\}$,}
	\end{align*}
and there exist positive constants $C_0, C_1$, and $C_2$ such that, for all $\bph,\hh \in \RRR^L$, 
\begin{alignat}{2}
	\label{ass:coerc}
	&  C_0 |{\AA}|^2 
	 \leq 
	 {\ov \CCC}(\bph) {\AA} : {\AA} 
	 \leq  C_1 | {\AA}|^2 
	 \qquad 
	&& \forall {\AA} \in \RRR^{d \times d}_{\rm sym}\setminus \{\0\},
	 \\
	 & |{\ov \CCC}'(\bph)\hh \, \AA:\BB| \leq C_2 |\hh||\AA||\BB|
	 \qquad 
	&& \forall {\AA,\BB} \in \RRR^{d \times d}_{\rm sym}\setminus \{\0\},
\end{alignat}
where $\AA: \BB = \sum_{i,j=1}^n A_{ij}B_{ij}$, $|\AA| = \sqrt{\AA : \AA}$, and for every $\hh=(h_1,\dots,h_L) \in \RRR^L$,
\begin{align*}
	[{\ov \CCC}'(\bph) \hh]_{ijkl} 
	:= 
	\sum_{m=1}^L \partial_m {\ov \CCC}_{ijkl}(\bph) h_m 
	\quad 
	\text{ for all }
	i,j,k,l\in\{1,\dots,d\}.
\end{align*}
The set $\RRR^{d\times d}_{\rm sym}$ consists of the symmetric $({d\times d})$-matrices. The same assumptions are made for $\hat \CCC$.

\item \label{ass:potential}
	The multiwell potential $\Psi$ possesses the form
	\begin{align*}
		\Psi : \RRR^L \to \RRR \cup\{+\infty\}, 
		\quad
		\Psi = \tilde{\Psi} + I_\Delta,
	\end{align*}
	where $\tilde{\Psi} \in {C^{1,1}(\RRR^L)}$ and the indicator function $I_\Delta$ of the simplex $\Delta^L$ is defined as
	\begin{align*}
	I_\Delta (\bph) =
		\begin{cases}
		 0 \quad & \text{if $\bph \in \Delta^L$},
		\\
		+\infty \quad & \text{otherwise}.
		\end{cases}
	\end{align*}
	
\item \label{ass:chi}
	The function $\chi :\RRR^L \to \RRR$ is ${C^{1,1}(\RRR^L)}$ and there exist a positive constant $\chi_0$ such that
	\begin{align*}
	0 \leq \chi(\bph)  \leq \chi_0
	\qquad 
	\text{ for all } \bph \in \RRR^L.
	\end{align*}
\item \label{ass:data} The data of the problems satisfy
	\begin{align*}
	 \ov{\FF}, \hat{\FF} &\in L^2(\Omega, \RRR^d),
	\quad 
	\ov \UU \in H^{1/2}(\ov \Gamma_D, \RRR^d), 
	\quad 
	\hat \UU\in H^{1/2}(\hat \Gamma_D, \RRR^d),
	\\ 
	\ov{\gg} &\in H^{1/2}_{00}(\ov \Gamma_N,\RRR^d)^*, 
	\quad 
	\hat{\gg} \in H^{1/2}_{00}(\hat \Gamma_N,\RRR^d)^*.
\end{align*}
\item \label{ass:utar}
	The target displacement $\uu^{\rm tar} \in L^2(\Gamma^{\rm tar},\RRR^d)$, where $\Gamma^{\rm tar} \subset \hat \Gamma_N$.
\end{enumerate}

It is worth pointing out that condition \ref{ass:potential} entails that $\Psi(\bph)  = \tilde{\Psi}(\bph) $ for every $\bph \in \Uad$. 

\section{Analysis of the design optimisation problem}\label{SEC:ANA}
\subsection{Linear elasticity system with mixed boundary conditions}
In this section we provide a preliminary well-posedness result for the following linear elasticity system with mixed boundary conditions
\begin{subequations}\label{sys:abs}
\begin{alignat}{2}
	\label{eq:abs:1} -\div \big(\CCC \E (\uu) + \FFF \big) & = \ff \qquad && \text{in $\Omega$},
	\\
	\label{eq:abs:2} \uu & = \0\qquad && \text{on $\Gamma_D$},	
	\\
	 \label{eq:abs:3} (\CCC \E (\uu) + \FFF) \nn & =  \gg \qquad && \text{on $\Gamma_N$}.
\end{alignat}
\end{subequations}
The well-posedness of the system \eqref{sys:abs} is formulated as follows.
\begin{prop}\label{THM:EX:ABS}
In addition to \ref{ass:dom}, suppose that $\CCC \in L^\infty(\Omega, \RRR^{d \times d \times d \times d})$ fulfils the symmetry conditions 
\begin{alignat*}{2}
	\CCC_{ijkl} & = \CCC_{klij} = \CCC_{ijlk}= \CCC_{jikl} \quad & \text{for all $i,j,k,l \in \{1, \dots, d\}$,}
\end{alignat*}
and satisfies
\begin{align}
	\label{ass:coerc:abs}
	 \lambda |{\AA}|^2 
	 \leq 
	 (\CCC {\AA}) : {\AA} 
	 \leq  \Lambda | {\AA}|^2 
	 \quad 
	 \text{ for all } \AA \in \RRR^{d \times d}_{\rm sym}\setminus \{\0\} 
	 \,\, \text{and a.e. in $\Omega$},
\end{align}
for some positive constants $\lambda$ and $\Lambda$. Then, for every $\FFF \in L^2(\Omega, \RRR^{d \times d})$, 
$\ff \in L^2(\Omega, \RRR^d)$ and $\gg \in H^{1/2}_{00}(\Gamma_N,\RRR^d)^*$, there exists a unique weak solution $\uu \in H^1_D(\Omega, \RRR^d) :=  \{ \bm{v} \in H^1(\Omega, \RRR^d) \, : \, \bm{v} = \bm{0} \text{ a.e.~in } \Gamma_D \}$ to the elasticity system \eqref{sys:abs} in the sense that
\begin{align}
	\label{wf:abs}
	\int_\Omega \CCC \E (\uu) : \E(\vv) \dx 
	=
	- \int_\Omega (\FFF : \E( \vv )
	- \ff \cdot \vv) \dx
	+ \<\gg,\vv>_{H^{1/2}_{00}(\Gamma_N)}
	\quad \forall \vv \in \HD,
\end{align}
and there exists a positive constant $C$ independent of $\uu$ such that
\begin{align}\label{reg:abs}
	\norma{\uu}_{H^1(\Omega)} 
	\leq C 
	\left(
		\norma{\FFF}_{L^2(\Omega)} 
		+ \norma{\ff}_{L^2(\Omega)} 
		+ \norma{\gg}_{H^{1/2}_{00}(\Gamma_N)^*}
	\right).
\end{align}
\end{prop}
Note that whenever $\gg \in L^{2}(\Gamma_N, \RRR^d) \subset H^{1/2}_{00}(\Gamma_N, \RRR^d)^*$, we can identify the duality product in \eqref{wf:abs} as the standard boundary integral, that is,
\begin{align*}
	\<\gg,\vv>_{H^{1/2}_{00}(\Gamma_N)} = \int_{\Gamma_N} \gg \cdot \vv \dH.
\end{align*}

\begin{proof}[Proof of Proposition~\ref{THM:EX:ABS}]
The variational equality \eqref{wf:abs} admits a unique solution by a direct application of the Lax--Milgram theorem (cf., e.g., \cite{Alt}). In this direction, we set
\begin{align*}
	& {\bf V} = H^1_D(\Omega, \RRR^d) \\	
	& a(\cdot,\cdot): {\bf V} \times {\bf V} \to \RRR, \quad a(\uu,\vv): = \int_\Omega \CCC \E(\uu) : \E(\vv) \dx ,
	\\
	&  \langle \FF, \vv \rangle :=  - \int_\Omega (\FFF: \E(\vv ) - \ff \cdot \vv )\dx +  \<\gg,\vv>_{H^{1/2}_{00}(\Gamma_N)}.
\end{align*}
It is worth noticing that it readily follows from the assumptions on $\FFF$, $\ff$, and $\gg$ that $\FF \in \VV^*$.
With the above notation, \eqref{wf:abs} can then be rewritten as the variational problem
\begin{align*}
	a (\uu,\vv) = \<\FF, \vv> \quad \forall \vv \in {\bf V}.
\end{align*}
Thus, to apply the Lax--Milgram theorem, it is sufficient to show the bilinear form $a(\cdot, \cdot)$ is continuous and coercive in $\VV$. By \eqref{ass:coerc:abs} we have
\begin{align*}
	|a(\uu,\vv)| 
	\leq C \norma{\uu}_{H^1(\Omega)}\norma{\vv}_{H^1(\Omega)} 
	\leq C \norma{\uu}_{{\bf V}}\norma{\vv}_{{\bf V}}
	\quad \forall  \uu,\vv \in {{\bf V}},
\end{align*}
while \eqref{ass:coerc:abs} and Korn's inequality yield the ${\bf V}$-coercivity:
\begin{align*}
	|a(\vv,\vv)| 
	\geq 
	\lambda \norma{\E(\vv)}_{L^2(\Omega)}^2  
	\geq
	C(\lambda,{C_K}) \norma{\vv}_{{\bf V}}^2
	\quad \forall  \vv \in {{\bf V}},
\end{align*}
with a constant ${C_K}$ arising from Korn's inequality. Thus, the existence and uniqueness of $\uu \in H^1_D(\Omega, \RRR^d)$ solving \eqref{wf:abs}, as well as the estimate \eqref{reg:abs}, readily follow from the Lax--Milgram theorem.
\end{proof}

We end this section with another abstract result that will be useful for the subsequent analysis. Consider the following problem with inhomogeneous data on the Dirichlet boundary:
\begin{equation}\label{Diri}
\begin{alignedat}{2}
	-\div \big(\CCC \E (\uu) \big) & = \0\qquad && \text{ in } \Omega,
	\\
	\uu & = {\cal U}\qquad && \text{ on } \Gamma_D,	
	\\
	 \big (\CCC \E (\uu) \big) \nn & =  \0 \qquad && \text{ on } \Gamma_N.
\end{alignedat}
\end{equation}
Well-known theory yields that, for every ${\cal U} \in H^{1/2} (\Gamma_D, \RRR^d)$, there exists a unique weak solution $\uu \in H^1(\Omega,\RRR^d)$. The proof follows similarly to the above as a direct consequence of the Lax--Milgram theorem.
This allows us to introduce the associated solution operator, that we call the {\it extension operator}
\begin{align}
		{\cal H} : H^{1/2}(\Gamma_D, \RRR^d) \to H^1(\Omega,\RRR^d), 
		\quad 
		{\cal H}: {\cal U} \mapsto \uu,
		\quad 
		\label{harm:ext}
\end{align}
where $\uu$ is the unique weak solution to system \eqref{Diri}.

\subsection{Well-posedness of the state systems}
Similarly to \eqref{Diri} and \eqref{harm:ext}, we can introduce the extension operators $\ov {\cal H}$ and  $\hat {\cal H}$ related to 
$\ov \CCC$, $\ov \Gamma_D$ and $\hat \CCC$, $\hat \Gamma_D$, respectively. 
Then defining the functions $\ov H :=  \ov{\cal H}(\bU), \hat H := \hat {\cal H}(\hU) \in H^1(\Omega, \RRR^d)$ allows us to transform \eqref{bu:sys} and \eqref{hu:sys} into the equivalent problems
\begin{equation}\label{bu:new:sys}
\begin{alignedat}{2}
	 -\div \big(\ov \CCC (\bph) \E (\bu^{\rm new} + {\ov H} ) \big) & = \ov{\FF} \qquad && \text{ in } \Omega,
	\\
	 \bu^{\rm new} & = \0 \qquad && \text{ on } \ov \Gamma_D,	
	\\
	 \big( \ov \CCC (\bph) \E (\bu^{\rm new} + {\ov H})\big) \nn & = \ov{\gg} \qquad && \text{ on } \ov \Gamma_N, 
	 \end{alignedat}
\end{equation}
and
\begin{equation}\label{hu:new:sys}
\begin{alignedat}{2}
	 -\div \big(\hat \CCC (\bph) ( \E (\hu^{\rm new} + {\hat H}) - \chi(\bph) \E(\bu^{\rm new} + {\ov H})\big) & = \hat{\FF} \qquad && \text{ in } \Omega,
	\\
	 \hu^{\rm new} & = \0 \qquad && \text{ on } \hat \Gamma_D,
	\\
	 \big(\hat \CCC (\bph) ( \E (\hu^{\rm new} + {\hat H})- \chi(\bph) \E(\bu^{\rm new} + {\ov H}) \big) \nn & = \hat{\gg} \qquad && \text{ on } \hat \Gamma_N,
\end{alignedat}
\end{equation}
where we set
\[
\bu = \bu^{\rm new} + {\ov H}, \quad \hu = \hu^{\rm new} + {\hat H}.
\]
For a cleaner presentation, we abuse notation and use the same variables $\bu$ and $\hu$ to denote $\bu^{\rm new}$ and $\hu^{\rm new}$.  Then, the well-posedness of \eqref{bu:sys} and \eqref{hu:sys} (equivalently \eqref{bu:new:sys} and \eqref{hu:new:sys}) are formulated as follows.

\begin{thm}\label{THM:EX:STATE}
Under \ref{ass:dom}--\ref{ass:data}, for every $\bph \in L^\infty(\Omega, \RRR^L)$, there exists a unique solution pair $(\bu, \hu) \in \ovHD \times \hatHD$ satisfying
\begin{align}
	& 
	\iO {\ov \CCC(\bph) \E(\bu) : \E( \ov \bz)}
	= 	
	\iO {\ov{\FF} \cdot \ov \bz}
	- \iO {\ov \CCC(\bph) \E({\ov H}) : \E( \ov \bz)}
	\label{wf:sys:1}
	+ \< \ov{\gg} ,\ov \bz >_{H^{1/2}_{00}(\ov \Gamma_N)},
	\\
	& 
	\notag
	\iO {\hat \CCC(\bph) \E(\hu) : \E( \hat \bz)}
	- \iO {\hat \CCC(\bph) \chi(\bph) \E(\bu) :\E( \hat \bz)}
	= 	
	\iO {\hat{\FF} \cdot \hat \bz}
	\\
	\label{wf:sys:2}
	& \quad 
	+ \iO {\hat \CCC(\bph) \E({\hat H}) : \E( \hat \bz)}
	- \iO {\hat \CCC(\bph) \chi(\bph) \E({\ov H}) :\E(\hat  \bz)}
	+ \< \hat{\gg} ,\hat \bz >_{H^{1/2}_{00}(\hat \Gamma_N)}
\end{align}
for all $\ov \bz \in \ovHD$ and $\hat \bz \in \hatHD$.  Moreover, there exists a positive constant $C$, independent of $\bph$, such that
\begin{align}\label{reg:state}
	\norma{\bu}_{H^1(\Omega)} 
	+\norma{\hu}_{H^1(\Omega)} 
	\leq C.
\end{align}
\end{thm}
\begin{proof}
For \eqref{wf:sys:1} we invoke Proposition \ref{THM:EX:ABS} with the specifications
\begin{align*}
	& \uu = \bu,
	\quad
	\CCC= \ov \CCC(\bph),	
	\quad 
	\FFF= \ov \CCC(\bph) \E({\ov H}), 
	\quad 
	\ff= \ov{\FF}, 
	\quad 
	\gg = \ov{\gg},
	\quad
	\Gamma_D = \ov \Gamma_D,
	\quad
	\Gamma_N = \ov \Gamma_N,
\end{align*}
and for \eqref{wf:sys:2} we consider
\begin{align*}
	& \uu = \hu,
	\quad
	\CCC= \hat \CCC(\bph),	
	\quad 
	\FFF=  \hat \CCC(\bph)( \E ({\hat H})- \chi(\bph) \E(\bu + {\ov H} ) ), 
	\quad 
	\ff= \hat{\FF}, 
	\quad 
	\gg = \hat{\gg},
	\\ & 	
	\Gamma_D = \hat \Gamma_D,
	\quad
	\Gamma_N = \hat \Gamma_N,
\end{align*}
to obtain the existence and uniqueness of solutions $\bu$ and $\hu$. Lastly the estimate \eqref{reg:state} can be obtained from \eqref{reg:abs} and the uniform boundedness of the tensors $\ov{\CCC}$ and $\hat{\CCC}$ in \ref{ass:elasticity}.
\end{proof}

\begin{thm}\label{THM:CD}
Under \ref{ass:dom}--\ref{ass:data}, for $i = 1,2$, let $\bph_i \in L^\infty(\Omega, \RRR^L) $ with $\norma{\bph_i}_{L^\infty(\Omega)} \leq R$, where $R> 0$ is fixed, and let $(\bu_i, \hu_i)$ denote the unique solutions 
to systems \eqref{bu:new:sys} and \eqref{hu:new:sys} corresponding to $\bph_i$, but with the same data ${\ov \FF}$, $\hat{\FF}$, $\ov{\gg}$, $\hat{\gg}$, ${\ov H}$ and ${\hat H}$.
Then, there exists a positive constant $C$, independent of the differences, such that
\begin{equation}
	\label{cd:est}
\begin{aligned}
	 \norma{\bu_1-\bu_2}_{H^1(\Omega)} &  \leq C \norma{\bph_1- \bph_2}_{L^\infty(\Omega)}, \\
	 \norma{\hu_1-\hu_2}_{H^1(\Omega)} & \leq C \norma{\E(\bu_1 - \bu_2)}_{L^2(\Omega)} + C \norma{\bph_1- \bph_2}_{L^\infty(\Omega)}.
\end{aligned}
\end{equation}
\end{thm}
\begin{proof}
To start, let us set
\begin{align*}
	& \bph: = \bph_1 - \bph_2,
	\quad 
	\bu:= \bu_1-\bu_2,
	\quad
	\hu:= \hu_1-\hu_2.
\end{align*}
Then, we consider the difference between the variational equalities \eqref{wf:sys:1}--\eqref{wf:sys:2} written for $(\bph_1,\bu_1,\hu_1)$ and for $(\bph_2,\bu_2,\hu_2)$ to infer that
\begin{align}
	& 
	\iO {\big(\ov \CCC(\bph_1) \E(\bu_1)  - \ov \CCC(\bph_2) \E(\bu_2)\big) : \E( \ov \bz)}
	= 0,
	\label{wf:CD:1}
	\\
	& 
	\notag
	\iO {\big(\hat \CCC(\bph_1) \E(\hu_1)  - \hat \CCC(\bph_2) \E(\hu_2) \big): \E( \hat \bz)}
	\\ 
	\label{wf:CD:2}
	& \quad 
	- \iO {\big(\hat \CCC(\bph_1) \chi(\bph_1) \E(\bu_1) - \hat \CCC(\bph_2)\chi(\bph_2) \E(\bu_2)\big) :\E(\hat \bz)} = 	
	0,
\end{align}
for all $\ov \bz \in \ovHD$ and $\hat \bz \in \hatHD.$
Choosing $\ov \bz = \bu$ and invoking condition \eqref{ass:coerc} yields
\begin{equation}\label{cd:bu}
\begin{aligned}
C_0 \norma{\E(\bu)}_{L^2(\Omega)}^2 & \leq - \iO{(\ov \CCC(\bph_1)-\ov\CCC(\bph_2)) \E(\bu_2) : \E( \bu)} \\
& \leq \frac{C_0}{2} \norma{\E(\bu)}_{L^2(\Omega)}^2 + C \norma{\bph}_{L^\infty(\Omega)}^2.
\end{aligned}
\end{equation}
By Korn's inequality we infer
\begin{align}\label{cd:est:1}
\norma{\bu_1 - \bu_2}_{H^1(\Omega)} \leq C \norma{\bph_1- \bph_2}_{L^\infty(\Omega)}.
\end{align}
Then, inserting $\hat \bz = \hu$ in \eqref{wf:CD:2}, and invoking the Lipschitz continuity of $\hat{\CCC}$ and $\chi$ from \ref{ass:elasticity} and \ref{ass:chi}, as well as \eqref{cd:est:1}, yields
\begin{equation}\label{cd:hu}
\begin{aligned}
& C_0 \norma{\E(\hu)}_{L^2(\Omega)}^2 \\
& \quad \leq - \iO {(\hat \CCC(\bph_1)-\hat\CCC(\bph_2)) \E(\hu_2) : \E( \hu) + (\hat \CCC(\bph_1)-\hat\CCC(\bph_2)) \chi(\bph_1) \E(\bu_1)  :\E( \hu)}\\
& \qquad -\iO {\hat \CCC(\bph_2)(\chi(\bph_1)  - \chi(\bph_2) ) \E(\bu_1) :\E( \hu)} - \iO {\hat \CCC(\bph_2) \chi(\bph_2) \E(\bu) :\E( \hu)} \\
& \quad \leq \frac{C_0}{2} \norma{\E(\hu)}_{L^2(\Omega)}^2 + C \Big ( \norma{\E(\bu)}_{L^2(\Omega)}^2 + \norma{\bph}_{L^\infty(\Omega)}^2 \Big ).
\end{aligned}
\end{equation}
Applying Korn's inequality leads to \eqref{cd:est}. 
\end{proof}

\begin{cor}\label{COR:cts}
Suppose that \ref{ass:dom}--\ref{ass:data} hold.
Let $R > 0$ and let $\{\bph_n\}_{n \in \NNN}$ be a sequence of functions in $L^\infty(\Omega, \RRR^L)$ such that $\| \bph_n \|_{L^\infty(\Omega)} \leq R$ for all $n \in \NNN$ and $\bph_n \to \bph_*$ strongly in $L^1(\Omega, \RRR^L)$ as $n \to \infty$.  Let $(\bu_n, \hu_n)$ denote the solutions to \eqref{wf:sys:1} and \eqref{wf:sys:2} corresponding to data $\bph_n$,  $\ov{\FF}$, $\hat{\FF}$, $\ov{\gg}$, $\hat{\gg}$, ${\ov H}$ and ${\hat H}$.  Then, it holds that, as $n \to \infty$,
\[
\bu_n \to \bu_* \text{ strongly in } \ovHD, \quad \hu_n \to \hu_* \text{ strongly in } \hatHD,
\]
where $(\bu_*, \hu_*) \in \ovHD \times \hatHD$ are the unique solutions to \eqref{wf:sys:1} and \eqref{wf:sys:2} corresponding to data $\bph_*$, $\ov{\FF}$, $\hat{\FF}$, $\ov{\gg}$, $\hat{\gg}$, ${\ov H}$ and ${\hat H}$.
\end{cor}

\begin{proof}
From \eqref{reg:state} we infer that $\bu_n$ and $\hu_n$ are bounded in $\ovHD$ and $\hatHD$, respectively, and thus there exist limit functions $(\bu_*, \hu_*) \in \ovHD \times \hatHD$ such that, along a non-relabelled subsequence, $\bu_n \rightharpoonup \bu_*$ in $\ovHD$ and $\hu_n \rightharpoonup \hu_*$ in $\hatHD$.  To obtain strong convergence, in \eqref{wf:CD:1} and \eqref{wf:CD:2} we substitute $\bph_1 = \bph_n$, $\bph_2 = \bph_*$, $\bu_1 = \bu_n$, $\bu_2 = \bu_*$, $\bu = \bu_n - \bu_*$ and $\hu = \hu_n - \hu_*$.  Then, by virtue of the dominated convergence theorem, we infer the strong convergence, as $n \to \infty$,
\begin{align*}
 ({\ov \CCC}(\bph_n) - {\ov \CCC}(\bph_*))\E(\bu_*) \to \0 \quad \text{ in } L^2(\Omega, \RRR^{d \times d}).
\end{align*}
Hence, in the analogue of the first inequality in \eqref{cd:bu} we see that the integral on the right-hand side converges to zero, which implies via Korn's inequality that 
\[
\norma{\bu_n - \bu_* }_{H^1(\Omega)} \to 0.
\]
By the generalised dominated convergence theorem, we have the strong convergences
\begin{align*}
\begin{cases}
 ({\hat \CCC}(\bph_n) - {\hat \CCC}(\bph_*)) \chi(\bph_n) \E(\bu_n) & \to \0 \\
 {\hat \CCC}(\bph_*) (\chi(\bph_n) - \chi(\bph_*)) \E(\hu_n) & \to \0 \\
 \hat{\CCC}(\bph_*) \chi(\bph_*) \E(\bu_n - \bu_*) & \to \0
 \end{cases}  \text{ in } L^2(\Omega, \RRR^{d \times d}),
\end{align*}
and thus, in the analogue of the first inequality in \eqref{cd:hu} we see that the integral on the right-hand side converges to zero, leading to the assertion
\[
\norma{\hu_n - \hu_*}_{H^1(\Omega)} \to 0.
\]
Thus, by combining the weak convergences with the above norms convergence the claim follows.
\end{proof}
The above analysis for systems \eqref{bu:new:sys} and \eqref{hu:new:sys} allows us to define some solution operators. Namely, we introduce
\begin{align}\label{controltostateoperator}
	\S : L^\infty(\Omega,\RRR^L) \to \hatHD,
	\quad 
	\S: \bph \mapsto \hu = \hu(\bph),
\end{align}
as well as the intermediate operators:
\begin{alignat*}{2}
	& \S_1: L^\infty(\Omega,\RRR^L)  \to L^\infty(\Omega,\RRR^L)  \times \ovHD,
	\quad 
	&& \S_1: \bph \mapsto (\S_1^1(\bph),\S_1^2(\bph)) = (\bph,{\bu}),
	\\
	& \S_2:  L^\infty(\Omega,\RRR^L)  \times \ovHD \to \hatHD,
	\quad 
	&& \S_2: (\bph,\bu) \mapsto \hu,
\end{alignat*}
where $\bu = \bu (\bph)$ and $\hu= \hu(\bph,\bu)$ are the unique solutions obtained from Theorem \ref{THM:EX:STATE}. Then, the overall solution operator $\S$ in \eqref{controltostateoperator} is simply the composition of the intermediate mappings $\S_1$ and $\S_2$, i.e., $\S=\S_2 \circ \S_1$. In particular, we can define the {\it reduced cost functional}
\[
\Jred: \Uad \to \RRR, \quad 
\Jred: \bph \mapsto  J(\bph, \S(\bph)).
\]

\subsection{Existence of optimal designs}
\begin{thm}\label{THM:EXOPT}
Under \ref{ass:dom}--\ref{ass:utar}, the optimisation problem {\bf (P)} admits at least one solution.
\end{thm}
\begin{proof}As the proof is nowadays standard with the direct method of the calculus of variations, let us briefly outline the main points. Consider a minimising sequence $\{\bph_n\}_{n \in \NNN} \subset \Uad$ for the reduced cost functional $\Jred$, which satisfies
\begin{align*}
	\lim_{n\to \infty} \Jred(\bph_n)  = \inf_{\bph \in \Uad} \Jred(\bph) \geq 0.
\end{align*}
This yields that $\{\bph_n\}_{n \in \NNN}$ is bounded in $H^1(\Omega,\RRR^L) \cap L^\infty(\Omega,\RRR^L)$. By standard compactness arguments, since $\Uad$ is closed and convex, we obtain a limit function $\bph^* \in \Uad$ such that $\bph_n \to \bph^*$ weakly* in $H^1(\Omega, \RRR^L) \cap L^\infty(\Omega, \RRR^L)$ along a non-relabelled subsequence. Consequently, by \eqref{reg:state} the sequence $\{\hu_n = \S(\bph_n)\}_{n \in \NNN}$ is bounded in $\hatHD$, and on invoking Corollary \ref{COR:cts} there exists a limit function $\hu^* \in \hatHD$ such that, along a non-relabelled subsequence, $\hu_n \to \hu^*$ strongly in $\hatHD$ as $n \to \infty$. Continuity of the boundary trace operator gives $\hu_n \to \hu^*$ strongly in $L^2(\hat{\Gamma}_N, \RRR^d)$, and thus
\[
\int_{\Gamma^{\rm tar}} W(\hu_n - \bm{u}^{\rm tar}) \cdot (\hu_n - \bm{u}^{\rm tar}) \dH \to \int_{\Gamma^{\rm tar}}W(\hu^* - \bm{u}^{\rm tar}) \cdot (\hu^* - \bm{u}^{\rm tar}) \dH.
\]
By Fatou's lemma and the a.e.~convergence of $\bph_n$ to $\bph^*$, we have 
\begin{align*}
	\liminf_{n \to \infty} \| \Psi(\bph_n) \|_{L^1(\Omega)} \geq \| \Psi(\bph^*) \|_{L^1(\Omega)},
\end{align*}
and using also the weak lower semicontinuity of the $L^2$-norm, we infer that
\[
\Jred(\bph^*) \leq \lim_{n \to \infty} \Jred(\bph_n) = \inf_{\bph \in \Uad} \Jred(\bph).
\]
This shows that $\bph^*$ is a solution to $(\bf P)$.
\end{proof}

\section{Optimality conditions}
\label{SEC:OPT}
To derive the first order necessary optimality conditions for $\bph^*$, we first study the linearised system for the {\it linearised variables} introduced below, and use {\it adjoint variables} to provide a simplification of the optimality condition.  

\subsection{Linearised systems and Fr\'echet differentiability}
Here, we analyse some differentiability properties of the solutions operators $\S_1$ and $\S_2$ introduced above. This will help us to formulate the first order optimality conditions of {\bf (P)}.

\begin{thm}\label{THM:LIN:S1}
The solution operator $\S_1$ is Fr\'echet differentiable at $\bph$ as a mapping 
from $L^\infty(\Omega, \RRR^L)$ into $L^\infty(\Omega, \RRR^L) \times \ovHD$.
Moreover, it holds that
\begin{align*}
	D \S_1 \in {\cal L}(L^\infty(\Omega, \RRR^L) , L^\infty(\Omega, \RRR^L) \times \ovHD),
\end{align*}
and its directional derivative at $\bph \in L^\infty(\Omega, \RRR^L)$ along a direction $\hh \in  L^\infty(\Omega, \RRR^L)$ is given by
\begin{align}\label{DS1}
	D \S_1 (\bph) (\hh) = (\hh,\bv),
\end{align}
where $\bv \in \ovHD$ is the unique weak solution to the following system:
\begin{subequations}\label{sys:LIN}
\begin{alignat}{2}
	\label{SYS:LIN:1}
	 -\div \big(\ov \CCC'(\bph)\hh \E (\bu) 
	+ \ov \CCC(\bph)\E (\bv)\big) 
	& = \0 
	\qquad && \text{ in } \Omega,
	\\
	\label{SYS:LIN:3}
	 \bv 
	& = \0 
	\qquad && \text{ on } \ov \Gamma_D,	
	\\
	\label{SYS:LIN:5}
	 (\ov \CCC'(\bph)\hh \E (\bu) 
	+ \ov \CCC(\bph)\E (\bv)) \nn 
	& = \0 
	\qquad && \text{ on } \ov \Gamma_N,
\end{alignat}
\end{subequations}
in the sense
\begin{align}\label{wf:sys:LIN}
\iO{ \ov{\CCC}(\bph) \E(\bv) : \E(\bz)} + \iO{ \ov{\CCC}'(\bph) \hh \E(\bu) : \E(\bz) } = 0
\end{align}
for all $\bz \in \ovHD$, where $\bu$ is the unique solution to \eqref{wf:sys:1} associated to $\bph$ obtained from Theorem~\ref{THM:EX:STATE}.
\end{thm}

\begin{proof}
Firstly, the unique solvability of the linearised system \eqref{sys:LIN} follows directly from the application of Proposition \ref{THM:EX:ABS} upon choosing
\begin{align*}
	& \uu = \bv,
	\quad
	\CCC= \ov \CCC(\bph),	
	\quad 
	\FFF= \ov \CCC'(\bph) \hh \E(\bu), 
	\quad 
	\ff=\gg = \0,
	\quad \Gamma_D = \ov \Gamma_D,
	\quad
	\Gamma_N = \ov \Gamma_N.
\end{align*}
Next, we take $\bph \in \Uad$ and $\hh \in L^\infty(\Omega, \RRR^L)$ such that $\bph^{\hh} := \bph + \hh \in \Uad$, and set $\bu^\hh= \S_1^2(\bph^\hh)$, i.e., $(\bph^\hh, \bu^\hh)=\S_1(\bph^\hh)$.  Since the first component of $\S_1$ is just the identity in $L^\infty(\Omega, \RRR^L)$, we only need to investigate the Fr\'echet differentiability of the second component $\S_1^2$.  In this direction, we denote
\begin{align*}
	\ww : = \bu^\hh - \bu - \bv \in \ovHD
\end{align*}
with $\bv$ being the unique solution to the linearised system \eqref{sys:LIN} associated to $\bph$ and $\hh$. Our aim is to show the existence of a positive constant $C$ such that
\begin{align}\label{Fre:bu}
\| \ww \|_{H^1(\Omega)} = \| \S_1^2(\bph^\hh) - \S_1^2(\bph) - D\S_1^2(\bph) \hh \|_{H^1(\Omega)} \leq C \| \hh \|_{L^\infty(\Omega)}^2,
\end{align}
which would then imply the Fr\'echet differentiability of the operator $\S_1^2$. To this end, we subtract from \eqref{wf:sys:1} for $\bph^\hh$ the sum of \eqref{wf:sys:1} for $\bph$ and \eqref{wf:sys:LIN} for $\bv$ to obtain
\begin{align*}
	& \iO {  \ov \CCC(\bph ) \E(\ww): \E(\bz) }
	+ \iO { (\ov \CCC(\bph +\hh)- \ov \CCC(\bph) ) (\E(\bu^\hh) - \E(\bu) )  : \E(\bz)}
	\\ & \quad 
	+ \iO {[\ov \CCC(\bph+\hh) - \ov \CCC(\bph ) - \ov \CCC'(\bph) \hh] \E(\bu) : \E(\bz)} = 0
	\quad 
	\forall \bz \in \ovHD.
\end{align*}
Choosing $\bz = \ww$ and using \eqref{ass:coerc} we infer that
\begin{align*}
	 C_0 \norma{\E(\ww)}^2_{L^2(\Omega)}
	&\leq 
	- \iO { (\ov \CCC(\bph +\hh)- \ov \CCC(\bph) ) (\E(\bu^\hh) - \E(\bu) )  : \E(\ww)}
	\\ & \quad 
	- \iO {[\ov \CCC(\bph+\hh) - \ov \CCC(\bph ) - \ov \CCC'(\bph) \hh] \E(\bu) : \E(\ww)} .
\end{align*}
Lipschitz continuity of $\ov \CCC$ yields a positive constant $C_{\ov \CCC'} $ such that
\begin{align}\label{TaylorC}
	|\ov \CCC(\bph+\hh) - \ov \CCC(\bph ) - \ov \CCC'(\bph) \hh | 
	\leq 
	|\hh| \int_0^1 |\ov \CCC'(\bph + s \hh) - \ov \CCC'(\bph)| \ds
	\leq C_{\ov \CCC'} |\hh|^2,
\end{align}
and keeping in mind the estimate obtained from Theorem \ref{THM:CD}:
\begin{align}\label{cd:frechet:1}
	\norma{\bu^\hh - \bu}_{H^1(\Omega)} \leq C \norma{\hh}_{L^\infty(\Omega)},
\end{align}
we find that \begin{align*}
	& \iO { (\ov \CCC(\bph +\hh)- \ov \CCC(\bph) ) (\E(\bu^\hh) - \E(\bu) )  : \E(\ww)}
	\\ & \quad 
	\leq 
	C \norma{\hh}_{L^\infty(\Omega)} \norma{\E(\bu^\hh) - \E(\bu)}_{L^2(\Omega)} \norma{\E(\ww)}_{L^2(\Omega)}
	\leq 
	\frac {C_0}4 \norma{\E(\ww)}_{L^2(\Omega)}^2 
	+ C \norma{\hh}_{L^\infty(\Omega)}^4,
	\\
	 & \iO {[\ov \CCC(\bph+\hh) - \ov \CCC(\bph ) - \ov \CCC'(\bph) \hh] \E(\bu) : \E(\ww)} 
	 \\ & \quad
	\leq 
	C\norma{\hh}_{L^\infty(\Omega)}^2\norma{\E(\bu)}_{L^2(\Omega)}\norma{\E(\ww)}_{L^2(\Omega)} \leq 
	\frac {C_0}4\norma{\E(\ww)}_{L^2(\Omega)}^2 
	+ C \norma{\hh}_{L^\infty(\Omega)}^4.
\end{align*}
Then, by Korn's inequality we infer \eqref{Fre:bu}, and whence the claimed Frech\'et differentiability of $\S_1$.
\end{proof}

Before presenting the Fr\'echet differentiability of $\S_2$, let us provide a formal discussion. Recall that $\hu = \S_2(\bph, \bu)$ and thus the directional derivative $D \S_2(\bph, \bu)(\hh, \kk)$ of $\S_2$ at $(\bph, \bu)$ along a direction $(\hh, \kk)$ will be given by 
\[
D \S_2(\bph, \bu)(\hh, \kk) = D_{\bph} \S_2(\bph, \bu) (\hh) + D_{\bu} \S_2(\bph, \bu) (\kk)
\]
where $D_{\bph}$ and $D_{\bu}$ represent the partial derivatives with respect to $\bph$ and $\bu$, respectively.  Hence, we expect $\hat{\bth} := D \S_2(\bph, \bu)(\hh, \kk)$ to be a sum of two functions $\hv := D_{\bph} \S_2(\bph, \bu) (\hh)$ and $\hw := D_{\bu} \S_2(\bph, \bu) (\kk)$, and the Fr\'echet differentiability of $\S_2$ can be established by demonstrating
\begin{align}\label{Fre:2}
\| \S_2(\bph+\hh, \bu+\kk) - \S_2(\bph, \bu) - \hat{\bth} \|_{H^1(\Omega)} \leq C \| (\hh, \kk) \|_{L^\infty(\Omega) \times H^1(\Omega)}^2.
\end{align}
The result is formulated as follows.
\begin{thm}\label{THM:LIN:S2}
The solution operator $\S_2$ is Fr\'echet differentiable at $(\bph, \bu)$ as a mapping from $L^\infty(\Omega, \RRR^L) \times \ovHD$ into $\hatHD$.
Furthermore, 
\begin{align*}
	D\S_2 \in {\cal L}(L^\infty(\Omega, \RRR^L) \times \ovHD,  \hatHD),
\end{align*}
and its directional derivative at $(\bph, \bu) \in L^\infty(\Omega, \RRR^L) \times \ovHD$ along a direction $(\hh, \kk) \in  L^\infty(\Omega, \RRR^L)\times \ovHD$ is given by
\begin{align*}
\hat{\bth} := D \S_2 (\bph, \bu) (\hh, \kk) =D_\bph \S_2 (\bph, \bu) (\hh)
	+D_\bu \S_2 (\bph, \bu) (\kk) =: \hv + \hw,
\end{align*}
where $\hat \bth \in \hatHD$ is the unique weak solution to the following system:
\begin{subequations}\label{sys:LIN:2}
\begin{alignat}{2}
\notag
 -\div \big(\hat \CCC'(\bph) \hh ( \E (\hu)- \chi(\bph) \E(\bu)) \big) & && 
\label{SYS:LIN:sum:1} \\
 \quad - \div\big(\hat \CCC(\bph) ( \E (\hat \bth)- \chi'(\bph) \hh \E(\bu)- \chi(\bph) \E( \kk))\big) 
& = \0 \qquad && \text{ in } \Omega, \\
\label{SYS:LIN:sum:2}
 \hat \bth & = \0 \qquad && \text{ on } \hat \Gamma_D,
\\ 
\notag  \big(\hat \CCC'(\bph) \hh ( \E (\hu)- \chi(\bph) \E(\bu)) \big) \nn & && \\
\label{SYS:LIN:sum:3}
 \quad +  \big ( \hat \CCC(\bph) ( \E (\hat \bth)- \chi'(\bph) \hh \E(\bu)- \chi(\bph) \E( \kk)) \big) \nn & = \0 \qquad && \text{ on } \hat \Gamma_N,
\end{alignat}
\end{subequations}
in the sense that
\begin{equation}\label{wf:sys:hthe}
\begin{aligned}
& \iO{ \hat{\CCC}(\bph)(\E(\hat{\bth}) - \chi'(\bph) \hh \E(\bu) - \chi(\bph) \E(\kk)) :\E(\bz)} \\
& \quad + \iO{ \hat{\CCC}'(\bph) \hh (\E(\hu) - \chi(\bph) \E(\bu)) : \E(\bz) } = 0
\end{aligned}
\end{equation}
for all $\bz \in \hatHD$, where $\bu$ is the unique solution to \eqref{wf:sys:1} associated to $\bph$ and $\hu$ is the unique solution to \eqref{wf:sys:2} associated to $(\bph, \bu)$.  Moreover, $\hv \in \hatHD$ and $\hw \in \hatHD$ are the unique solutions to the following equations
\begin{align}
& \iO{ \hat{\CCC}(\bph) (\E(\hv) - \chi'(\bph) \hh \E(\bu)) :\E(\bz) + \hat{\CCC}'(\bph) \hh (\E(\hu) - \chi(\bph) \E(\bu)) : \E(\bz)} = 0, \label{wf:sys:hv} \\
& \iO{\hat{\CCC}(\bph) (\E(\hw) - \chi(\bph) \E(\kk)) : \E(\bz) } = 0, \label{wf:sys:hw}
\end{align}
for all $\bz \in \hatHD$.
\end{thm}
\begin{proof}
Unique solvability of \eqref{wf:sys:hthe}, \eqref{wf:sys:hv} and \eqref{wf:sys:hw} are obtained by application of Theorem \ref{THM:EX:ABS}, and thus we focus on demonstrating the crucial estimate \eqref{Fre:2}.  Let $\bph \in \Uad$, $\hh \in L^\infty(\Omega, \RRR^L)$ such that $\bph^\hh = \bph + \hh \in \Uad$ and $\kk \in \ovHD$.  Setting
\[
\hu = \S_2(\bph, \bu), \quad \hu^{\kk} := \S_2(\bph, \bu+\kk), \quad \hu^{\hh,\kk} := \S_2(\bph+\hh, \bu+\kk),
\]
then, setting $\bxi := \hu^{\hh,\kk} - \hu - \hat{\bth}$, \eqref{Fre:2} is equivalent to
\[
\| \bxi \|_{H^1(\Omega)} \leq C \| (\hh, \kk) \|_{L^\infty(\Omega) \times H^1(\Omega)}^2.
\]
We observe that by subtracting from \eqref{wf:sys:2} for $(\bph+\hh, \bu + \kk, \hu^{\hh,\kk})$ the sum of \eqref{wf:sys:2} for $(\bph, \bu, \hu)$ and \eqref{wf:sys:hthe} for $\hat{\bth}$, we obtain
\begin{equation}\label{Fre:hat:main}
\begin{aligned}
	& \iO {\big( \hat \CCC (\bph + \hh) ( \E (\hu^{\hh,\kk})- \chi(\bph + \hh) \E(\bu + \kk) \big)   : \E( \bz)} 
	\\ & \quad 
	- 	\iO {\hat \CCC (\bph ) ( \E (\hu)- \chi(\bph) \E(\bu))   : \E( \bz)} 
	\\ & \quad 
	- \iO {\hat \CCC'(\bph) \hh ( \E (\hu)- \chi(\bph) \E(\bu)): \E( \bz)}  
	\\ & \quad 
	- \iO { \hat \CCC(\bph) ( \E (\hat \bth)- \chi'(\bph) \hh \E(\bu)- \chi(\bph) \E( \kk) ): \E( \bz)}
	=0
	\quad \forall   \bz \in \hatHD.
\end{aligned}
\end{equation}
Before proceeding with some computations, let us point out the following identities:
\begin{equation}\label{X1}
\begin{aligned}
	& 
	\hat \CCC (\bph + \hh)  \E (\hu^{\hh,\kk})
	-\hat \CCC (\bph)  \E (\hu)
	- \hat \CCC'(\bph) \hh  \E (\hu)
	-  \hat \CCC(\bph) \E (\hat \bth)
	\\ &  \quad 
	= \hat \CCC (\bph)  \E (\bxi)
	+ (\hat \CCC (\bph + \hh) -  \hat \CCC (\bph ))   ( \E(\hu^{\hh,\kk})- \E(\hu^\hh))
	\\ & \qquad 
	+ (\hat \CCC (\bph + \hh) -  \hat \CCC (\bph ))   ( \E(\hu^{\hh})- \E(\hu))
	\\ & \qquad 
	+ [ \hat \CCC (\bph + \hh) -  \hat \CCC (\bph )-  \hat \CCC' (\bph )\hh] \E(\hu)
	=: \hat \CCC (\bph)  \E (\bxi) + {\cal Y}_1 ,
\end{aligned}
\end{equation} 
and
\begin{equation}\label{X2}
\begin{aligned}
	&  - \hat \CCC (\bph + \hh) \chi(\bph + \hh) \E(\bu+ \kk)
	+ \hat \CCC (\bph ) \chi(\bph) \E(\bu)
	+ \hat \CCC' (\bph ) \hh \chi(\bph) \E(\bu)
	\\ & \qquad 
	+ \hat \CCC (\bph )  \chi'(\bph) \hh \E(\bu)
	+ \hat \CCC (\bph )  \chi(\bph)\E( \kk)
	\\ & \quad 
	= 
	-[ \hat \CCC (\bph + \hh) -  \hat \CCC (\bph )-  \hat \CCC' (\bph )\hh] \chi(\bph) \E(\bu)
	\\ & \qquad 
	- \hat \CCC(\bph) [\chi(\bph + \hh) - \chi(\bph) - \chi'(\bph)\hh] \E(\bu)
	\\ & \qquad 
	- (\hat \CCC (\bph + \hh) -  \hat \CCC (\bph )) (\chi(\bph+ \hh)-\chi(\bph))  \E(\bu)
	\\ & \qquad 
	- (\hat \CCC (\bph + \hh) -  \hat \CCC (\bph )) (\chi(\bph+ \hh)-\chi(\bph))  \E(\kk)
	\\ & \qquad 
	- \hat \CCC (\bph ) (\chi(\bph+ \hh)-\chi(\bph))  \E(\kk)
	\\ & \qquad 
	- (\hat \CCC (\bph + \hh) -  \hat \CCC (\bph )) \chi(\bph)  \E(\kk)
	=:  {\cal Y}_2.
\end{aligned}
\end{equation}
Similar to \eqref{TaylorC}, we have, for positive constants $C_{\chi'}$ and $C_{\hat{\CCC}'}$, that
\begin{equation}\label{Taylor}
\begin{aligned}
|\chi(\bph+\hh) - \chi(\bph ) - \chi'(\bph) \hh | 
&	\leq 	|\hh| \int_0^1 |\chi'(\bph + s \hh) - \chi'(\bph)| \ds
	\leq C_{\chi'} |\hh|^2, \\
|\hat{\CCC}(\bph+\hh) - \hat{\CCC}(\bph ) - \hat{\CCC}'(\bph) \hh | 
&	\leq 	|\hh| \int_0^1 |\hat{\CCC}'(\bph + s \hh) - \hat{\CCC}'(\bph)| \ds
	\leq C_{\hat{\CCC}'} |\hh|^2,
\end{aligned}
\end{equation}
and upon choosing $\bz = \bxi$ in \eqref{Fre:hat:main}, we infer that
\begin{align}\label{Fre:est:main}
\norma{\E(\bxi)}_{L^2(\Omega)} \leq C \| {\cal Y}_1 \|_{L^2(\Omega)} + C \| {\cal Y}_2 \|_{L^2(\Omega)}.
\end{align}
Then, employing the Young and H\"older inequalities, \eqref{cd:frechet:1}, as well as the stability estimates
\[
\norma{\hu^{\hh,\kk} - \hu^{\hh}}_{H^1(\Omega)} \leq C \norma{\kk}_{H^1(\Omega)}, \quad \norma{\hu^{\hh} - \hu}_{H^1(\Omega)} \leq C \norma{\hh}_{L^\infty(\Omega)}
\]
deduced from \eqref{cd:est}, we find that
\begin{align*}
\| {\cal Y}_1\|_{L^2(\Omega)} & \leq C \| \hh \|_{L^\infty(\Omega)} \Big (\norma{\hu^{\hh,\kk} - \hu^\hh}_{H^1(\Omega)}  + \norma{\hu^{\hh} - \hu}_{H^1(\Omega)} + \norma{\hh}_{L^\infty(\Omega)} \norma{\E(\hu)}_{L^2(\Omega)}  \Big) \\
& \leq C \norma{(\hh,\kk)}_{L^\infty(\Omega) \times H^1(\Omega)}^2,\\
\| {\cal Y}_2\|_{L^2(\Omega)} & \leq C \| \hh \|_{L^\infty(\Omega)} \Big ( \| \hh \|_{L^\infty(\Omega)} \| \E(\bu) \|_{L^2(\Omega)} + \norma{\hh}_{L^\infty(\Omega)}\| \kk \|_{H^1(\Omega)} + \| \kk \|_{H^1(\Omega)} \Big ) 
\\ & \leq C \norma{(\hh,\kk)}_{L^\infty(\Omega) \times H^1(\Omega)}^2,
\end{align*}
and thus we obtain by Korn's inequality
\[
\norma{\bxi}_{H^1(\Omega)} \leq  C \norma{(\hh,\kk)}_{L^\infty(\Omega) \times H^1(\Omega)}^2,
\]
which verifies the Fr\'echet differentiability of $\S_2$. Furthermore, it is clear that by the uniqueness of the solutions, the sum $\hv + \hw$ is equal to $\hat{\bth}$. To establish the identification of the partial derivatives $D_{\bph} \S_2(\bph, \bu)(\hh) = \hv$ and $D_{\bu}\S_2(\bph, \bu)(\kk) = \hw$, we argue as follows: Consider $\kk = \bm{0}$, so that from \eqref{wf:sys:hw} we obtain that $\hw = \bm{0}$ and $\hat{\bth} = \hv$.  Then, in  \eqref{Fre:hat:main} with $\kk = \bm{0}$, we now have for $\bxi = \hu^{\hh} - \hu - \hv$ the estimate \eqref{Fre:est:main}, where 
\begin{align*}
{\cal Y}_1 & = (\hat{\CCC}(\bph + \hh) - \hat{\CCC}(\bph))(\E(\hu^{\hh}) - \E(\hu)) + [ \hat \CCC (\bph + \hh) -  \hat \CCC (\bph )-  \hat \CCC' (\bph )\hh] \E(\hu), \\
{\cal Y}_2 & = -[ \hat \CCC (\bph + \hh) -  \hat \CCC (\bph )-  \hat \CCC' (\bph )\hh] \chi(\bph) \E(\bu)
	\\ & \quad - \hat \CCC(\bph) [\chi(\bph + \hh) - \chi(\bph) - \chi'(\bph)\hh] \E(\bu)
	\\ & \quad 	- (\hat \CCC (\bph + \hh) -  \hat \CCC (\bph )) (\chi(\bph+ \hh)-\chi(\bph))  \E(\bu),
\end{align*}
where we made use of $\hu^{\hh,\bm{0}} = \hu^{\hh}$. A short calculation shows that 
\[
\|\hu^{\hh} - \hu - \hv\|_{H^1(\Omega)} \leq C \| \hh \|_{L^\infty(\Omega)}^2,
\]
which entails that $D_{\bph} \S_2(\bph, \bu)(\hh) = \hv$.  The other identification $D_{\bu}\S_2(\bph, \bu)(\kk) = \hw$ is in fact more straightforward as $\S_2(\bph, \cdot)$ is a linear operator.  This completes the proof.
\end{proof}

\subsection{Adjoint systems}
We now move to the investigation of some adjoint systems which, as typically happens in constrained minimisation problems, will allow us to simplify the formulation of the optimality conditions of $(\bf P)$.
\begin{thm}\label{THM:EX:ADJ:1}
Under \ref{ass:dom}--\ref{ass:utar}, for every $\bph \in L^\infty(\Omega, \RRR^L)$, there exists a unique solution $\hat \qq \in \hatHD$ to 
\begin{subequations}\label{adj:1}
\begin{alignat}{2}
	\label{ADJ:SECOND:1}
	 -\div \big(\hat \CCC(\bph) \E(\hat \qq)\big) & = \0 \qquad && \text{ in } \Omega,
	\\
	\label{ADJ:SECOND:2} 
	 \hat \qq & = \0 \qquad && \text{ on } \hat \Gamma_D,	
	\\
	\label{ADJ:SECOND:3}
	 (\hat \CCC(\bph)  \E(\hat \qq)) \nn & = W(\hu-\uu^{\rm tar}) \bigchi_{\Gamma^{\rm tar}} \qquad && \text{ on } \hat \Gamma_N,
\end{alignat}
\end{subequations}
in the sense
\begin{align}\label{wf:adj:q}
\iO{\hat{\CCC}(\bph) \E(\hq) : \E(\bz) } = \int_{\Gamma^{\rm tar}} W(\hu-\uu^{\rm tar}) \cdot \bz \dH \quad \forall \bz \in \hatHD.
\end{align}
Moreover, there exists a unique solution $\ov \pp \in \ovHD $ to 
\begin{subequations}\label{adj:2}
\begin{alignat}{2}
	\label{ADJ:FIRST:1}
	 -\div \big(\ov \CCC(\bph)\E (\ov \pp) -  \hat \CCC(\bph) \chi(\bph) \E(\hat \qq))\big) & = \0 \qquad && \text{ in }\Omega,
	\\
	\label{ADJ:FIRST:2}
	 \ov \pp & = \0 \qquad && \text{ on } \ov \Gamma_D,	
	\\
	\label{ADJ:FIRST:3}
	 (\ov \CCC(\bph)\E (\ov \pp) - \hat \CCC(\bph) \chi(\bph) \E(\hat \qq)) \nn & = \0 \qquad && \text{ on } \ov \Gamma_N,
\end{alignat}
\end{subequations}
in the sense
\begin{align}\label{wf:adj:p}
\iO{\ov{\CCC}(\bph) \E(\bp) : \E(\bz) } =  \iO{ \hat{\CCC}(\bph) \chi(\bph) \E(\hq) : \E(\bz) }\quad \forall \bz \in \ovHD,
\end{align}
where $\hat \qq\in \hatHD$ is the unique solution to \eqref{wf:adj:q}.
\end{thm}
As the proof of existence and uniqueness is simply an application of Theorem \ref{THM:EX:ABS}, we omit the details.

\begin{remark}
Notice that the adjoint variable $\bp$ to the Stage 1 deformation $\bu$ is dependent on the adjoint variable $\hq$ to the Stage 2 deformation $\hu$.  This {\it backwards} dependence has some parallels with the adjoint systems associated to time-dependent state equations, which have to be solved backwards in time.
\end{remark}

\subsection{First order optimality conditions}
\begin{thm} \label{THM:FOC:FINAL}
Under \ref{ass:dom}--\ref{ass:utar}, let $\bph^* \in \Uad$ be a minimiser to $\Jred$ with corresponding states $\bu^* = \S_1^2(\bph^*)$, $\hu^* = \S(\bph^*)$, and adjoint variables $(\bp,\hq)$ as unique solutions to \eqref{wf:adj:q} and \eqref{wf:adj:p} corresponding to $(\bph^*, \bu^*, \hu^*)$.  Then, it necessarily holds that
\begin{align}
	\notag
	& - \int_\Omega \ov \CCC'(\bph^*) (\bphi-\bph^*)\E(\bu^*) : \E(\ov \pp) \dx
	\\ & \quad 	\notag
	- \int_\Omega   \hat \CCC'(\bph^*)  (\bphi-\bph^*) ( \E (\hu^*)- \chi(\bph^*) \E(\bu^*)): \E (\hat \qq) \dx
	\\ & \quad \notag
	+ \int_\Omega  \hat \CCC(\bph^*)  \chi'(\bph^*)(\bphi-\bph^*)\E(\bu^*): \E (\hat \qq) \dx
	 \\ & \quad 
	+ 2 \gamma \eps \int_\Omega \nabla \bph^* \cdot \nabla (\bphi - \bph^*)\dx
	+ \frac \gamma \eps \int_\Omega \tilde{\Psi}_{,\bph} (\bph^*) \cdot (\bphi- \bph^*) \dx
	 \geq 0 \quad \forall \bphi \in \Uad,
	 \label{foc:final}
\end{align}
where we set $\tilde{\Psi}_{,\bph}$ as the vector of partial derivatives of $\tilde{\Psi}$.
\end{thm}
\begin{proof}
As $\Uad$ is a non-empty, closed and convex set, standard results in optimal control and convex analysis yield that the first order necessary optimality condition for $\bph^*$ is
\begin{align*}
	\langle D \Jred (\bph^*), \bphi - \bph^* \rangle 
	\geq 
	0 
	\quad 
	\forall \bphi \in \Uad,
\end{align*}
which reads, in view of Theorem \ref{THM:LIN:S1} and Theorem \ref{THM:LIN:S2}, as
\begin{equation}\label{foc:first}
\begin{aligned}
& \int_{\Gamma^{\rm tar}} W(\hu^* - \uu^{\rm tar}) \cdot \hat{\bth} \dH \\
& \quad +2 \gamma \eps \iO{ \nabla \bph^* \cdot \nabla (\bphi - \bph^*) } + \frac{\gamma}{\eps} \iO{ \tilde\Psi_{,\bph}(\bph^*) \cdot (\bphi - \bph^*)} \geq 0 \quad \forall \bph \in \Uad,
\end{aligned}
\end{equation}
where $\hat{\bth}$ is the unique solution to \eqref{wf:sys:hthe} with $\bph = \bph^*$, $\bu = \bu^*$, $\hu = \hu^*$, $\hh = \bphi - \bph^*$ and $\kk = \bv$ from \eqref{wf:sys:LIN} (also with $\hh = \bphi - \bph^*$). To simplify the above relation, the procedure is to compare the equalities obtained from \eqref{wf:sys:LIN} with $\bz = \bp$, \eqref{wf:sys:hthe} with $\kk = \bv$ and $\bz = \hq$, \eqref{wf:adj:q} with $\bz = \hat{\bth}$ and \eqref{wf:adj:p} with $\bz = \bv$.  A short calculation then shows 
\begin{align*}
 \int_{\Gamma^{\rm tar}} W(\hu^* - \uu^{\rm tar}) \cdot \hat{\bth} \dH	& 
	=
	-\int_\Omega  \ov \CCC'(\bph^*) (\bph-\bph^*)\E(\bu^*) : \E(\ov \pp) \dx
	\\ & \quad 	
	- \int_\Omega\hat \CCC'(\bph^*)  (\bph-\bph^*) ( \E (\hu^*)- \chi(\bph^*) \E(\bu^*)): \E (\hat \qq) \dx
	\\ & \quad 
	+ \int_\Omega  \hat \CCC(\bph^*)  \chi'(\bph^*)(\bph-\bph^*)\E(\bu^*): \E (\hat \qq) \dx,
\end{align*}
which allows us to remove the dependence of $\hat \bth$ from \eqref{foc:first} and leads to \eqref{foc:final}.
\end{proof}

\section{Sharp interface asymptotics}\label{SEC:Sharp}
In this section we study the behaviour of solutions in the sharp interface limit $\eps \to 0$. For $\eps >0$, we denote
\begin{align*}
E_\eps(\bph) & = \eps \int_\Omega |\nabla \bph|^2 \dx + \frac{1}{\eps} \int_\Omega \Psi(\bph) \dx, \\
G(\bph) & = \frac{1}{2} \int_{\Gamma^{\rm tar}} W (\S(\bph) - \uu^{\rm tar}) \cdot (\S(\bph) - \uu^{\rm tar}) \dH,
\end{align*}
where $\S$ is the solution operator defined in \eqref{controltostateoperator}, so that the corresponding reduced functional can be expressed as the sum $\Jred^\eps(\bph) = G(\bph) + \gamma E_\eps(\bph)$.  The asymptotic behaviour of solutions can be studied under the framework of $\Gamma$-convergence.  In order to state the result some preparation is needed. A function $\bph \in L^1(\Omega, \RRR^L)$ is termed a function of bounded variation in $\Omega$, written as $\bph \in \BV(\Omega, \RRR^L)$ if there exists a matrix-valued measure $D \bph$ of dimension $L \times d$ on $\Omega$ such that 
\[
\sum_{j=1}^L \int_\Omega \vp_j  (\div {\bs \psi})_j \dx = - \sum_{j=1}^L \sum_{i=1}^d \int_\Omega \psi_i^j d D_i \vp_j,
\]
for all $\bm{\psi} = (\psi_i^j)_{1 \leq i \leq d, 1 \leq j \leq L}$ where $\psi_i^j \in C^1_c(\Omega)$.
Let $T = \{\bm{e}_1, \dots, \bm{e}_L\}$ where $\Psi^{-1}(0) = T$. For $\bph \in \BV(\Omega, T)$ we set, for $i \in \{1,...,L\}$,
\[
E_\bph^i = \{ \x \in \Omega \, : \, \bph(\x) = \bm{e}_i \}
\]
and define the essential boundary $\pd^* E_{\bph}^i$ as
\[
\pd^* E_\bph^i = \Big \{ \x \in \RRR^d \, : \, \lim_{ \rho \to 0^+} \frac{|E_\bph^i \cap B_\rho(\x)|}{|B_\rho(\x)|} \notin \{0,1 \} \Big \}
\]
where, for any $\rho >0$, $B_\rho(\x)$ is the $\rho$-ball in $\RRR^d$ centered in $\x$, i.e., $B_\rho(\x)= \{ \y \in \RRR^d \, : \, |\y-\x| < \rho \}$. Consider the extended functionals
\begin{align*}
\EEE_\eps(\bph) & := \begin{cases}
E_\eps(\bph) & \text{ if } \bph \in H^1(\Omega, \RRR^L), \\
+\infty & \text{ elsewhere in } L^1(\Omega, \RRR^L),
\end{cases} \\
\EEE_0(\bph) & := \begin{cases}
\sum_{i,j=1, \, i < j}^L b_{ij} \mathcal{H}^{d-1} (\Omega \cap \pd^* E_{\bph}^i \cap \pd^* E_{\bph}^j) & \text{ if } \bph \in \BV(\Omega, T), \\
+\infty & \text{ elsewhere in } L^1(\Omega, \RRR^L),
\end{cases}
\end{align*}
with constants $b_{ij}$ defined as
\[
b_{ij} = \inf \Big \{ \int_0^1 \Psi^{1/2}(\bm{\gamma}(t))|\bm{\gamma}'(t)| \dt \, : \, \bm{\gamma} \in C^1([0,1]; \Delta^L), \, \bm{\gamma}(0) = \bm{e}_i, \, \bm{\gamma}(1) = \bm{e}_j \Big \}.
\]
Then, the $\Gamma$-convergence of $\EEE_\eps$ to $\EEE_0$ as $\eps \to 0$ is expressed as the following assertions:
\begin{itemize}
\item {\bf Liminf property.} If $\{\bph^\eps\}_{\eps > 0}$ is a sequence such that $\liminf_{\eps \to 0} \mathbb{E}_\eps(\bph^\eps) < \infty$ and $\bph^\eps \to \bph^0$ in $L^1(\Omega, \RRR^L)$, then $\bph^0 \in \BV(\Omega, T)$ with $\mathbb{E}_0(\bph^0) \leq \liminf_{\eps \to 0} \mathbb{E}_\eps(\bph^\eps)$.
\item {\bf Limsup property.} For any $\bph^0 \in L^1(\Omega, T)$, there exists a sequence $\{\bph^\eps\}_{\eps > 0} \subset H^1(\Omega, \RRR^L)$, such that $\bph^\eps \to \bph^0$ in $L^1(\Omega, \RRR^L)$ and $\limsup_{\eps \to 0} \EEE_\eps(\bph^\eps) \leq \mathbb{E}_0(\bph^0)$.
\item {\bf Compactness property.} Let $\{\bph^\eps\}_{\eps > 0}$ be a sequence such that $\sup_\eps \mathbb{E}_\eps(\bph^\eps) < \infty$. Then, there exists a non-relabelled subsequence and a function $\bph^0 \in \BV(\Omega, T)$ such that $\bph^\eps \to \bph^0$ in $L^1(\Omega, \RRR^L)$. 
\end{itemize}
For a proof we refer to \cite[Thm.~2.5 and Prop.~4.1]{Baldo}, see also \cite[Thm.~3.1 and Rmk.~3.3]{Bellettini} with the choice $f(z, X) = |X|^2$.

\subsection{Convergences of minimisers}
\begin{lem}\label{lem:SI:Gamma}
For each $\eps \in (0,1]$, let $\bph^{\eps} \in \Uad$ denote a minimiser to the extended reduced cost functional $\Jred^\eps(\bph) = G(\bph) + \gamma \mathbb{E}_\eps(\bph)$. Then, there exists a non-relabelled subsequence $\eps \to 0$ and a limit function $\bph_0 \in \BV(\Omega, T)$ such that $\bph^\eps \to \bph_0$ strongly in $L^1(\Omega, \RRR^L)$, $\lim_{\eps \to 0} \Jred^\eps(\bph^\eps) = \Jred^0(\bph_0)$, where 
\begin{align*}
\Jred^0(\bph) =  G(\bph) +\gamma  \mathbb{E}_0(\bph) \text{ for } \bph \in \BV(\Omega, T),
\end{align*}
and $\bph_0$ is a minimiser to $\Jred^0$.
\end{lem}

\begin{proof}
The proof relies on the $\Gamma$-convergence of the Ginzburg--Landau functional and the stability of $\Gamma$-convergence under continuous perturbations. By Corollary \ref{COR:cts}, and the continuity of the trace operator, we see that $G$ is continuous.  For arbitrary $\bm{\psi} \in \BV(\Omega, T)$, we invoke the limsup property to find a sequence $\{\bm{\psi}^\eps\}_{\eps > 0}$ such that $\bm{\psi}^\eps \to \bm{\psi}$ strongly in $L^1(\Omega, \RRR^L)$ and $\limsup_{\eps \to 0}  \EEE_\eps(\bm{\psi}^\eps) \leq \EEE_{0}(\bm{\psi}) < \infty$.  Continuity of $G$ implies $G(\bm{\psi}^\eps) \to G(\bm{\psi})$ as $\eps \to 0$, leading to
\[
\limsup_{\eps \to 0} \Jred^\eps(\bm{\psi}^\eps) \leq \Jred^0(\bm{\psi}) < \infty.
\]
As $\bph^\eps$ minimises $\Jred^\eps$, we see that 
\[
\limsup_{\eps \to 0} \Jred^\eps(\bph^\eps) \leq \limsup_{\eps \to 0} \Jred^\eps(\bm{\psi}^\eps) \leq \Jred^0(\bm{\psi}) < \infty .
\]
By the non-negativity of $G$, the above estimate implies $\sup_{\eps \in (0,1]}  \EEE_\eps(\bph^\eps) < \infty$, and by the compactness property we deduce that there exists a limit function $\bph_0 \in \BV(\Omega,T)$ such that $\bph^\eps \to \bph_0$ strongly in $L^1(\Omega, \RRR^L)$ along a non-relabelled subsequence.  Continuity of $G$ then gives $G(\bph^\eps) \to G(\bph_0)$ and invoking the liminf property we subsequently infer that 
\[
\Jred^0(\bph_0) \leq \liminf_{\eps \to 0} \Jred^\eps(\bph^\eps) \leq \Jred^0(\bm{\psi}).
\]
As $\bm{\psi}$ is arbitrary, this shows that $\bph_0$ is a minimiser of $\Jred^0$.  We now return to the beginning of the proof and consider using the limsup property to construct a sequence $\{\bm{\varphi}^\eps\}_{\eps > 0}$ that converges strongly to $\bph_0$ in $L^1(\Omega, \RRR^d)$. Then, following a similar argument we arrive at
\[
\Jred^0(\bph_0) \leq \liminf_{\eps \to 0} \Jred^\eps(\bph^\eps) \leq \limsup_{\eps \to 0} \Jred^\eps(\bph^\eps) \leq \Jred^0(\bph_0),
\] 
which provides the claimed assertion $\lim_{\eps \to 0} \Jred^\eps(\bph^\eps) = \Jred^0(\bph_0)$.
\end{proof}

\subsection{Formally matched asymptotic expansions}
We turn our attention towards the optimality condition \eqref{foc:final} and study its sharp interface limit $\eps \to 0$ with the method of formally matched asymptotic expansions, where we assume the functions $\bph^\eps$, $\bu^\eps$, $\hu^\eps$, $\ov{\pp}^\eps$, and $\hat{\qq}^\eps$ admit asymptotic expansions in powers of $\eps$. From Lemma \ref{lem:SI:Gamma} we saw that $\bph^\eps$ converges to a function $\bph_0 \in \BV(\Omega, T)$ as $\eps \to 0$, and thus for $0 < \eps < 1$, we expect $\bph^\eps$ to change its values rapidly on a length scale proportional to $\eps$. This inspires us to consider two asymptotic expansions of $(\bph^\eps, \bu^\eps, \hu^\eps, \ov{\pp}^\eps, \hat{\qq}^\eps)$ in the bulk and interfacial regions (to be defined below), and the procedure is to match these expansions in an intermediate region to deduce the expected equations in the sharp interface limit. We follow the ideas in \cite{BGFS} that treats a similar system of equations, and refer the reader to, e.g., \cite{Bronsard,Fife,GLNS,GNS} for more details regarding the methodology.

Recalling $T = \Psi^{-1}(0) = \{ \bm{e}_1, \dots, \bm{e}_L\}$ as the set of corners of the Gibbs simplex $\Delta^L$, we partition the domain $\Omega$ into regions $\Omega^i$, $i = 1, \dots, L$, where $\Omega^i = \{\x \in \Omega \, : \, \bph_0(\x) = \bm{e}_i \}$.  Then, we assume the functions $\bph^\eps$, $\bu^\eps$, $\hu^\eps$, $\ov{\pp}^\eps$, and $\hat{\qq}^\eps$ are sufficiently smooth and admit the following asymptotic expansion in $\eps$:
\begin{align*}
& \bph^\eps(\x) = \sum_{k=0}^\infty \eps^k \bph_k(\x), \quad \bu^\eps(\x) = \sum_{k=0}^\infty \eps^k \bu_k(\x), \quad \hu^\eps(\x) = \sum_{k=0}^\infty \eps^k \hu_k(\x), \\
& \ov{\pp}^\eps(\x) = \sum_{k=0}^\infty \eps^k \ov{\pp}_k(\x), \quad \hat{\qq}^\eps(\x) = \sum_{k=0}^\infty \eps^k \hat{\qq}_k(\x),
\end{align*}
for all points $\x \in \Omega$ away from the interfaces $\Gamma_{ij} = \pd \Omega_i \cap \pd \Omega_j$ for $i,j \in \{1, \dots, L\}$, $i \neq j$. This is known as the outer expansion. Furthermore, we assume that 
\[
\bph_k(\x) \in T \Sigma^L := \Big \{ \bm{v} = (v_1, \dots, v_L) \in \RRR^L \, : \, \sum_{i=1}^L v_i = 0 \Big \}, \quad k \geq 1,
\]
where $T \Sigma^L$ is the tangent space of the affine hyperplane $\Sigma^L = \{ \bm{v} \in \RRR^L \, : \, \sum_{i=1}^L v_i = 1\}$, so that by the above construction $\bph^\eps(\x) \in \Delta^L$ for $\eps$ sufficiently small. We assume that there are constant elasticity tensors $\ov{\CCC}_i$ and $\hat{\CCC}_i$ for $i = 1, \dots, L$, satisfying the standard symmetric conditions and are positive definite.  Then, for $\bph = (\vp_1, \dots, \vp_L)$ such that $\bph(\x) \in \Delta^L$, we consider
\[
\ov{\CCC}(\bph) = \sum_{i=1}^L \ov{\CCC}_i \vp_i, \quad \hat{\CCC}(\bph) = \sum_{i=1}^L \hat{\CCC}_i \vp_i.
\]
Then, substituting the outer expansions into the state systems \eqref{bu:sys}, \eqref{hu:sys} and the adjoint systems \eqref{adj:1} and \eqref{adj:2}, to leading order we obtain the following equations for $i = 1, \dots, L$:
\begin{equation}\label{outer:sys1}
\begin{aligned}
& \begin{cases}
-\div \big(\ov{\CCC}_i \E (\bu_0)\big) = \ov{\FF} & \text{ in } \Omega^i,
	\\
 \bu_0 = \bU  & \text{ on } \ov \Gamma_D \cap \pd \Omega^i,	
	\\
(\ov{\CCC}_i \E (\bu_0)) \nn = \ov{\gg} & \text{ on } \ov \Gamma_N \cap \pd \Omega^i,\end{cases}   \\
& \begin{cases}
-\div \big(\hat{\CCC}_i ( \E (\hu_0)- \chi_i \E(\bu_0))\big) = \hat{\FF} & \text{ in } \Omega^i, \\
 \hu_0 = \hU & \text{ on } \hat \Gamma_D \cap \pd \Omega^i,\\
  (\hat{\CCC}_i ( \E (\hu_0)- \chi_i \E(\bu_0))) \nn = \hat{\gg} & \text{ on } \hat \Gamma_N \cap \pd \Omega^i,
\end{cases}
\end{aligned}
\end{equation}
where $\chi_i := \chi(\bm{e}_i)$, and 
\begin{equation}\label{outer:sys2}
\begin{aligned}
& \begin{cases}
-\div \big(\hat \CCC_i \E(\hat \qq_0)\big) = \0 & \text{ in } \Omega^i,
	\\
	\hat \qq_0= \0 & \text{ on } \hat \Gamma_D \cap \pd \Omega^i,	
	\\
 (\hat \CCC_i\E(\hat \qq_0)) \nn = W(\hu_0-\uu^{\rm tar}) \bigchi_{\Gamma^{\rm tar}} & \text{ on } \hat \Gamma_N \cap \pd \Omega^i, 
\end{cases} \\
& \begin{cases}
-\div \big(\ov{\CCC}_i \E (\ov \pp_0) -  \hat{\CCC}_i \chi_i \E(\hat \qq_0))\big) = \0 & \text{ in }\Omega^i,
	\\
	\ov \pp_0 = \0  & \text{ on } \ov \Gamma_D \cap \pd \Omega^i,	
	\\
	(\ov{\CCC}_i \E (\ov \pp_0) - \hat{\CCC}_i \chi_i \E(\hat \qq_0)) \nn = \0 & \text{ on } \ov \Gamma_N \cap \pd \Omega^i.
\end{cases}
\end{aligned}
\end{equation}
It then remains to furnish the above with boundary conditions for $(\bu_0, \hu_0, \ov{\pp}_0, \hat{\qq}_0)$ on the interfaces $\Gamma_{ij}$ for $i,j \in \{1, \dots, L \}$, $i < j$, which we assume are smooth hypersurfaces that can be obtained in the limit $\eps \to 0$.  These boundary conditions can be inferred with the help of a corresponding inner expansion for $(\bph^\eps,\bu^\eps, \hu^\eps, \ov{\pp}^\eps, \hat{\qq}^\eps)$ in the interfacial regions bordering two bulk regions $\Omega^i$ and $\Omega^j$. To this end, we focus on one particular interface $\Gamma_{ij}$ and introduce a change of coordinates with the help of a local parameterisation $\bm{\gamma}: U \subset \RRR^{d-1} \to \RRR^d$ of $\Gamma_{ij}$, where $U$ is a spatial parameter domain.

Close to $\bm{\gamma}(U)$, consider the signed distance function $d$ such that $d(\x) > 0$ if $\x \in \Omega^j$ and $d(\x) <0$ if $\x \in \Omega^i$, so that the unit normal $\bm{\nu}$ to $\Gamma_{ij}$ points from $\Omega^i$ to $\Omega^j$. Introducing the rescaled signed distance $z = \frac{d}{\eps}$, a local parameterization of $\x \in \RRR^d$ close to ${\bm \gamma}(U)$ can be given as
\[
\x = {\cal G}^\eps(\s, z) = \bm{\gamma}(\s) + \eps z \bm{\nu}(\s), \quad \s \in U \subset \RRR^{d-1}, \, z \in \RRR.
\]
This representation allows us to infer the following expansions for gradients, divergences and Laplacians \cite{GStin}:
\begin{align*}
\nabla_{x} b & = \frac{1}{\eps} \pd_z \hat b \bm{\nu} + \nabla_{\Gamma_{ij}} \hat b + \mathcal{O}(\eps), \quad \div_{x} \bm{j} = \frac{1}{\eps} \pd_z \hat{\bm{j}} \cdot \bm{\nu} + \div_{\Gamma_{ij}} \hat{\bm{j}} + \mathcal{O}(\eps), \\
\nabla_{x} \bm{j} & = \frac{1}{\eps} \pd_z \hat{\bm{j}} \otimes \bm{\nu} + \nabla_{\Gamma_{ij}} \hat{\bm{j}} + \mathcal{O}(\eps), \quad \Delta_{x} b = \frac{1}{\eps^2} \pd_{zz} \hat b - \frac{1}{\eps} \kappa_{\Gamma_{ij}} \pd_z \hat b + \mathcal{O}(1),
\end{align*}
for scalar functions $b(\x) = \hat{b}(\s(\x), z(\x))$ and vector functions $\bm{j}(\x) =\hat{\bm{j}}(\s(\x), z(\x))$, along with the curvature $\kappa_{\Gamma_{ij}}$ of $\Gamma_{ij}$, the surface gradient operator $\nabla_{\Gamma_{ij}}$, and the surface divergence operator $\div_{\Gamma_{ij}}$ on $\Gamma_{ij}$. Then, for points close by $\Gamma_{ij}$, we assume an inner expansion of the form
\begin{align*}
& \bph^\eps(\x) = \sum_{k=0}^\infty \eps^k \bm{\Phi}_k(\s,z), \quad \bu^\eps(\x) = \sum_{k=0}^\infty \eps^k \bU_k(\s,z), \quad \hu^\eps(\x) = \sum_{k=0}^\infty \eps^k \hU_k(\s,z), \\
& \ov{\pp}^\eps(\x) = \sum_{k=0}^\infty \eps^k \bP_k(\s,z), \quad \hat{\qq}^\eps(\x) = \sum_{k=0}^\infty \eps^k \hQ_k(\s,z).
\end{align*}
Lastly, we assume in a tubular neighborhood of $\Gamma_{ij}$ the outer expansions and the inner expansions hold simultaneously. Within this intermediate region certain matching conditions relating the outer expansions to the inner expansions must hold.  For a scalar function $b(\x)$ admitting an outer expansion $\sum_{k=0}^\infty \eps^k b_k(\x)$ and an inner expansion $\sum_{k=0}^\infty \eps^k B_k(\s,z)$, it holds that (see \cite[Appendix D]{GStin})
\begin{align*}
B_0(\s,z) & \to \begin{cases}
\lim_{\delta \searrow 0} b_0(\x + \delta \bm{\nu}(\x)) =: b_0^+(\x) & \text{ for } z \to +\infty, \\
\lim_{\delta \searrow 0} b_0(\x - \delta \bm{\nu}(\x)) =: b_0^-(\x) & \text{ for } z \to -\infty,
\end{cases} \\
 \pd_z B_0(\s,z) & \to 0 \text{ as } z \to \pm\infty, \\
\pd_z B_1(\s,z) & \to \begin{cases}
\lim_{\delta \searrow 0} (\nabla b_0)(\x + \delta \bm{\nu}(\x)) \cdot \bm{\nu}(\x) =: \nabla b_0^+ \cdot \bm{\nu} & \text{ for } z \to +\infty, \\
\lim_{\delta \searrow 0} (\nabla b_0)(\x - \delta \bm{\nu}(\x)) \cdot \bm{\nu}(\x) =: \nabla b_0^- \cdot \bm{\nu}& \text{ for } z \to -\infty,
\end{cases}
\end{align*}
for $\x \in \Gamma_{ij}$. Consequently, we denote the jump of a quantity $b$ across $\Gamma_{ij}$ as 
\[
[b]_{-}^{+} := \lim_{\delta \searrow 0} b(\x + \delta \bm{\nu}(\x)) - \lim_{\delta \searrow 0} b(\x - \delta \bm{\nu}(\x) )  \quad \text{ for } \x \in \Gamma_{ij}.
\]
Note that the above matching conditions also apply to vectorial functions.  We introduce the orthogonal projection
\[
\bm{P}_{T \Sigma} : \RRR^L \to T \Sigma^L, \quad \bm{P}_{T \Sigma} \bph = \bph - \Big ( \frac{1}{L} \sum_{i=1}^L \vp_i \Big ) \bm{1}
\]
where $\bm{1} := (1, \dots, 1)^{\top}$, so that the optimality condition \eqref{foc:final} can be expressed in the following strong form
\begin{equation}\label{Opt:strong}
\begin{aligned}
& - 2 \gamma \eps \Delta \bph^\eps + \bm{P}_{T \Sigma} \Big ( \frac{\gamma}{\eps}\tilde{\Psi}_{,\bph} (\bph^\eps) + \hat{\CCC}(\bph^\eps) \chi'(\bph^\eps) \E(\bu^\eps) : \E(\hat{\qq}^\eps) \Big )\\
& \quad -   \bm{P}_{T \Sigma} \Big ( \hat{\CCC}'(\bph^\eps) (\E(\hu^\eps) - \chi(\bph^\eps) \E(\bu^\eps)) : \E(\hat{\qq}^\eps) + \ov{\CCC}'(\bph^\eps) \E(\bu^\eps): \E(\bp^\eps)  \Big )= \bm{0}.
\end{aligned}
\end{equation}
We substitute the inner expansions into the equations \eqref{SYS:1}, \eqref{SYS:2}, \eqref{ADJ:SECOND:1}, \eqref{ADJ:FIRST:1}, and \eqref{Opt:strong} and collect terms of the same order. Then, we perform some computations to deduce the boundary conditions posed on $\Gamma_{ij}$. As the subsequent analysis is similar to that performed in \cite{BGFS,GLNS}, we omit most of the straightforward details. In the sequel we use the notation 
\[
(\bm{B})^{\mathrm{sym}} = \frac{1}{2}(\bm{B} + \bm{B}^{\top}), \quad \E_\XX := (\pd_z \XX_1 \otimes \bm{\nu} + \nabla_{\Gamma_{ij}} \XX_0)^{\mathrm{sym}},
\] 
for second order tensors $\bm{B}$ and for $\XX \in \{\bU, \hU, \bP, \hQ\}$.

To leading order $\mathcal{O}(\frac{1}{\eps^2})$, equations \eqref{SYS:1} and \eqref{ADJ:SECOND:1} yield
\[
\pd_z \Big (\ov{\CCC}(\bm{\Phi}_0) ( \pd_z \bU_0 \otimes \bm{\nu})^{\mathrm{sym}} \bm{\nu} \Big ) = \bm{0}, \quad \pd_z \Big (\hat{\CCC}(\bm{\Phi}_0) ( \pd_z \hQ_0 \otimes \bm{\nu})^{\mathrm{sym}} \bm{\nu} \Big ) = \bm{0}.
\]
Multiplying by $\bU_0$ and $\hQ_0$, respectively, integrating over $z \in \RRR$, integrating by parts and applying the matching conditions allow us to deduce that $\pd_z \bU_0 = \pd_z \hQ_0 = \bm{0}$, and hence both $\bU_0$ and $\hQ_0$ are constant in $z$.  Applying matching conditions we infer that
\[
[\bu_0]_{i}^j = [\hq_0]_{i}^j = \bm{0} \quad \text{ on } \Gamma_{ij}.
\]
Then, to leading order $\mathcal{O}(\frac{1}{\eps^2})$, equations \eqref{SYS:2} and \eqref{ADJ:FIRST:1} yield
\[
\pd_z \Big (\hat{\CCC}(\bm{\Phi}_0) ( \pd_z \hU_0 \otimes \bm{\nu})^{\mathrm{sym}} \bm{\nu} \Big ) = \bm{0}, \quad \pd_z \Big (\ov{\CCC}(\bm{\Phi}_0) ( \pd_z \bP_0 \otimes \bm{\nu})^{\mathrm{sym}} \bm{\nu} \Big ) = \bm{0},
\]
on account of the fact that $\pd_z \bU_0 = \pd_z \hQ_0 = \bm{0}$.  Hence, we also obtain
\[
[\hu_0]_{i}^j = [\bp_0]_{i}^j = \bm{0} \quad \text{ on } \Gamma_{ij}.
\]
To first order $\mathcal{O}(\frac{1}{\eps})$, we get from \eqref{SYS:1} and \eqref{ADJ:SECOND:1} that
\begin{align}\label{SI:bUhQ}
\pd_z \Big ( \ov{\CCC}(\bm{\Phi}_0) \E_{\bU} \bm{\nu} \Big )  = \bm{0}, \quad
\pd_z \Big ( \hat{\CCC}(\bm{\Phi}_0) \E_{\hQ}\bm{\nu} \Big )  = \bm{0}.
\end{align}
Integrating over $z \in \RRR$ and using the matching conditions yields
\[
[\ov{\CCC} \E(\bu_0) \bm{\nu}]_{i}^j = \bm{0}, \quad [\hat{\CCC} \E(\hq_0) \bm{\nu}]_i^j = \bm{0} \quad \text{ on } \Gamma_{ij}.
\]
Similarly, from equations \eqref{SYS:2} and \eqref{ADJ:FIRST:1}, we obtain to first order $\mathcal{O}(\frac{1}{\eps})$ that
\begin{align}\label{SI:hUbP}
 \pd_z \Big ( \hat{\CCC}(\bm{\Phi}_0) (\E_{\hU} - \chi(\bm{\Phi_0}) \E_{\bU}) \bm{\nu} \Big )  = \bm{0}, \quad \pd_z \Big ( \ov{\CCC}(\bm{\Phi}_0) \E_{\bP} \bm{\nu} - \chi(\bm{\Phi}_0) \hat{\CCC}(\bm{\Phi}_0) \E_{\hQ}\bm{\nu} \Big )  = \bm{0}.
\end{align}
Integrating over $z \in \RRR$ and applying the matching conditions, we obtain
\[
[\hat{\CCC} (\E(\hu_0) - \chi(\bph_0) \E(\bu_0))]_{i}^j \bm{\nu} = \bm{0}, \quad [\ov{\CCC} \E(\bp_0) - \chi(\bph_0) \hat{\CCC} \E(\hq_0)]_i^j \bm{\nu} = \bm{0} \quad \text{ on } \Gamma_{ij}.
\]
Turning now to the optimality condition \eqref{Opt:strong}, we use the fact that $\pd_z \bU_0 = \pd_z \hU_0 = \pd_z \bP_0 = \pd_z \hQ_0 = \bm{0}$ to see that the elasticity terms do not contribute to leading order $\mathcal{O}(\frac{1}{\eps^2})$ and first order $\mathcal{O}(\frac{1}{\eps})$. Hence, to first order $\mathcal{O}(\frac{1}{\eps})$ we obtain from \eqref{Opt:strong} that
\begin{align*}
2 \pd_{zz} \bm{\Phi}_0 - \bm{P}_{T \Sigma} \Big ( \tilde{\Psi}_{,\bph} (\bm{\Phi}_0)  \Big ) = \bm{0}.
\end{align*}
Following \cite{Bronsard}, $\bm{\Phi}_0$ can be chosen independent of $s$ and as the solution to the above ordinary differential system such that $\lim_{z \to - \infty} \bm{\Phi}_0(z) = \bm{e}_i$ and $\lim_{z \to +\infty} \bm{\Phi}_0(z) = \bm{e}_j$. To the next order $\mathcal{O}(1)$, we obtain \begin{equation}\label{SI:Phi1}
\begin{aligned}
& - 2 \gamma \pd_{zz} \bm{\Phi}_1 + \gamma   \bm{P}_{T \Sigma} \Big (  \tilde{\Psi}_{,\bph\bph}(\bm{\Phi}_0) \bm{\Phi}_1 \Big ) + 2 \gamma \kappa_{\Gamma_{ij}} \pd_z \bm{\Phi}_0 +  \bm{P}_{T \Sigma} \Big ( \hat{\CCC}(\bm{\Phi}_0) \chi'(\bm{\Phi}_0) \E_{\bU}: \E_{\hQ}  \Big ) \\
& \quad -  \bm{P}_{T \Sigma} \Big (  \hat{\CCC}'(\bm{\Phi}_0) (\E_{\hU} -\chi(\bm{\Phi}_0) \E_{\bU}) : \E_{\hQ} + \ov{\CCC}'(\bm{\Phi}_0) \E_{\bU} : \E_{\bP} \Big ) = \bm{0},
\end{aligned}
\end{equation}
where $\tilde\Psi_{,\bph\bph}$ denotes the Hessian matrix of $\tilde\Psi$, and we have used that $\pd_z \bm{\Phi}_0 \in T \Sigma^L$. Note that by the fact that $\pd_z \XX_0 = \bm{0}$ for $\XX \in \{\bU, \hU, \bP, \hQ\}$, and by the symmetry of the elasticity tensors $\CCC_{ijkl} = \CCC_{jikl}$, we have the relations
\begin{align}\label{Identities:SI}
\pd_z \E_\XX = (\pd_{zz} \XX_1 \otimes \bm{\nu})^{\mathrm{sym}}, \quad \CCC \E_\YY : \pd_z \E_\XX = (( \CCC \E_\YY) \bm{\nu}) \cdot \pd_{zz} \XX_1,
\end{align}
for any $\XX,\YY \in \{\bU, \hU, \bP, \hQ\}$. To obtain a solution $\bm{\Phi}_1$, a solvability condition has to hold, which can be derived by multiplying \eqref{SI:Phi1} with $\pd_z \bm{\Phi}_0$ and integrating over $z$. Using the relations $\bm{P}_{T \Sigma} (\pd_z \bm{\Phi}_0) = \pd_z \bm{\Phi}_0$, $\bm{P}_{T \Sigma}(\pd_z \bm{\Phi}_1) = \pd_z \bm{\Phi}_1$, after integrating by parts and applying the matching conditions, we obtain
\begin{equation}\label{SI:Phi1:a}
\begin{aligned}
0 & = \int_{-\infty}^\infty \gamma \underbrace{(-2 \pd_{zz} \bm{\Phi}_0 + \bm{P}_{T \Sigma} \Big ( \tilde{\Psi}_{,\bph}(\bm{\Phi}_0)) \Big )}_{=\bm{0}} \cdot \, \pd_z \bm{\Phi}_1 dz + 2\gamma \kappa_{\Gamma_{ij}} \int_{-\infty}^\infty |\pd_z \bm{\Phi}_0|^2 dz \\
& \quad + \int_{-\infty}^\infty \pd_z[ \hat{\CCC}(\bm{\Phi}_0) \chi(\bm{\Phi}_0)] \E_{\bU} : \E_{\hQ} - \pd_z \hat{\CCC}(\bm{\Phi}_0) \E_{\hU} : \E_{\hQ}  - \pd_z \ov{\CCC}(\bm{\Phi}_0) \E_{\bU} : \E_{\bP} \; dz.
\end{aligned}
\end{equation}
We employ the identities obtained from \eqref{SI:bUhQ}, \eqref{SI:hUbP} and \eqref{Identities:SI} to obtain that
\begin{align*}
-\hat{\CCC}(\bm{\Phi}_0) \E_{\hU} : \pd_z \E_{\hQ} + \hat{\CCC}(\bm{\Phi}_0) \E_{\hU} \bm{\nu} \cdot \pd_{zz} \hQ_1 = \bm{0}, \quad \pd_z(\hat{\CCC}(\bm{\Phi}_0) \E_{\hQ} \bm{\nu})  = \bm{0},
\end{align*}
as well as $\pd_z \bm{\nu} = \bm{0}$ to see that
\begin{align*}
& - \pd_z \hat{\CCC}(\bm{\Phi}_0) \E_{\hU} : \E_{\hQ} \\
& \quad = - \pd_z \hat{\CCC}(\bm{\Phi}_0) \E_{\hU} : \E_{\hQ} - \hat{\CCC}(\bm{\Phi}_0) \E_{\hU} : \pd_z \E_{\hQ} - \hat{\CCC}(\bm{\Phi}_0) \E_{\hQ} : \pd_z \E_{\hU}  \\
& \qquad + \hat{\CCC}(\bm{\Phi}_0) \E_{\hU} \bm{\nu} \cdot \pd_{zz} \hQ_1 + \hat{\CCC}(\bm{\Phi}_0) \E_{\hQ} \bm{\nu} \cdot \pd_{zz} \hU_1\\
& \qquad + \pd_z \Big ( \hat{\CCC}(\bm{\Phi}_0) \E_{\hQ} \Big ) \bm{\nu} \cdot \pd_z \hU_1 + \pd_z \Big (\hat{\CCC}(\bm{\Phi}_0) (\E_{\hU} - \chi(\bm{\Phi}_0) \E_{\bU}) \Big ) \bm{\nu} \cdot \pd_z \hQ_1 \\
& \quad = - \pd_z \Big ( \hat{\CCC}(\bm{\Phi_0})  \E_{\hU} : \E_{\hQ} - \hat{\CCC}(\bm{\Phi_0}) \E_{\hQ}  \bm{\nu} \cdot \pd_z \hU_1  - \hat{\CCC}(\bm{\Phi_0})  \E_{\hU}  \bm{\nu} \cdot  \pd_z \hQ_1 \big) \Big )  \\
& \qquad - \pd_z \Big ( \hat{\CCC}(\bm{\Phi}_0) \chi(\bm{\Phi}_0) \E_{\bU} \Big ) \bm{\nu} \cdot \pd_z \hQ_1.
\end{align*}
Via a similar calculation we infer
\begin{align*}
& - \pd_z \ov{\CCC}(\bm{\Phi}_0) \E_{\bU} : \E_{\bP}\\
& \quad  = - \pd_z \Big ( \ov{\CCC}(\bm{\Phi}_0) \E_{\bU} : \E_{\bP} - \ov{\CCC}(\bm{\Phi}_0) \E_{\bP} \bm{\nu} \cdot \pd_z \bU_1 - \ov{\CCC}(\bm{\Phi}_0) \E_{\bU} \bm{\nu} \cdot \pd_z \bP_1 \Big ) \\
& \qquad - \pd_z \Big ( \hat{\CCC}(\bm{\Phi}_0) \chi(\bm{\Phi}_0) \E_{\hQ} \Big ) \bm{\nu} \cdot \pd_z \bU_1, \\
& \pd_z \big [ \hat{\CCC}(\bm{\Phi}_0) \chi(\bm{\Phi}_0) \big] \E_{\bU} : \E_{\hQ} \\
& \quad = \pd_z \Big ( \chi(\bm{\Phi}_0) \hat{\CCC}(\bm{\Phi}_0) \E_{\bU} : \E_{\hQ} \Big ) \\
& \qquad - \chi(\bm{\Phi}_0) \hat{\CCC}(\bm{\Phi}_0) \E_{\bU} \bm{\nu} \cdot \pd_{zz} \hQ_1 - \chi(\bm{\Phi}_0) \hat{\CCC}(\bm{\Phi}_0) \E_{\hQ} \bm{\nu} \cdot \pd_{zz} \bU_1, 
\end{align*}
and thus, setting $b_{ij} := \int_{-\infty}^{\infty} 2|\pd_z \bm{\Phi}_0|^2 dz$, and applying matching conditions for $\E_\XX$ and $\pd_z \XX_1$ with $\XX \in \{\bU, \hU, \bP, \hQ\}$, we obtain from \eqref{SI:Phi1:a} the solvability condition
\begin{equation}\label{SI:Phi1:b}
\begin{aligned}
0 & =  \gamma \kappa_{\Gamma_{ij}}b_{ij} + [\chi(\bph_0) \hat{\CCC} \E(\bu_0) : \E(\hq_0)]_i^j - [\hat{\CCC}\E(\hu_0) : \E(\hq_0)]_i^j - [\ov{\CCC} \E(\bu_0) : \E(\bp_0) ]_i^j \\
& \quad + [\hat{\CCC} \E(\hq_0) \bm{\nu} \cdot (\nabla \hu_0) \bm{\nu} + \hat{\CCC} \E(\hu_0) \bm{\nu} \cdot (\nabla \hq_0) \bm{\nu}]_i^j - [\chi(\bph_0)\hat{\CCC} \E(\bu_0) \bm{\nu} \cdot (\nabla \hq_0) \bm{\nu}]_i^j \\
& \quad + [\ov{\CCC} \E(\bp_0) \bm{\nu} \cdot (\nabla \bu_0) \bm{\nu} + \ov{\CCC} \E(\bu_0) \bm{\nu} \cdot (\nabla \bp_0) \bm{\nu}]_i^j - [\chi(\bph_0) \hat{\CCC} \E(\hq_0) \bm{\nu} \cdot (\nabla \bu_0) \bm{\nu}]_i^j
\end{aligned}
\end{equation}
that has to hold on $\Gamma_{ij}$. Thus, the sharp interface limit consists of the equations \eqref{outer:sys1} and \eqref{outer:sys2} posed in $\Omega^i$, $1 \leq i \leq L$, furnished by the boundary conditions \eqref{SI:Phi1:b} and 
\begin{align*}
& [\hu_0]_i^j = \bm{0}, \quad [\bp_0]_i^j = \bm{0}, \quad [\bu_0]_i^j =  \bm{0}, \quad [\hq_0]_i^j = \bm{0}, \\
& [\ov{\CCC} \E(\bu_0) \bm{\nu}]_{i}^j = \bm{0}, \quad [\hat{\CCC} \E(\hq_0)) \bm{\nu}]_i^j = \bm{0}, \quad  [\ov{\CCC} \E(\bu_0) \bm{\nu}]_{i}^j = \bm{0}, \quad [\hat{\CCC} \E(\hq_0)) \bm{\nu}]_i^j = \bm{0} 
\end{align*}
on $\Gamma_{ij}$, $1 \leq i < j \leq L$.

\begin{remark}
It is possible to consider the sharp interface limit near a triple junction where three regions meet.  We refer to \cite{BGFS,Bronsard,Owen} for more details regarding the asymptotic analysis around a triple junction.
\end{remark}

\subsection{Rigorous convergence in the two-phase setting}
In the two phase case $L = 2$, since $\bph = (\vp_1, \vp_2) \in \Delta^2$, we may use the difference $\vp := \vp_2 - \vp_1 \in [-1,1]$ to encode the vector $\bph = (\frac{1}{2}(1-\vp), \frac{1}{2}(1+\vp))$, so that $|\nabla \bph|^2 = \frac{1}{2} |\nabla \vp|^2$.  
Hence, the problem {\bf (P)} can be re-phrased in terms of the scalar function $\vp$ ranging between $-1$ and $1$, and it suffices to consider the following
\begin{align*}
(\bm  P^\eps) \quad \min_{\vp \in \Uad} \Jred^\eps(\vp) & =  \frac{1}{2}\int_{\Gamma^{\mathrm{tar}}} W(\S(\vp) - \bm{u}^{\mathrm{tar}}) \cdot (\S(\vp) - \uu^{\rm tar}) \dH \\
& \quad + \gamma \int_\Omega \frac{\eps}{2} |\nabla \vp|^2 + \frac{1}{\eps} 
\Psi(\vp) \dx,
\end{align*}
where $\Uad = \{ f \in H^1(\Omega) \, : \, f \in [-1,1] \text{ a.e.~in } \Omega \}$, 
and, as no confusion can arise, we use the short-hand notations 
$\S(\vp)$ and $\Psi(\vp)$ for the functions $\S(\bph)$ and $\Psi(\bph)$ 
evaluated at $\bph = (\frac{1}{2}(1-\vp), \frac{1}{2}(1+\vp))$.
On recalling \ref{ass:potential}, we hence assume that
\begin{equation} \label{eq:obs}
\Psi(s) = \widetilde\Psi(s) + I_{[-1,1]}(s) \text{ for }  s \in \RRR,
\end{equation}
for a $\widetilde\Psi \in C^{1,1}(\RRR)$.
To simplify the calculations, we consider $\bU = \hU = \bm{0}$ (homogeneous Dirichlet data), $\ov{\FF}, \hat{\FF} \in H^1(\Omega,\RRR^d)$, $\ov{\gg} \in H^2(\ov \Gamma_N,\RRR^d)$ and $\hat{\gg} \in H^2(\hat \Gamma_N,\RRR^d)$. We now consider deriving an alternative set of optimal conditions for a minimiser $\vp^* \in \Uad$ based on geometric variations.  To this end, we consider the following admissible transformations and their corresponding velocity fields.
\begin{defn}
The space $\Vad$ of admissible velocity fields is defined as the set of all $\VV \in C^0([-\tau, \tau] \times \ov{\Omega}, \RRR^d)$, where $\tau > 0$ is a fixed small constant, such that it holds:
\begin{itemize}
\item $\VV(t,\cdot) \in C^2(\ov{\Omega}, \RRR^d)$, and $\exists C > 0$ such that $\| \VV(\cdot, \y) - \VV(\cdot, \x) \|_{C^0([-\tau,\tau], \RRR^d)} \leq C\|\x- \y\|$ for all $\x, \y \in \ov{\Omega}$;
\item $\VV(t,\x) \cdot \nn(\x) = 0$ for all $\x \in \pd \Omega$;
\item $\VV(t,\x) = \bm{0}$ for all $\x \in \ov{\Gamma}_D \cup \hat{\Gamma}_D$.
\end{itemize} 
Then, the space $\Tad$ of admissible transformations is defined as the set of solutions to the ordinary differential equations
\[
\pd_t T_t(\x) = \VV(t, T_t(\x)), \quad T_0(\x) = \x
\]
with $\VV \in \Vad$.
\end{defn}
Notice that by the second property it holds $\ov{T_t(\Omega)} = \ov{\Omega}$ for all $t \in [-\tau, \tau]$. Let $\VV \in \Vad$ be an admissible velocity field with corresponding transformation $T \in \Tad$.  For $\vp \in \Uad$ we define $\vp^t := \vp \circ T_t^{-1}$, along with the unique solutions $(\bu^t, \hu^t) \in \ovHD \times \hatHD$, where $\bu^t = \S_1^2(\vp^t)$ and $\hu^t = \S(\vp^t)$. 

Setting $(\vp_0, \bu_0, \hu_0) = (\vp, \bu, \hu)$, by following a similar proof to \cite[Lem.~25]{BGHR}, we define for $\tau_0 > 0$ sufficiently small the function ${\bm F}_1: (-\tau_0, \tau_0) \times \ovHD \to (\ovHD)^*$ by
\begin{align*}
{\bm F}_1(t, \uu)[\vv] & = \iO { \big (\ov{\CCC}(\vp) ( \nabla T_t^{-1} \nabla \uu )^{\rm sym} : ( \nabla T_t^{-1} \nabla \vv)^{\rm sym} - ((\ov{\FF} \circ T_t) \cdot \vv) \big ) \det (\nabla T_t) } \\
& \quad -\int_{\ov{\Gamma}_N} ((\ov{\gg} \circ T_t) \cdot \vv \det (\nabla T_t) \| \nabla T_t^{-\top} \nn \| \dH.
\end{align*}
Using a change of variables $y = T_t(\x)$, the relation
\[
\nabla T_t^{-1} \nabla (\tilde{\uu} \circ T_t) = (\nabla \tilde{ \uu} ) \circ T_t \quad \text{ for } \tilde{\uu} : T_t(\Omega) \to \RRR^d,
\]
and also \cite[Prop.~2.47]{Soko} for the boundary transformation, we obtain
\begin{align*}
{\bm F}_1(t, \uu)[\vv] & = \int_{T_t(\Omega)} \ov{\CCC}(\vp \circ T_t^{-1}) \E_y(\uu \circ T_t^{-1}) : \E_y(\vv \circ T_t^{-1}) \dy \\
& \quad - \int_{T_t(\Omega)} \ov{\FF} \cdot (\vv \circ T_t^{-1}) \dy - \int_{T_t(\ov{\Gamma}_N)} \ov{\gg} \cdot (\vv \circ T_t^{-1} )\dH_y,
\end{align*}
where $\E_y(\uu) = \frac{1}{2}(\nabla_y \uu + (\nabla_y \uu)^{\top})$ for $\uu : T_t(\Omega) \to \RRR^d$ and $\dH_y$ denotes the $(d-1)$-dimensional Hausdorff measure related to $\y$. From the properties of the mapping $T_t$, we find that $T_t(\ov{\Gamma}_N) = \ov{\Gamma}_N$ and $\tilde{\vv} := \vv \circ T_t^{-1} \in \ov{H}^1_D(T_t(\Omega), \RRR^d)$.  Hence, from the above identity we observe that 
\[
{\bm F}_1(t, \bu^t \circ T_t)[\vv] = \int_{\Omega} \ov{\CCC}(\vp^t) \E_y(\bu^t) : \E_y(\tilde{\vv}) - \ov{\FF} \cdot \tilde{\vv} \dy - \int_{\ov{\Gamma}_N} \ov{\gg} \cdot \tilde{\vv} \dH_y = 0
\]
for all $\vv \in \ovHD$.  Denoting by $D_\uu {\bm F}_1$ as the partial derivative of ${\bm F}_1$ with respect to its second argument, we find that $D_\uu {\bm F}_1 (0, \uu) : \ovHD \to (\ovHD)^*$ is given by 
\[
\Big (D_\uu {\bm F}_1(0,\uu)[\vv] \Big )[\ww] = \int_\Omega \ov{\CCC}(\vp) \E(\vv) : \E(\ww) \dx 
\]
for all $\vv, \ww \in \ovHD$, where we have used the relation $\nabla T_t^{-1} \vert_{t=0} = \mathbb{I}$ the identity matrix.  As $D_\uu {\bm F}_1(0,\uu)$ is an isomorphism by the Lax--Milgram theorem, the application of the implicit function theorem allows us to infer that the mapping
\[
t \mapsto (\bu^t \circ T_t)
\]
is differentiable at $t = 0$ with derivative $\dot{\bu}[\VV] := \pd_t \vert_{t=0} (\bu^t \circ T_t) \in \ovHD$ being the unique solution to the distributional equation
\[
D_\uu {\bm F}_1(0, \bu) \big [ \dot{\bu}[\VV] \big ]= - \pd_t {\bm F}_1(0,\bu) \text{ in } (\ovHD)^*,
\]
which reads as
\begin{equation}\label{bu:shape}
\begin{aligned}
& \int_\Omega \ov{\CCC}(\vp) \E(\dot{\bu}[\VV]) : \E(\bz) \dx = \int_\Omega  \ov{\CCC}(\vp) (\nabla \VV(0) \nabla \bu)^{\mathrm{sym}} : \E(\bz)  \dx \\
& \quad + \int_\Omega \ov{\CCC}(\vp) \E(\bu) : (\nabla \VV(0) \nabla \bz)^{\mathrm{sym}} - \ov{\CCC}(\vp) \E(\bu) : \E(\bz) \div \VV(0) \dx \\
& \quad +
\int_\Omega \Big ( \nabla \ov{\FF} \VV(0) + \ov{\FF} \div (\VV(0)) \Big ) \cdot \bz \dx \\
& \quad + \int_{\ov{\Gamma}_N} \Big (\nabla \ov{\gg} \VV(0) + \ov{\gg} \Big ( \div(\VV(0)) - \nn \cdot \nabla \VV(0) \nn \Big ) \Big ) \cdot \bz  \dH
\end{aligned}
\end{equation}
for all $\bz \in \ovHD$. In the above, we have made use of the following relations (see \cite[Lem.~2.31, Prop.~2.36, Lem.~2.49]{Soko})
\begin{align*}
\pd_t \nabla T_t \vert_{t=0} = \nabla \VV(0),  \quad \pd_t \nabla T_t^{-1} \vert_{t=0} = - \nabla \VV(0), \\
\pd_t \det \nabla T_t \vert_{t=0} = \div \VV(0), \quad \pd_t (f \circ T_t) \vert_{t=0} = \nabla f \cdot \VV(0), \\
 \pd_t (\det (\nabla T_t) \| (\nabla T_t)^{-1} \cdot \nn \|) \vert_{t=0} = \div \VV(0)  - \nn \cdot \nabla \VV(0) \nn.
\end{align*}
Furthermore, substituting $\bz = \dot{\bu}[\VV]$ into \eqref{bu:shape}, by means of Korn's inequality and the smoothness of $\VV(0)$, we obtain the estimate
\begin{align}\label{bu:shape:est}
\norma{\dot{\bu}[\VV]}_{H^1(\Omega)} \leq C \Big ( \| \bu \|_{H^1(\Omega)} + \| \ov{\FF} \|_{H^1(\Omega)} + \| \ov{\gg} \|_{H^2(\ov \Gamma_N)} \Big ).
\end{align}
Via a similar procedure, for a small $\tau_0>0$, we consider ${\bm F}_2: (-\tau_0, \tau_0) \times \hatHD \to (\hatHD)^*$ defined as
\begin{align*}
{\bm F}_2(t,\uu )[\vv] & = \iO { \big (\hat{\CCC}(\vp) ( \nabla T_t^{-1} \nabla \uu )^{\rm sym} : ( \nabla T_t^{-1} \nabla \vv)^{\rm sym} - ((\hat{\FF} \circ T_t) \cdot \vv) \big ) \det (\nabla T_t) } \\
& \quad -\int_{\hat{\Gamma}_N} ((\hat{\gg} \circ T_t) \cdot \vv \det (\nabla T_t) \| \nabla T_t^{-\top} \nn \| \dH \\
& \quad - \iO {\chi(\vp) \hat{\CCC}(\vp) ( \nabla T_t^{-1} \nabla( \bu^t \circ T_t) )^{\rm sym} : ( \nabla T_t^{-1} \nabla \vv)^{\rm sym}  \det (\nabla T_t) }.
\end{align*}
Then, by a change of variables $\y = T_t(\x)$, we find that 
\begin{align*}
{\bm F}_2(t, \hu^t \circ T_t)[\vv] & = \int_{\Omega} \hat{\CCC}(\vp^t) \E_y(\hu^t) : \E_y(\tilde{\vv}) - \hat{\FF} \cdot \tilde{\vv} \dy - \int_{\hat{\Gamma}_N} \hat{\gg} \cdot \tilde{\vv} \dH_y \\
& \quad - \int_\Omega \chi(\vp^t) \hat{\CCC}(\vp^t) \E_y(\bu^t) : \E_y(\tilde{\vv}) \dy = 0
\end{align*}
where $\tilde{\vv} = \vv \circ T_t^{-1} \in \hat{H}^1_D(T_t(\Omega), \RRR^d)$. Since $D_\uu {\bm F}_2(0,\uu): \hatHD \to (\hatHD)^*$ given by
\[
\Big (D_\uu {\bm F}_2(0,\uu)[\vv] \Big )[\ww] = \int_\Omega \hat{\CCC}(\vp) \E(\vv) : \E(\ww) \dx 
\quad \text{ for all } \vv,\ww \in \hat{H}_D^1(\Omega, \RRR^d)
\]
is an isomorphism, by the implicit function theorem we infer that the mapping
\[
t \mapsto  (\hu^t \circ T_t)
\]
is differentiable at $t = 0$ with derivative $\dot{\hu}[\VV] := \pd_t \vert_{t=0} (\hu^t \circ T_t) \in \hatHD$ being the unique solution to the distributional equation
\[
D_\uu {\bm F}_2(0, \hu) \big [ \dot{\hu}[\VV] \big ] = - \pd_t {\bm F}_2(0,\hu) \text{ in } (\hatHD)^*,
\]
which reads as
\begin{equation}\label{hu:shape}
\begin{aligned}
&  \int_\Omega \hat{\CCC}(\vp) \E(\dot{\hu}[\VV]) : \E(\bz) \dx = \int_\Omega  \hat{\CCC}(\vp) (\nabla \VV(0) \nabla \hu - \chi(\vp) \nabla \VV(0) \nabla \bu)^{\mathrm{sym}} : \E(\bz)  \dx \\
& \quad + \int_\Omega \hat{\CCC}(\vp) (\E(\hu) - \chi(\vp) \E(\bu)) : (\nabla \VV(0) \nabla \bz)^{\mathrm{sym}} + \chi(\vp) \hat{\CCC}(\vp) \E(\dot{\bu}[\VV]) : \E(\bz) \dx \\
& \quad -\int_\Omega \hat{\CCC}(\vp) (\E(\hu)  - \chi(\vp) \E(\bu)): \E(\bz) \div \VV(0) \dx \\
& \quad +
\int_\Omega \Big ( \nabla \hat{\FF} \VV(0) + \hat{\FF} \div (\VV(0)) \Big ) \cdot \bz \dx \\
& \quad + \int_{\hat{\Gamma}_N} \Big (\nabla \hat{\gg} \VV(0) + \hat{\gg} \Big ( \div(\VV(0)) - \nn \cdot \nabla \VV(0) \nn \Big ) \Big ) \cdot \bz  \dH 
\end{aligned}
\end{equation}
for all $\bz \in \hatHD$. Substituting $\bz = \dot{\hu}[\VV]$ into \eqref{hu:shape}, then using \eqref{bu:shape:est} and Korn's inequality, we obtain the estimate
\begin{equation}\label{hu:shape:est}
\begin{aligned}
\norma{\dot{\hu}[\VV]}_{H^1(\Omega)} & \leq C \Big ( \| \hu \|_{H^1(\Omega)} + \| \bu \|_{H^1(\Omega)}  \Big ) \\
& \quad + C \Big ( \| \ov{\FF} \|_{H^1(\Omega)} + \| \ov{\gg} \|_{H^2(\ov \Gamma_N)} + \| \hat{\FF} \|_{H^1(\Omega)} + \| \hat{\gg} \|_{H^2(\hat \Gamma_N)} \Big ).
\end{aligned}
\end{equation}
The next result details an optimality condition for minimisers $\vp^\eps$ to $(\bm P^\eps)$ obtained via geometric variations.
\begin{thm}\label{thm:opt:variation}
Assume \ref{ass:dom}--\ref{ass:utar}, and additionally suppose that $\ov{\FF}, \hat{\FF} \in H^1(\Omega, \RRR^d)$, $\bU = \hU = \bm{0}, \ov{\gg} \in H^2(\ov \Gamma_N, \RRR^d)$, $\hat{\gg} \in H^2(\hat \Gamma_N, \RRR^d)$  and $\uu^{\rm tar} \in H^2(\Gamma^{\rm tar}, \RRR^d)$ hold. Let $\vp^\eps \in \Uad$ be a minimiser to $(\bm P^\eps)$, with corresponding solutions $\bu^\eps = \S_1^2(\vp^\eps)$ and $\hu^\eps = \S(\vp^\eps)$.  For every admissible velocity $\VV \in \Vad$, let $(\dot{\bu^\eps}[\VV], \dot{\hu^\eps}[\VV]) \in \ovHD \times \hatHD$ denote the unique solutions to \eqref{bu:shape} and \eqref{hu:shape} with $(\vp,\bu, \hu) = (\vp^\eps, \bu^\eps, \hu^\eps)$. Then, for all $\VV \in \Vad$, it holds that 
\begin{equation}\label{opt:variation}
\begin{aligned}
0 & = \frac{1}{2}\int_{\Gamma^{\rm tar}}  W(\hu^\eps - \uu^{\rm tar}) \cdot (\hu^\eps - \uu^{\rm tar}) \Big ( \div \VV(0)  - \nn \cdot \nabla \VV(0) \nn \Big )\dH \\
& \quad + \int_{\Gamma^{\rm tar}} W(\hu^\eps - \uu^{\rm tar}) \cdot (\dot{\hu^\eps}[\VV] - (\nabla \uu^{\rm tar}) \VV(0)) \dH \\
& \quad + \gamma \iO { \Big ( \frac{\eps}{2} |\nabla \vp^\eps|^2 + \frac{1}{\eps} \Psi(\vp^\eps) \Big ) \div \VV(0) - \eps \nabla \vp^\eps \cdot \nabla \VV(0) \nabla \vp^\eps}.
\end{aligned}
\end{equation}
\end{thm}

\begin{remark}
With more regularity, it is possible to relate \eqref{opt:variation} to the optimality condition \eqref{foc:final}, see \cite{BGHR,Hecht} for more details.
\end{remark}

\begin{proof}
For any $\VV \in \Vad$, let $T \in \Tad$ be the associated transformation and consider the scalar function $g(t) := \Jred^\eps(\vp^\eps \circ T_t^{-1})$ for $t \in (-\tau_0, \tau_0)$, where $\tau_0$ is sufficiently small.  As $\vp^\eps$ is a minimiser of $\Jred^\eps$, we have 
\[
g'(0) = \frac{d}{dt} \Jred^\eps(\vp^\eps \circ T_t^{-1}) \vert_{t=0} = 0.
\]
The directional derivative
\[
\frac{d}{dt} E_\eps(\vp^\eps \circ T_t^{-1}) \vert_{t=0} =  \iO { \Big ( \frac{\eps}{2} |\nabla \vp^\eps|^2 + \frac{1}{\eps} \Psi(\vp^\eps) \Big ) \div \VV(0) - \eps \nabla \vp^\eps \cdot \nabla \VV(0) \nabla \vp^\eps} 
\]
can be obtained as in \cite[Lem.~7.5]{Hecht}.  Denoting by $\hu^\eps(t) = \S(\vp^\eps \circ T_t^{-1})$, then the derivative of $G(\vp^\eps \circ T_t^{-1})$ can be obtained by a standard change of variables:
\begin{align*}
& \frac{d}{dt} G(\vp^\eps \circ T_t^{-1}) \\
&\quad  = \frac{1}{2} \frac{d}{dt} \int_{\Gamma^{\rm tar}} W((\hu^\eps(t) - \uu^{\rm tar}) \circ T_t) \cdot ((\hu^\eps(t) - \uu^{\rm tar}) \circ T_t) \det(\nabla T_t) \| \nabla T_t^{-\top} \nn \| \dH \\
& \quad = \int_{\Gamma^{\rm tar}} W(\hu^{\eps} - \uu^{\rm tar}) \cdot \frac{d}{dt} (\hu^\eps(t) \circ T_t - \uu^{\rm tar} \circ T_t) \vert_{t=0} \dH \\
& \qquad + \frac{1}{2} \int_{\Gamma^{\rm tar}} W(\hu^{\eps} - \uu^{\rm tar}) \cdot (\hu^{\eps} - \uu^{\rm tar}) \frac{d}{dt} (\det(\nabla T_t) \| \nabla T_t^{-\top} \nn \|) \dH \\
& \quad = \int_{\Gamma^{\rm tar}} W(\hu^\eps - \uu^{\rm tar}) \cdot (\dot{\hu^\eps}[\VV] - (\nabla \uu^{\rm tar}) \VV(0)) \dH  \\
 & \qquad + \frac{1}{2} \int_{\Gamma^{\rm tar}}  W(\hu^\eps - \uu^{\rm tar}) \cdot (\hu^\eps - \uu^{\rm tar}) \Big ( \div \VV(0)  - \nn \cdot \nabla \VV(0) \nn \Big )\dH
\end{align*}
leading to \eqref{opt:variation}.
\end{proof}

The convergence of \eqref{opt:variation} to the sharp interface limit $\eps \to 0$ is formulated as follows.

\begin{thm}
Suppose the hypotheses of Theorem \ref{thm:opt:variation} hold, and let $\vp^\eps \in \Uad$ be a minimiser to $(\bf P^\eps)$.  For any $\VV \in \Vad$ with corresponding transformation $T \in \Tad$, there exists a non-relabelled subsequence $\eps \to 0$ such that 
\begin{align*}
\vp^\eps & \to \vp_0 \text{ in } L^1(\Omega), \quad \Jred^\eps(\vp^\eps) \to \Jred^0(\vp_0) \text{ in } \RRR, \\
\dot{\bu^\eps}[\VV] & \rightharpoonup \dot{\bu}_0[\VV] \text{ in } \ovHD, \quad \dot{\hu^\eps}[\VV]  \rightharpoonup \dot{\hu}_0[\VV] \text{ in } \hatHD, 
\end{align*}
where $\vp_0 \in \BV(\Omega, \{-1,1\})$ is a minimiser to the reduced functional $\Jred^0(\vp) = G(\vp) + \gamma TV(\vp)$, where the total variation $TV(\vp)$ for $\vp \in \BV(\Omega)$ is defined as
\begin{align*}
TV(\vp) = \sup \Big \{ \int_\Omega \vp \div \bm{\phi} \dx \text{ s.t. } \bm{\phi} \in C^1_0(\Omega, \RRR^d), \, \| \bm{\phi} \|_{L^\infty(\Omega)} \leq 1 \Big \}.
\end{align*}
Furthermore, $\dot{\bu}_0[\VV]$ and $\dot{\hu}_0[\VV]$ satisfy \eqref{bu:shape} and \eqref{hu:shape}, respectively, with $(\vp, \bu, \hu)$ replaced by $(\vp_0, \bu_0 = \S_1^2(\vp_0), \hu_0 = \S_2(\vp_0))$. Lastly, it holds that 
\begin{align}\label{opt:cond:conv}
\frac{d}{dt} \Jred^\eps(\vp^\eps \circ T_t^{-1}) \vert_{t=0} \to \frac{d}{dt} \Jred^0(\vp_0 \circ T_t^{-1}) \vert_{t=0} \text{ in } \RRR, 
\end{align}
where
\begin{equation}\label{J0red:opt}
\begin{aligned}
& \frac{d}{dt} \Jred^0(\vp_0 \circ T_t^{-1}) \vert_{t=0} \\
& \quad  = \int_{\Gamma^{\rm tar}} W(\hu_0 - \uu^{\rm tar}) \cdot (\dot{\hu}_0[\VV] - (\nabla \uu^{\rm tar}) \VV(0)) \dH  \\
& \qquad + \frac{1}{2} \int_{\Gamma^{\rm tar}}  W(\hu_0 - \uu^{\rm tar}) \cdot (\hu_0 - \uu^{\rm tar}) \Big ( \div \VV(0)  - \nn \cdot \nabla \VV(0) \nn \Big )\dH \\
& \qquad +\gamma  \int_\Omega \Big ( \div \VV(0) - \mu \cdot \nabla \VV(0) \mu \Big ) d | \mathrm{D} \bigchi_{\{ \vp_0 = 1 \}}|,
\end{aligned}
\end{equation}
with $\mu = \frac{\mathrm{D} \bigchi_{\{ \vp_0 = 1 \}}}{|\mathrm{D} \bigchi_{\{ \vp_0 = 1 \}}|}$ as the generalised unit normal on the set $\{\vp_0 = 1\}$.
\end{thm}

\begin{remark}
With more regularity, it is possible to relate \eqref{J0red:opt} with the solvability condition \eqref{SI:Phi1:b} in the two-phase setting, see \cite{BGHR,Hecht} for more details.
\end{remark}

\begin{proof}
The first two assertions on the convergence of $\vp^\eps$ and $\Jred^\eps(\vp^\eps)$ come from Lemma \ref{lem:SI:Gamma}. Consequently, by the calculations in the proof of \cite[Thm.~4.2]{Garcke} we infer the convergence, as $\eps \to 0$,
\begin{align*}
& \iO{\Big ( \frac{\eps}{2} |\nabla \vp^\eps|^2 + \frac{1}{\eps} \Psi(\vp^\eps) \Big ) \div \VV(0) - \eps \nabla \vp^\eps \cdot \nabla \VV(0) \nabla \vp^\eps}  \\
& \quad \to  \int_\Omega \Big ( \div \VV(0) - \mu \cdot \nabla \VV(0) \mu \Big ) d | \mathrm{D} \bigchi_{\{ \vp_0 = 1 \}}|.
\end{align*}
Next, from \eqref{bu:shape} and \eqref{hu:shape}, we see that $\dot{\bu^\eps}[\VV]$ and $\dot{\hu^\eps}[\VV]$ satisfy
\begin{align*}
\iO { \ov{\CCC}(\vp^\eps) \E(\dot{\bu^\eps}[\VV]) : \E(\bz) } = \mathcal{F}_{\vp^\eps, \bu^\eps}(\bz), \quad \iO{ \hat{\CCC}(\vp^\eps) \E(\dot{\hu^\eps}[\VV]) : \E(\bz)} = \mathcal{F}_{\vp^\eps, \bu^\eps, \hu^\eps}(\bz),
\end{align*}
where $\mathcal{F}_{\vp^\eps, \bu^\eps}(\bz)$ and $\mathcal{F}_{\vp^\eps, \bu^\eps, \hu^\eps}(\bz)$ denotes the right-hand sides of \eqref{bu:shape} and \eqref{hu:shape}, respectively.  Thanks to Corollary \ref{COR:cts}, along a non-relabelled subsequence, $\bu^\eps \to \bu_0 \in \ovHD$ and $\hu^\eps \to \hu_0 \in \hatHD$ strongly as $\eps \to 0$.  Hence, together with the dominated convergence theorem, it is clear that, as $\eps \to 0$,
\[
 \mathcal{F}_{\vp^\eps, \bu^\eps}(\bz) \to \mathcal{F}_{\vp_0, \bu_0}(\bz), \quad \mathcal{F}_{\vp^\eps, \bu^\eps, \hu^\eps}(\bz) \to \mathcal{F}_{\vp_0, \bu_0, \hu_0}(\bz).
\]
On the other hand, we infer from \eqref{bu:shape:est} and \eqref{hu:shape:est} that $\dot{\bu^\eps}[\VV]$ and $\dot{\hu^\eps}[\VV]$ are uniformly bounded in $\ovHD$ and $\hatHD$, which then implies the existence of limit functions $\dot{\bu}_0[\VV] \in \ovHD$ and $\dot{\hu}_0[\VV] \in \hatHD$, where
\begin{align*}
\iO { \ov{\CCC}(\vp^\eps) \E(\dot{\bu^\eps}[\VV]) : \E(\bz) } &  \to  \iO{ \ov{\CCC}(\vp_0) \E( \dot{\bu}_0[\VV]) : \E(\bz)  }, \\
\iO { \hat{\CCC}(\vp^\eps) \E(\dot{\hu^\eps}[\VV]) : \E(\bz) } &  \to  \iO{ \hat{\CCC}(\vp_0) \E( \dot{\hu}_0[\VV]) : \E(\bz)  }.
\end{align*}
Lastly, using the compactness of the boundary-trace operator, we obtain, as $\eps \to 0$,
\begin{align*}
&\int_{\Gamma^{\rm tar}} W(\hu^\eps - \uu^{\rm tar}) \cdot (\dot{\hu^\eps}[\VV] - (\nabla \uu^{\rm tar}) \VV(0)) \dH  \\
 & \qquad + \frac{1}{2} \int_{\Gamma^{\rm tar}}  W(\hu^\eps - \uu^{\rm tar}) \cdot (\hu^\eps - \uu^{\rm tar}) \Big ( \div \VV(0)  - \nn \cdot \nabla \VV(0) \nn \Big )\dH \\
 & \quad \to \int_{\Gamma^{\rm tar}} W(\hu_0 - \uu^{\rm tar}) \cdot (\dot{\hu}_0[\VV] - (\nabla \uu^{\rm tar}) \VV(0)) \dH \\
& \qquad + \frac{1}{2}\int_{\Gamma^{\rm tar}}W(\hu_0 - \uu^{\rm tar}) \cdot (\hu_0 - \uu^{\rm tar}) \Big ( \div \VV(0)  - \nn \cdot \nabla \VV(0) \nn \Big )\dH.
\end{align*}
This shows \eqref{opt:cond:conv} and completes the proof.
\end{proof}

\section{Numerical simulations}
\label{SEC:NUM}

In this section we present the finite element discretisation and showcase several numerical simulations in two and three dimensions for the two-phase case. Namely, we have $L = 2$ and we consider the optimal distribution of a single type of active material within a passive material.

\subsection{Finite element discretisation}
We assume that $\Omega$ is a polyhedral domain and
let $\mathcal{T}_{h}$ be a regular triangulation of $\Omega$ into disjoint open simplices. Associated with $\mathcal{T}_h$ are the piecewise linear finite element spaces
\begin{align*}
S^{h} = \left \{\zeta \in C^{0}(\overline\Omega) : \,  \zeta_{\vert_{o}} \in P_{1}(o) \, \forall o \in \mathcal{T}_{h} \right \}
\quad\text{and}\quad \SS^{h} =S^h \times \cdots \times S^h =  [S^{h}]^d,
\end{align*}
where we denote by $P_{1}(o)$ the set of all affine linear functions on $o$,
cf.\ \cite{Ciarlet78}. 
In addition we define 
\begin{equation} \label{eq:Khm}
\V^h = \left\{ \zeta \in S^h : |\zeta| \leq 1  \text{ in } \overline\Omega \right\},
\end{equation}
as well as 
\begin{equation*}
\ov{\SS}^h_D = \left\{ 
\bm{\eta} \in \SS^h : \bm{\eta} = {\bf 0} \text{ on } \ov{\Gamma}_D
\right\}, \quad 
\hat \SS^h_D = \left\{ 
\bm{\eta} \in \SS^h : \bm{\eta} = {\bf 0} \text{ on } \hat \Gamma_D
\right\}.
\end{equation*}
We also let $(\cdot,\cdot)$ denote the $L^{2}$--inner product on $\Omega$,
and let $(\cdot,\cdot)^{h}$ be the usual mass lumped $L^{2}$--inner product on
$\Omega$ associated with $\mathcal{T}_{h}$.
In addition, $\langle \AA,\BB \rangle_{\CCC} = (\CCC \AA,\BB)$ 
for any fourth order tensor $\CCC$ and any matrices $\AA$ and $\BB$.
Finally, $\tau$ denotes a chosen uniform  time step size.

We now introduce finite element approximations of the state equations \eqref{wf:sys:1} and \eqref{wf:sys:2}, adjoint systems \eqref{wf:adj:q} and \eqref{wf:adj:p}, as well as the optimality conditions \eqref{foc:final} with a pseudo-time evolution based on an $L^2$-gradient flow approach. 
In particular, we consider the obstacle potential \eqref{eq:obs} 
with $\widetilde\Psi(s) = \frac{1}{2}(1-s^2)$,
which leads to a variational inequality.

The fully discrete numerical scheme is formulated as follows: Given $\vp_h^{n-1} \in \V^h$, find 
$(\bu_h^n, \hu_h^n, \hat{\qq}_h^n, \bp_h^n,\vp_h^n) \in \ov{\SS}^h_D \times \hat{\SS}^h_D \times \hat{\SS}^h_D \times \ov{\SS}^h_D\times \V^h$ such that
\begin{subequations} \label{eq:FEA}
\begin{align}
&
\langle \E(\bu_h^n),
\E(\bz)\rangle_{ \ov{\CCC}(\vp_h^{n-1})}
= \left ( \ov{\FF}, \bz \right)^h + \int_{\ov{\Gamma}_N} \ov{\gg} \cdot \bz \dH \qquad \forall \bz \in \ov{\SS}^h_D, \label{eq:FEAa}\\
&
\langle \E(\hu_h^n) - \chi(\vp_h^{n-1}) \E(\bu_h^n),
\E(\bz)\rangle_{ \hat{\CCC}(\vp_h^{n-1})}
= \big( \hat{\FF}, \bz \big)^h + \int_{\hat{\Gamma}_N} \hat{\gg} \cdot \bz \dH \qquad \forall \bz \in \hat{\SS}^h_D, \label{eq:FEAb}\\
&
\langle \E(\hat{\qq}_h^n),
\E(\bz)\rangle_{ \hat{\CCC}(\vp_h^{n-1})}
= \int_{\Gamma^{\rm tar}} W(\hu_h^n - \uu^{\rm tar}) \cdot \bz  \dH \qquad \forall \bz \in \hat{\SS}^h_D, \label{eq:FEAc}\\
&
\langle \E(\bp_h^n),
\E(\bz)\rangle_{ \ov{\CCC}(\vp_h^{n-1})}
= \langle \chi(\vp_h^{n-1}) \E(\hat{\qq}_h^n), \E(\bz) \rangle_{\hat{\CCC}(\vp_h^{n-1})} \qquad \forall \bz \in \ov{\SS}^h_D, \label{eq:FEAd}\\
& \left( \tfrac{\eps}{\tau} (\varphi_h^n - \varphi_h^{n-1}) 
- \tfrac{\gamma}{\eps} \varphi_h^n, \zeta - \varphi_h^n \right)^h
+ \gamma \eps (\nabla \varphi_h^n, 
\nabla(\zeta - \varphi_h^n)) \nonumber \\ & \qquad \quad
+ \langle \chi'(\vp_h^{n-1}) \E(\bu_h^n), (\zeta - \varphi_h^n)
\E(\hat{\qq}_h^n) \rangle_{\hat{\CCC}(\vp_h^{n-1})} 
\nonumber \\ & \qquad \quad
-\langle \E(\hu_h^n) - \chi(\vp_h^{n-1}) \E(\bu_h^n), (\zeta - \varphi_h^n)
\E(\hat{\qq}_h^n) \rangle_{ \hat{\CCC}'(\vp_h^{n-1})} \nonumber \\
& \quad \qquad  - \langle \E(\bu_h^n), (\zeta - \varphi_h^n) \E(\bp_h^n)
\rangle_{\ov{\CCC}'(\vp_h^{n-1})}
\geq 0
\qquad \forall \zeta \in \V^h. \label{eq:FEAe}
\end{align}
\end{subequations}
We implemented the scheme \eqref{eq:FEA} with the help of the finite element 
toolbox ALBERTA, see \cite{Alberta}. 
To increase computational efficiency, we employ adaptive meshes, which have a 
finer mesh size $h_{f}$ within the diffuse interfacial regions and a coarser 
mesh size $h_{c}$ away from them, see \cite{voids3d,voids}
for a more detailed description. 

Clearly, we first solve the linear systems \eqref{eq:FEAa} in order to obtain $\bu_h^n$, then \eqref{eq:FEAb} for $\hu_h^n$, then \eqref{eq:FEAc} for $\hat{\qq}_h^n$, then \eqref{eq:FEAd} for $\bp_h^n$, and finally the variational inequality \eqref{eq:FEAe} for $\vp_h^n$. 
In two space dimensions we employ the package LDL, see \cite{Davis05}, together with the sparse 
matrix ordering AMD, see \cite{AmestoyDD04}, in order to solve the linear
systems \eqref{eq:FEAa}--\eqref{eq:FEAd}.
In three space dimensions, on the other hand, we use a preconditioned conjugate gradient
algorithm, with a $W$-cycle multigrid step as preconditioner.
In order to solve the variational inequality \eqref{eq:FEAe} we employ a 
nonlinear multigrid method similar to \cite{Kornhuber96}.

For the computational domain we will choose 
$\overline\Omega = [0, L_1] \times [-\frac12,\frac12]$ in two dimensions, and 
$\overline\Omega = [0, L_1] \times [-\frac12 L_2, \frac12 L_2] \times
[-\frac12,\frac12]$ in three dimensions with positive lengths $L_i$,
$i \in \{1,2\}$, given below. For the physical parameters we loosely
follow the settings in \cite{Maute}. In particular, for the forcings we choose
$\ov{\FF} = \hat{\FF} = {\bf 0}$ throughout, as well as 
$\widehat{\bm{g}} = {\bf 0}$ and
\begin{equation} \label{eq:barg}
\overline{\bm{g}}(\x) = \begin{cases}
g {\bm e}_1 & x_1 = L_1, \\
{\bf 0} & x_1 < L_1,
\end{cases}
\qquad \text{with} \quad g=0.1.
\end{equation}
For the interpolated elasticity tensors we choose
$\overline{\CCC}(s) = 
\frac12 [ (1 + s) \overline{\CCC}_+ + (1 - s) \overline{\CCC}_-]$, where the 
two tensors $\overline{\CCC}_\pm$ are defined through the Young's moduli
$\overline{E}_\pm$ and Poisson ratios $\overline\nu_\pm$ via
\begin{equation} \label{eq:barEnu}
\overline{E}_+ = 3,\ \overline{E}_- = 0.7,\quad  
\overline\nu_+ = \overline\nu_-=0.45.
\end{equation}
Similarly, $\widehat{\CCC}(s) = 
\frac12 [ (1 + s) \widehat{\CCC}_+ + (1 - s) \widehat{\CCC}_-]$, with 
the Young's moduli and Poisson ratios
\begin{equation} \label{eq:hatEnu}
\widehat{E}_+ = 13,\ \widehat{E}_- = 0.6,\quad 
\widehat\nu_+ = \widehat\nu_-=0.45.
\end{equation}
For the fixity function we choose 
\begin{equation} \label{eq:chi}
\chi(s) = \tfrac25(1+s).
\end{equation}

For the visualisation of the progress of the discrete gradient flow
computations, we define the discrete cost functional 
\begin{align} \label{eq:J}
J^h(\varphi_h^h, \hu_h^n) &
= \gamma \left(\frac{\eps}2 |\nabla \varphi_h^n|^2 
+ \frac1{\eps} \psi(\varphi_h^n) , 1 \right) 
+ {\frac12} \int_{\Gamma^{\rm tar}} W(\hu_h^n - \bm{u}^{\rm tar}) \cdot 
(\hu_h^n - \bm{u}^{\rm tar}) \dH
\nonumber \\ &
=: \gamma \mathcal{E}^h(\varphi_h^n) + \mathcal{E}^{\rm h,tar}(\hu_h^n).
\end{align}
As the initial data $\varphi_h^0$ we choose a random mixture with mean
zero. Choosing other initial data, including random mixtures with positive or
negative mean had no visible influence on the obtained optimal shapes.

\subsection{Numerical simulations in two dimensions}
For the target shapes we consider the parabolic profile
\begin{subequations} \label{eq:utarget2d}
\begin{equation} \label{eq:utarget}
\bm{u}^{\rm tar}(x_1,x_2) = u^{\rm tar}(x_1) {\bm e}_2,\quad
u^{\rm tar}(x_1) = c^{\rm tar} (x_1)^2 ,
\end{equation}
with $c^{\rm tar} > 0$, and the cosine profile
\begin{equation} \label{eq:utarget2}
\bm{u}^{\rm tar}(x_1,x_2) =  u^{\rm tar}(x_1) {\bm e}_2,\quad
u^{\rm tar}(x_1) = c^{\rm tar} \Big (1- \cos \frac{k^{\rm tar}\pi x_1}{L_1} \Big )
\end{equation}
\end{subequations}
with $c^{\rm tar}>0$ and $k^{\rm tar}>0$. In Figure~\ref{fig:utarget12kickp} we plot 
two examples for the profiles in \eqref{eq:utarget2d} for the domain length $L_1 = 12$.

In each of the Figures~\ref{fig:2d8pi_61kick_g001_Wid}, \ref{fig:2d8pi_cos61_g001_187}, \ref{fig:2d8pi_121kick_g001_Wid} and \ref{fig:2d8pi_cos121kickp15t_g001_Wid}, we provide visualisations of the numerical solution $\vp_h^n$ at various pseudo-times (black denotes the passive material $\{\vp_h^n = -1\}$ and grey denotes the active material $\{\vp_h^n = 1\}$), and the corresponding displacements $\bu_h^n$ (in red) and $\hu_h^n$ (in green). As mentioned before, each time the gradient flow is started from a random mixture $\varphi_h^0$ with zero mean. We also provide pseudo-time plots of the cost functional $J^h(\vp_h^n, \hu_h^n)$, the proportion of the elastic $\mathcal{E}^{\rm h,tar}(\hu_h^n)$ and interfacial $\gamma \mathcal{E}^h(\varphi_h^n)$ energies, as well as log-plots of the elastic energies. The parameter details are summarised in the Table~\ref{tbl:1}.

\begin{table}[h]
\centering
\begin{tabular}{c|c|c|c|c|c|c}
Figure & Profile & Domain & $W$ & $\Gamma^{\rm tar}$ & $c^{\rm tar}$ & $ k^{\rm tar}$ \\
\hline
\ref{fig:2d8pi_61kick_g001_Wid} & Parabolic \eqref{eq:utarget} & $[0,6] \times [-\frac{1}{2}, \frac{1}{2}]$ & $\Id$ & $\pd \Omega$ & $0.075$ & - \\[1ex]
\ref{fig:2d8pi_cos61_g001_187} & Cosine \eqref{eq:utarget2} & $[0,6] \times [-\frac{1}{2},\frac{1}{2}]$ & $\bm{e}_2 \otimes \bm{e}_2$ & $\pd \Omega$ & $0.25$ & $2$ \\[1ex]
\ref{fig:2d8pi_121kick_g001_Wid} & Parabolic \eqref{eq:utarget} & $[0,12] \times [-\frac{1}{2}, \frac{1}{2}]$ & $\Id$ & $\pd \Omega$ & $0.02$ & -  \\[1ex]
\ref{fig:2d8pi_cos121kickp15t_g001_Wid} & Cosine \eqref{eq:utarget2} & $[0,12] \times [-\frac{1}{2},\frac{1}{2}]$ & $\Id$ & $\pd_{\rm top} \Omega$ & $1$ & $1.5$ 
\end{tabular}
\caption{Parameter details for numerical simulations in 2D.}
\label{tbl:1}
\end{table}

For all the presented simulations we choose the parameters $\eps = \frac{1}{8\pi}$ and $\gamma = 0.01$. In each case the cost functional decreases monotonically, but the proportions of the two energies (elastic vs interfacial) differ from case to case.

The first experiment in Figure~\ref{fig:2d8pi_61kick_g001_Wid} is for the
parabolic profile on the domain $[0,6] \times [-\frac{1}{2}, \frac{1}{2}]$. We
observe that in the optimal distribution of material, the active phase occupies
most of the lower domain. This ensures that in the programmed stage, the
printed active composite is able to attain the desired target shape.

\begin{figure}
\center
\includegraphics[angle=-0,width=0.4\textwidth]{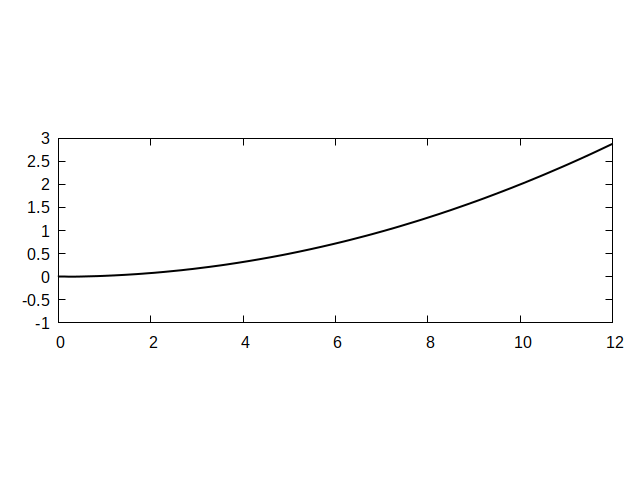}
\includegraphics[angle=-0,width=0.4\textwidth]{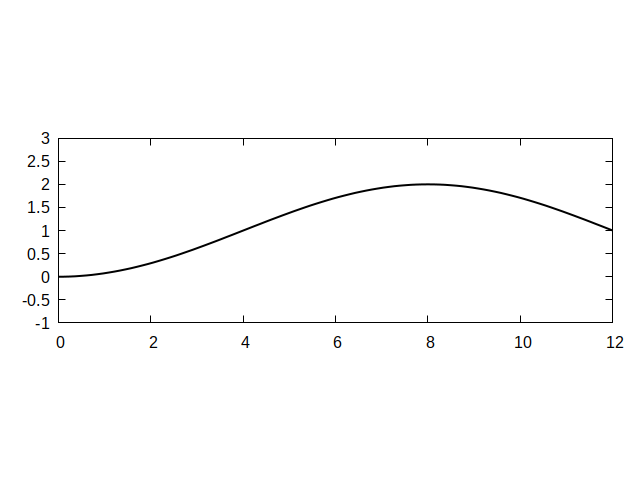}
\caption{Plots of the function $u^{\rm tar}$ in 
\eqref{eq:utarget}, with $c^{\rm tar} = 0.02$, (left) and \eqref{eq:utarget2},
with $c^{\rm tar} = 1$ and $k^{\rm tar}=1.5$, (right), when $L_1 = 12$.
}
\label{fig:utarget12kickp}
\end{figure}%

Not surprisingly, a very different
distribution of material is obtained when changing the
target shape functional to enforce a cosine profile at the programmed stage.
As can be seen from Figure~\ref{fig:2d8pi_61kick_g001_Wid}, the optimal
distribution of the active material is now given by an elongated region that
connects the lower left of the domain with the upper right.
\begin{figure}
\center
\includegraphics[angle=-0,width=0.3\textwidth]{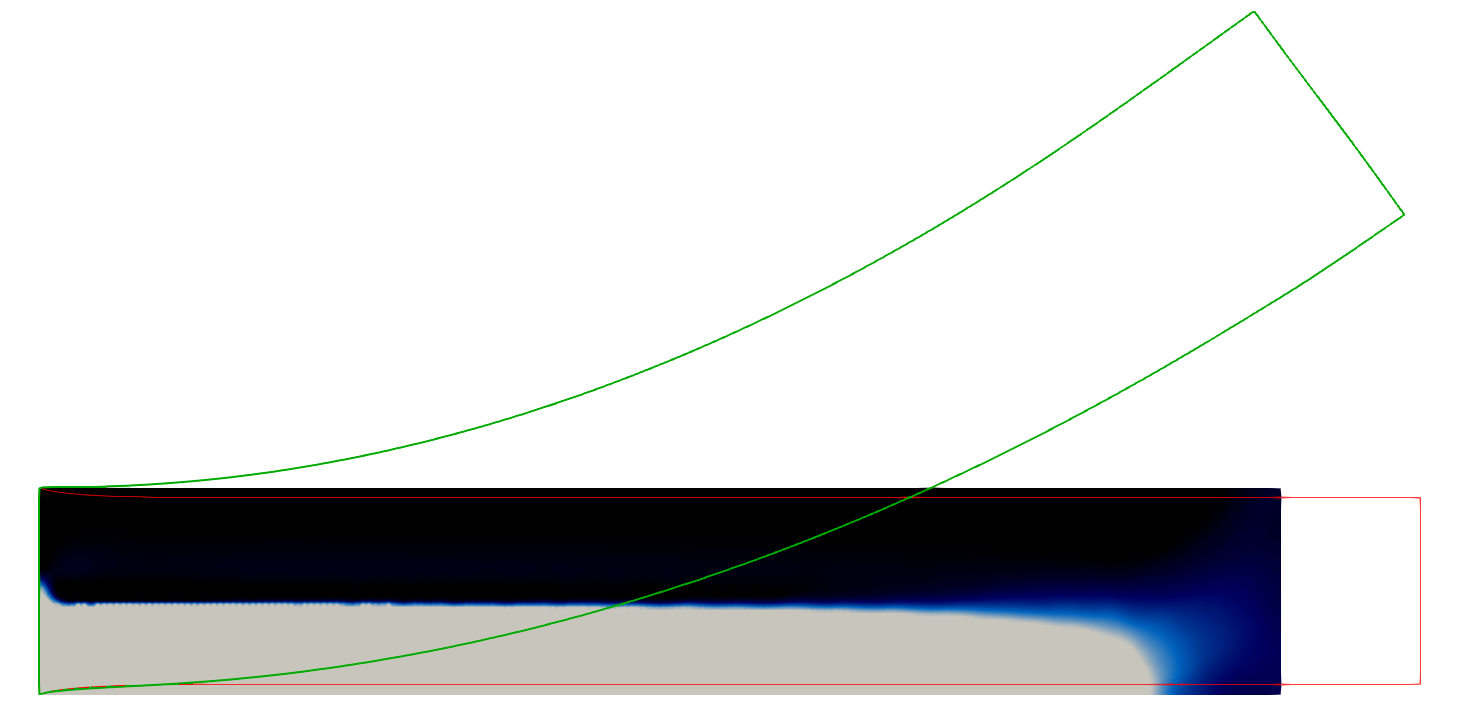}
\includegraphics[angle=-0,width=0.3\textwidth]{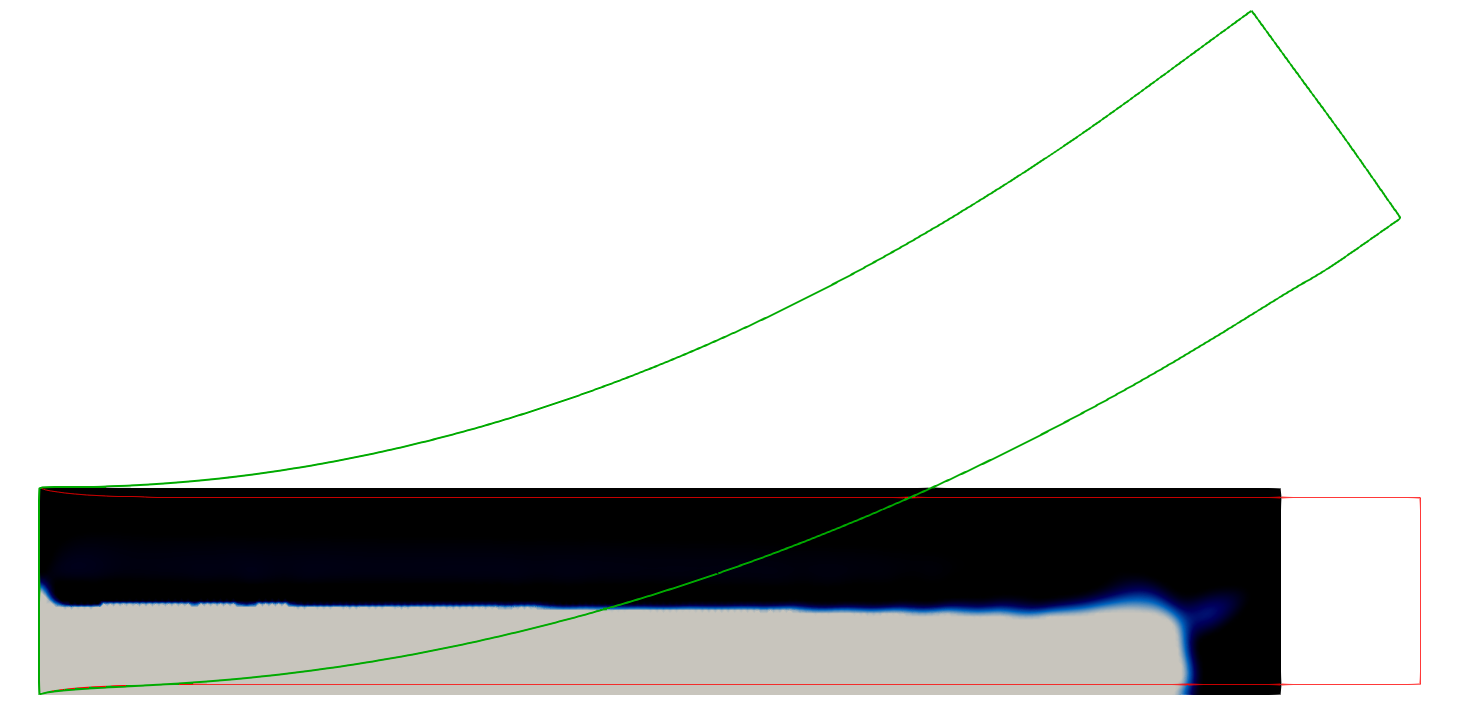}
\includegraphics[angle=-0,width=0.3\textwidth]{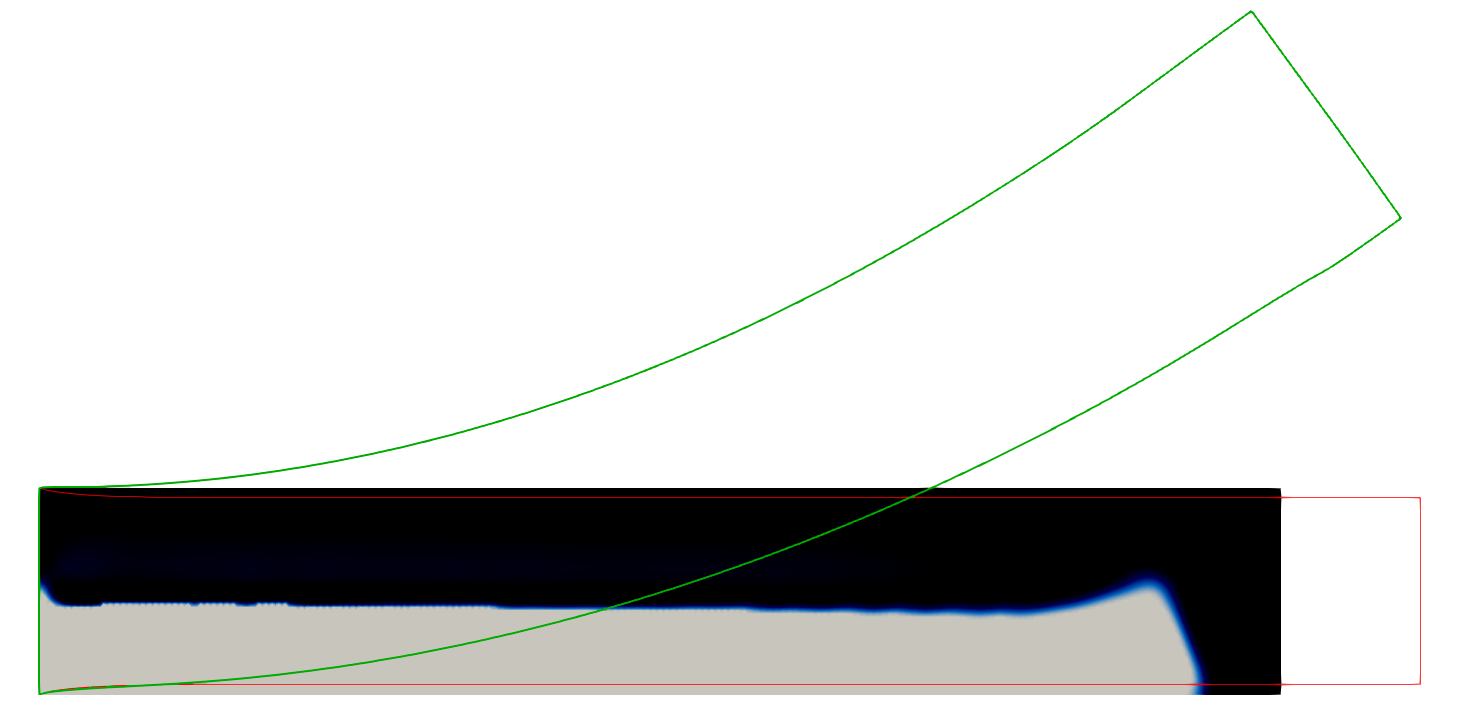} \\[1ex]
\includegraphics[angle=-0,width=0.3\textwidth]{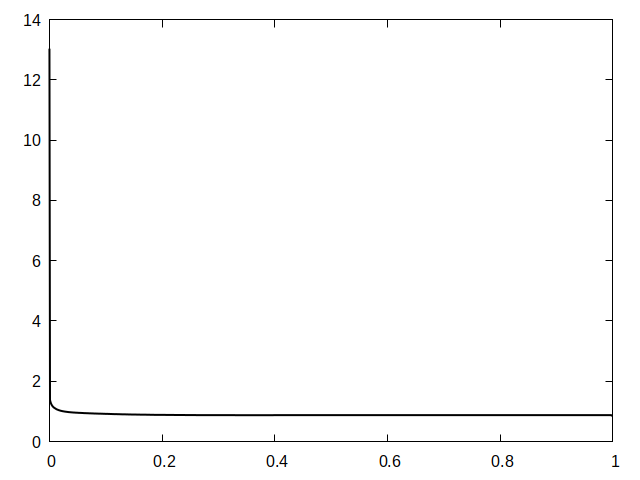}
\includegraphics[angle=-0,width=0.3\textwidth]{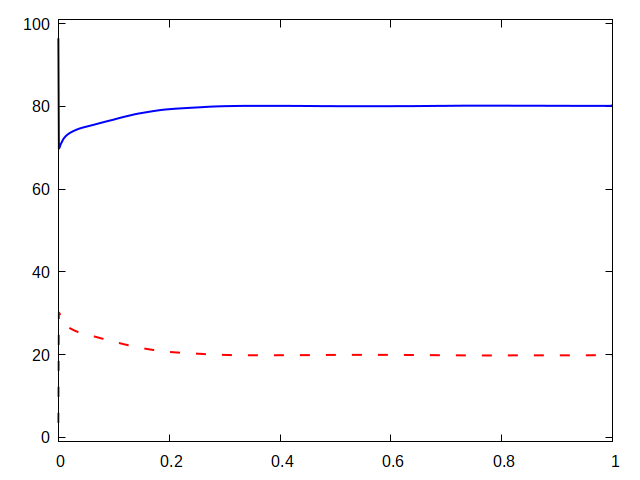}
\includegraphics[angle=-0,width=0.3\textwidth]{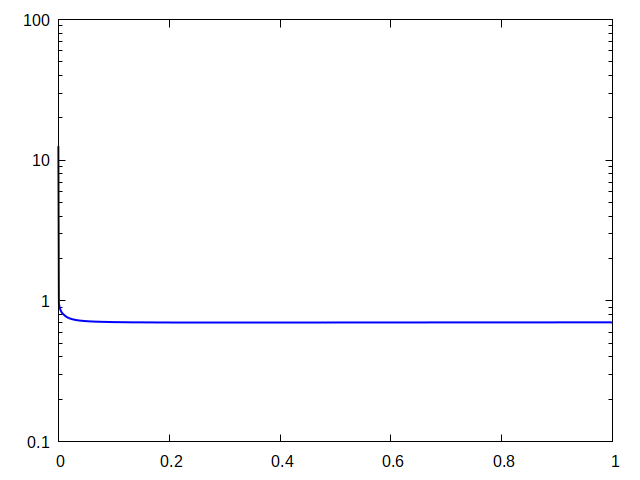}
\caption{%
Computation on $\overline\Omega = [0,6] \times [-\tfrac12, \tfrac12]$
for the target shape \eqref{eq:utarget} with $c^{\rm tar} = 0.075$ and 
$W = \Id$ and $\Gamma^{\rm tar} = \partial\Omega$, 
with $\eps=\frac1{8\pi}$ and $\gamma=0.01$. 
We display $\vp_h^n$ and the 
displacements $\bu_h^n$ (red) and $\hu_h^n$ (green) at pseudo-times 
$t=0.1,\,0.5,\,1$.
Below we show plots of the cost functional $J^h(\vp_h^n, \hu_h^n)$,
the proportion in it of the elastic energy $\mathcal{E}^{\rm h,tar}(\hu_h^n)$
(solid blue) and the interfacial energy $\gamma \mathcal{E}^h(\varphi_h^n)$ 
(dashed red), as well as of $\log_{10} \mathcal{E}^{\rm h,tar}(\hu_h^n)$.
}
\label{fig:2d8pi_61kick_g001_Wid}
\end{figure}%

It turns out that on longer (or thinner) domains, far less active material is
needed to achieve significant deformations at the programmed state. For
example, in Figure~\ref{fig:2d8pi_cos61_g001_187} a miniscule amount of active
material, spread in several connected components arranged at the bottom of the
domain, is sufficient to result in the desired parabolic target shape.
\begin{figure}
\center
\mbox{
\includegraphics[angle=-0,width=0.3\textwidth]{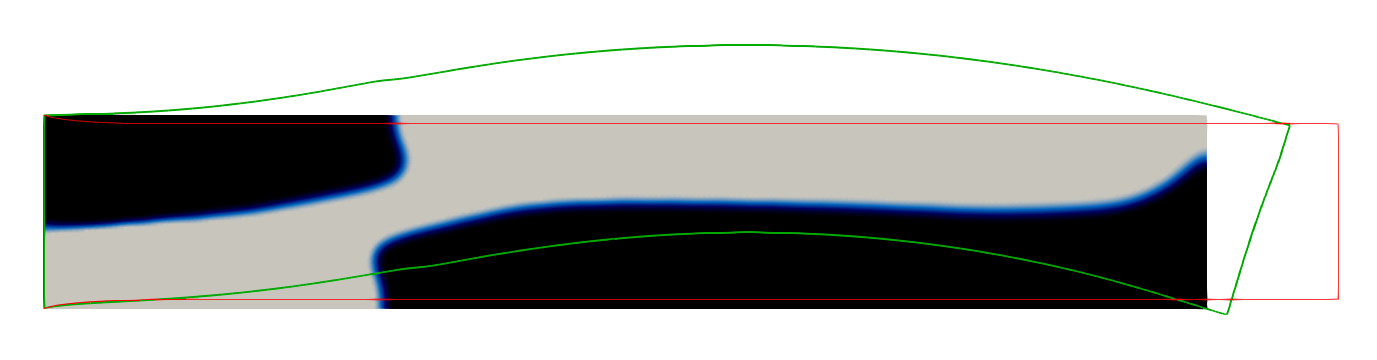}
\includegraphics[angle=-0,width=0.3\textwidth]{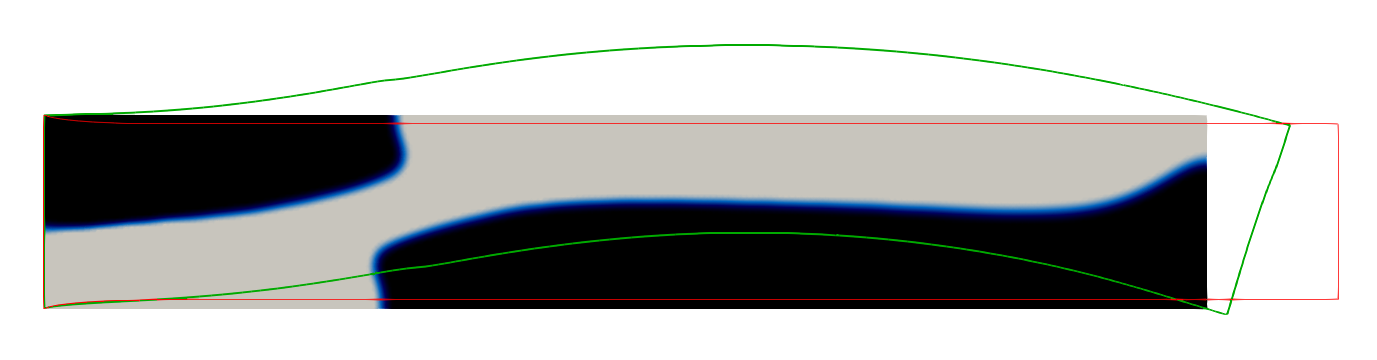}
\includegraphics[angle=-0,width=0.3\textwidth]{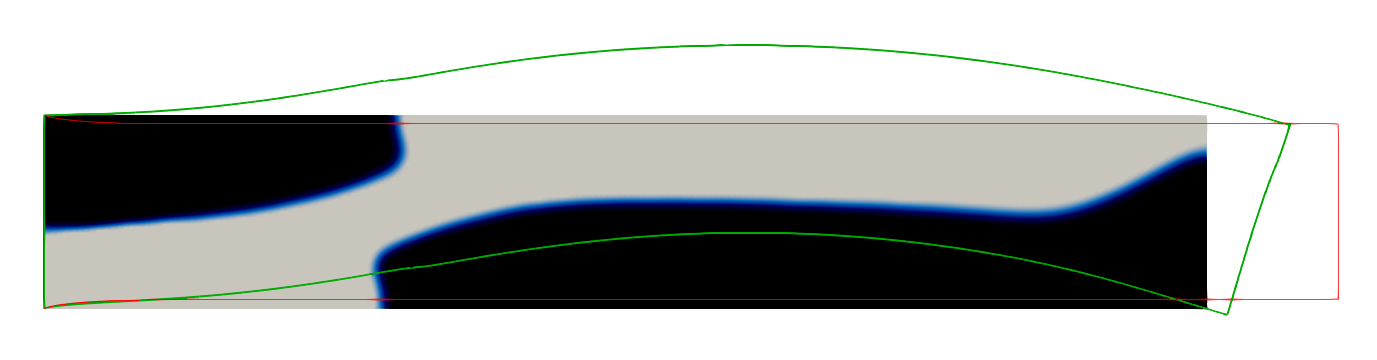}
} \\[1ex]
\includegraphics[angle=-0,width=0.3\textwidth]{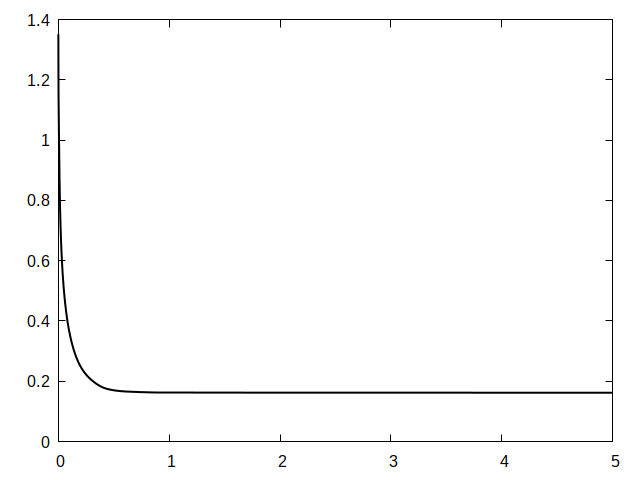}
\includegraphics[angle=-0,width=0.3\textwidth]{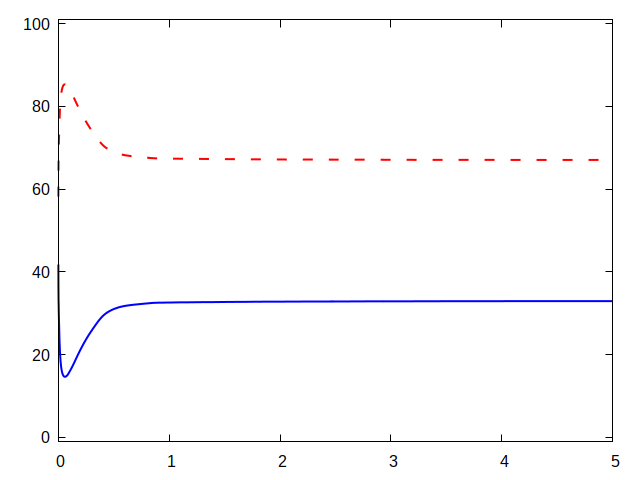}
\includegraphics[angle=-0,width=0.3\textwidth]{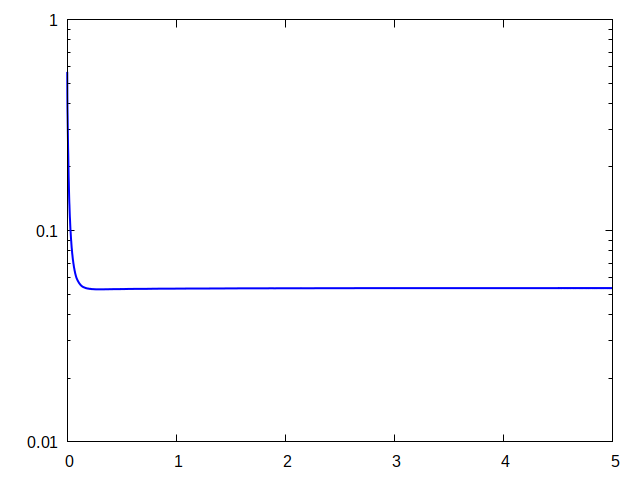}
\caption{%
Computation on $\overline\Omega = [0,6] \times [-\tfrac12, \tfrac12]$
for the target shape \eqref{eq:utarget2} with $c^{\rm tar} = 0.25$, $k^{\rm tar} = 2$ 
and $W = \bm{e}_2 \otimes \bm{e}_2$ and $\Gamma^{\rm tar} = \partial\Omega$, 
with $\eps=\frac1{8\pi}$ and $\gamma=0.01$. 
We display $\vp_h^n$ and the 
displacements $\bu_h^n$ (red) and $\hu_h^n$ (green) at pseudo-times 
$t=1,\,2,\,5$.
Below we show plots of the cost functional $J^h(\vp_h^n, \hu_h^n)$,
the proportion in it of the elastic energy $\mathcal{E}^{\rm h,tar}(\hu_h^n)$
(solid blue) and the interfacial energy $\gamma \mathcal{E}^h(\varphi_h^n)$ 
(dashed red), as well as of $\log_{10} \mathcal{E}^{\rm h,tar}(\hu_h^n)$.
}
\label{fig:2d8pi_cos61_g001_187}
\end{figure}%

Similarly, we observe in Figure~\ref{fig:2d8pi_121kick_g001_Wid} that three
strategically placed small amounts of active material guarantee a
cosine profile at the programmed stage for the printed active composite.
\begin{figure}
\center
\includegraphics[angle=-0,width=0.3\textwidth]{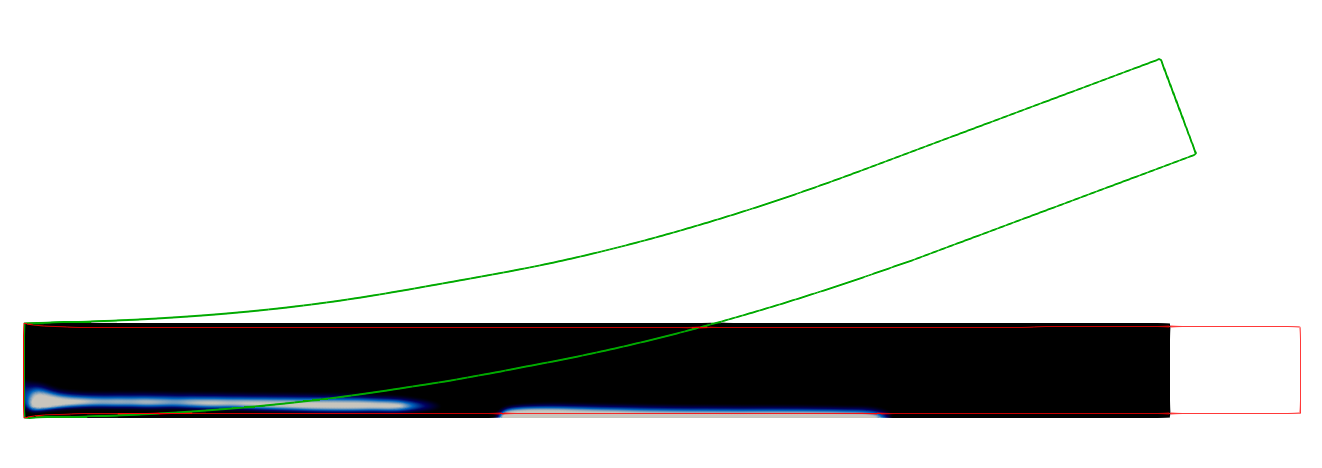}
\includegraphics[angle=-0,width=0.3\textwidth]{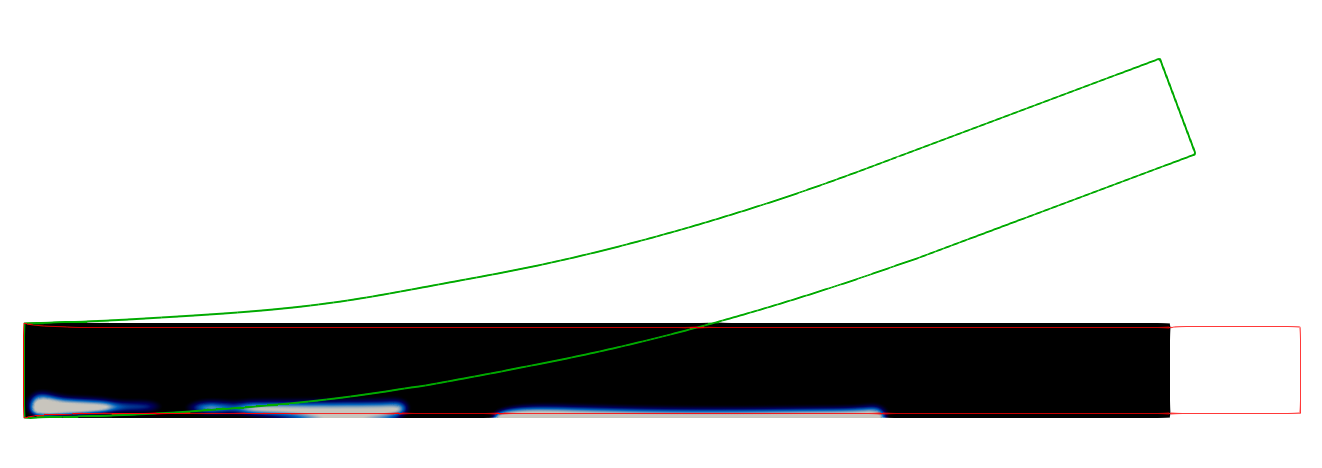}
\includegraphics[angle=-0,width=0.3\textwidth]{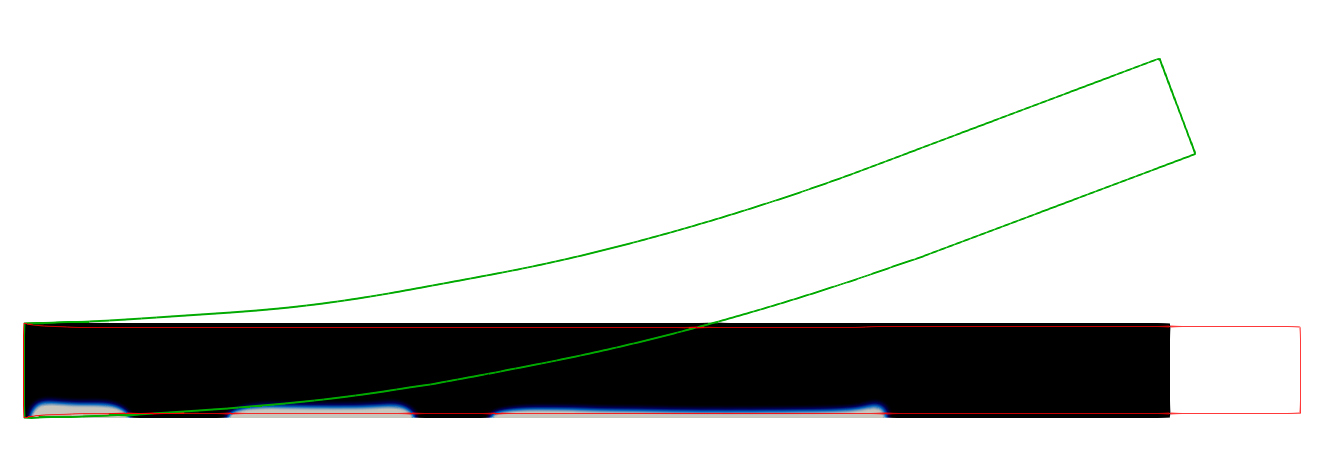} \\[1ex]
\includegraphics[angle=-0,width=0.3\textwidth]{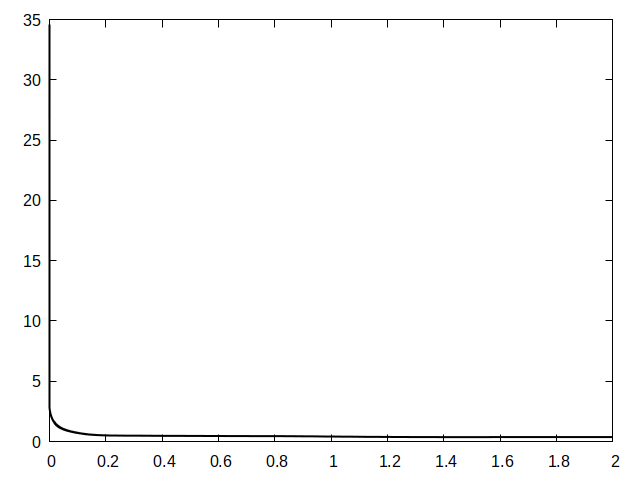}
\includegraphics[angle=-0,width=0.3\textwidth]{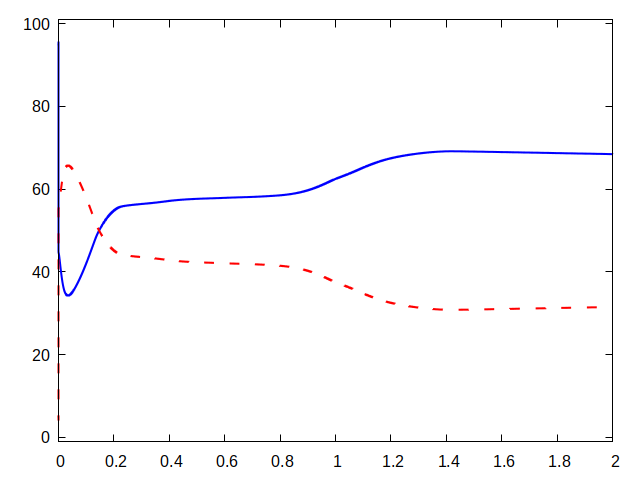}
\includegraphics[angle=-0,width=0.3\textwidth]{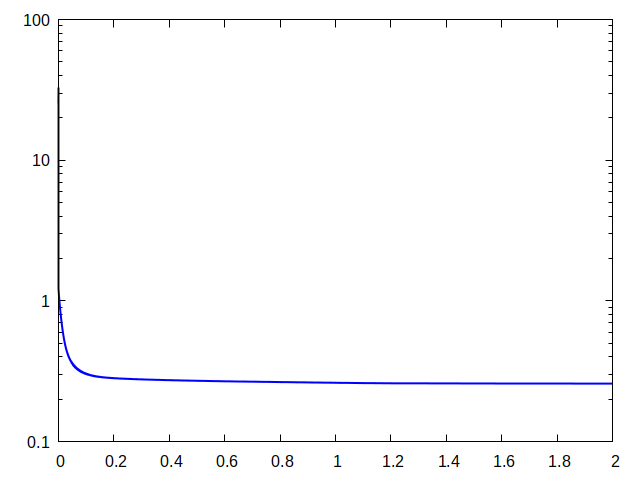}
\caption{%
Computation on $\overline\Omega = [0,12] \times [-\tfrac12, \tfrac12]$
for the target shape \eqref{eq:utarget} with $c^{\rm tar} = 0.02$ and 
$W = \Id$ and $\Gamma^{\rm tar} = \partial\Omega$, 
with $\eps=\frac1{8\pi}$ and $\gamma=0.01$. 
We display $\vp_h^n$ and the 
displacements $\bu_h^n$ (red) and $\hu_h^n$ (green) at pseudo-times 
$t=0.5,\,1,\,2$.
Below we show plots of the cost functional $J^h(\vp_h^n, \hu_h^n)$,
the proportion in it of the elastic energy $\mathcal{E}^{\rm h,tar}(\hu_h^n)$
(solid blue) and the interfacial energy $\gamma \mathcal{E}^h(\varphi_h^n)$ 
(dashed red), as well as of $\log_{10} \mathcal{E}^{\rm h,tar}(\hu_h^n)$.
}
\label{fig:2d8pi_121kick_g001_Wid}
\end{figure}

\begin{figure}
\center
\includegraphics[angle=-0,width=0.3\textwidth]{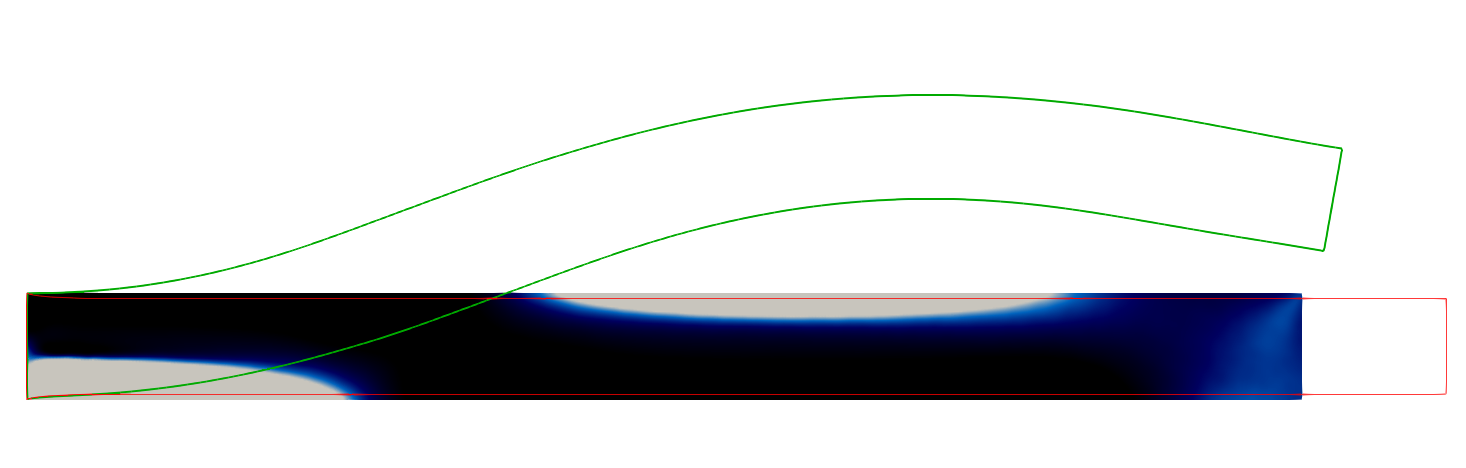}
\includegraphics[angle=-0,width=0.3\textwidth]{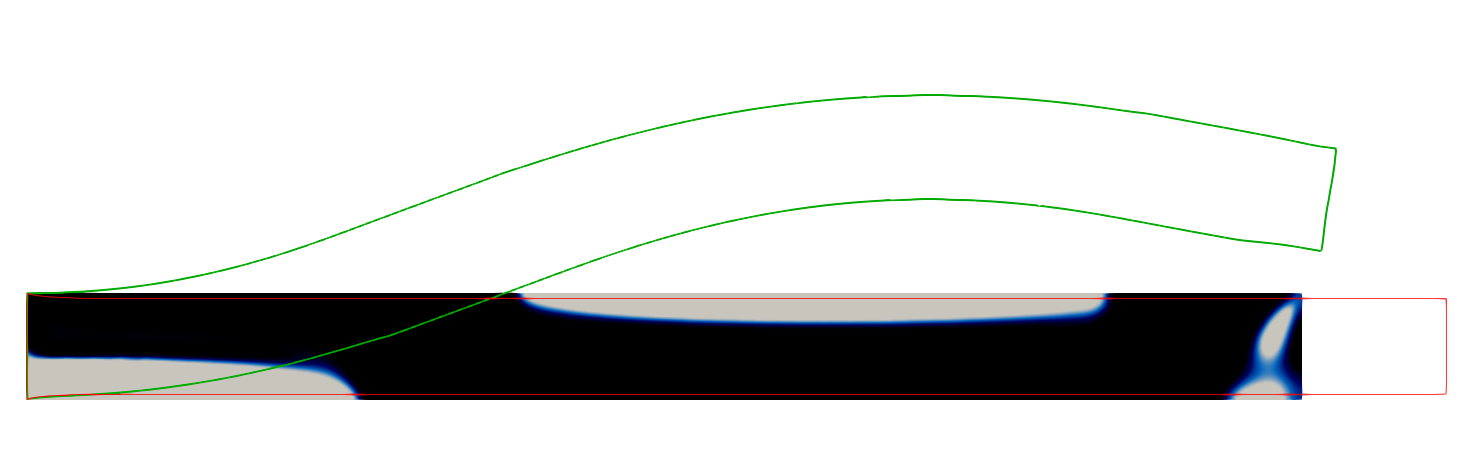}
\includegraphics[angle=-0,width=0.3\textwidth]{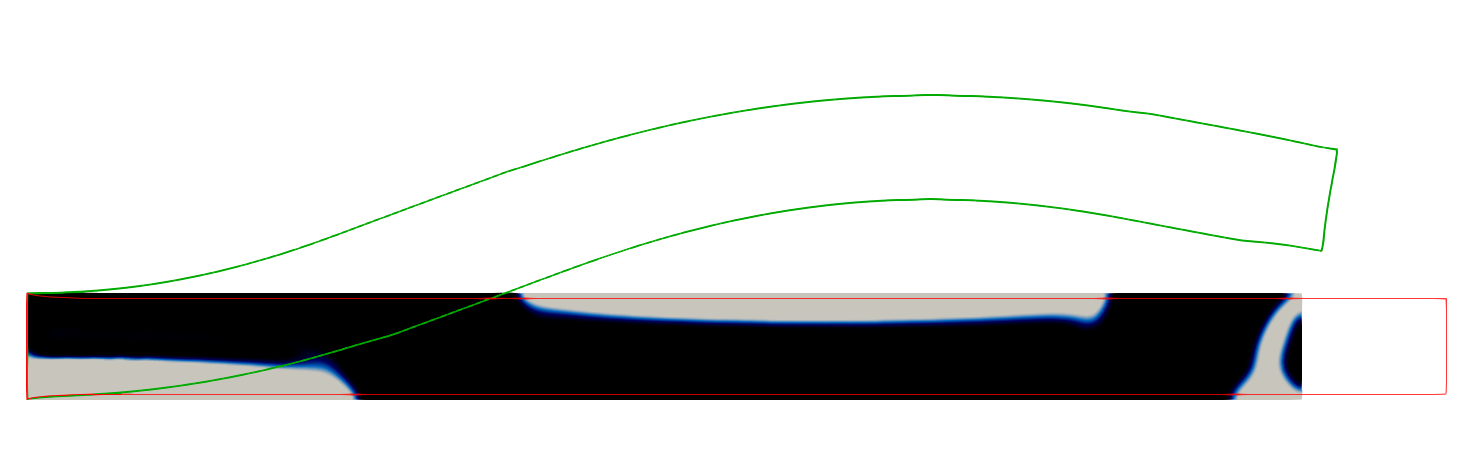} \\[1ex]
\includegraphics[angle=-0,width=0.3\textwidth]{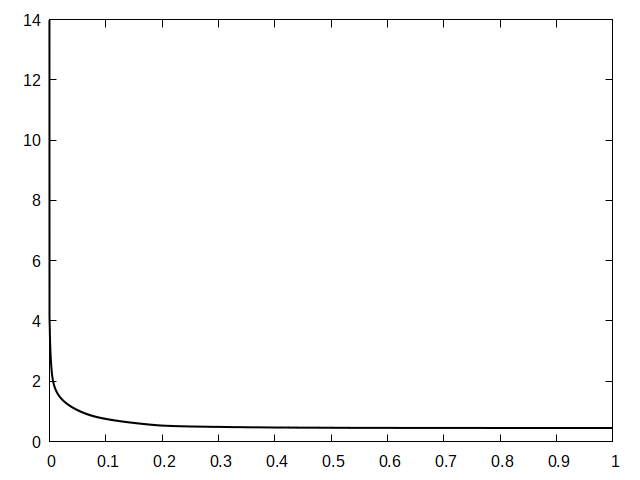}
\includegraphics[angle=-0,width=0.3\textwidth]{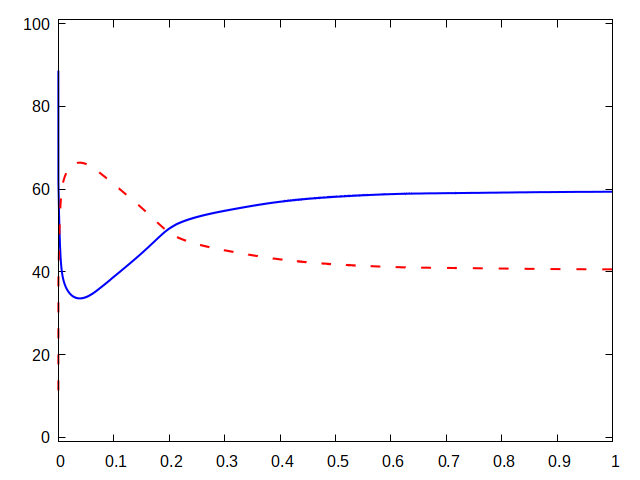}
\includegraphics[angle=-0,width=0.3\textwidth]{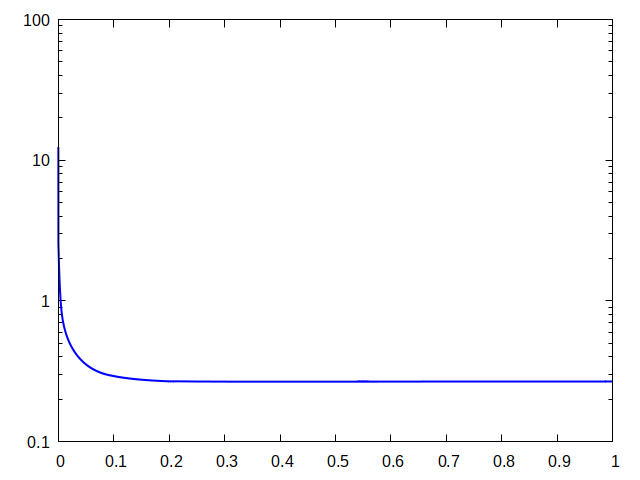}
\caption{
Computation on $\overline\Omega = [0,12] \times [-\tfrac12, \tfrac12]$
for the target shape \eqref{eq:utarget2} with $c^{\rm tar} = 1$, $k^{\rm tar} = 1.5$ 
and $W = \Id$ and $\Gamma^{\rm tar} = \partial_{\rm top}\Omega$, 
with $\eps=\frac1{8\pi}$ and $\gamma=0.01$. 
We display $\vp_h^n$ and the 
displacements $\bu_h^n$ (red) and $\hu_h^n$ (green) at pseudo-times 
$t=0.1,\,0.5,\,1$.
Below we show plots of the cost functional $J^h(\vp_h^n, \hu_h^n)$,
the proportion in it of the elastic energy $\mathcal{E}^{\rm h,tar}(\hu_h^n)$
(solid blue) and the interfacial energy $\gamma \mathcal{E}^h(\varphi_h^n)$ 
(dashed red), as well as of $\log_{10} \mathcal{E}^{\rm h,tar}(\hu_h^n)$.
}
\label{fig:2d8pi_cos121kickp15t_g001_Wid}
\end{figure}%

\subsection{Numerical simulations in three dimensions}
For the target shapes we consider a function of the form
\[
\bm{u}^{\rm tar}(x_1,x_2,x_3) = u^{\rm tar}(x_1,x_2) {\bm e}_3
\]
and choose $u^{\rm tar}$ from one of the following: 
\begin{equation} \label{eq:3dutarget}
u^{\rm tar}(x_1,x_2) = c^{\rm tar} (x_1)^2.
\end{equation}
\begin{equation} \label{eq:3dutarget2}
u^{\rm tar}(x_1,x_2) = c^{\rm tar} \Big (1- \cos \frac{k^{\rm tar}\pi x_1}{L_1} \Big ) .
\end{equation}
\begin{equation} \label{eq:3dutarget1}
u^{\rm tar}(x_1,x_2) = c^{\rm tar} x_1  x_2 .
\end{equation}
Notice that \eqref{eq:3dutarget} and \eqref{eq:3dutarget2} are simply the three dimensional analogues of the parabolic profile \eqref{eq:utarget} and the cosine profile \eqref{eq:utarget2}, respectively.  On the other hand, \eqref{eq:3dutarget1} yields a linear profile in $x_1$ with a twisting in the $x_2$-direction. 

In each of the Figures~\ref{fig:3d4pi_61kicktop_g001_Wid}, \ref{fig:3d4pi_cos1231f1top_g01}, and \ref{fig:3d4pi_cos121kick_g001t_Wid}, \ref{fig:3d2pi_linear661_g002twist} we provide visualisations of the numerical solution $\vp_h^n$ at various pseudo-times (black denotes the passive material $\{\vp_h^n = -1\}$ and grey denotes the active material $\{\vp_h^n = 1\}$), and the corresponding displacements $\hu_h^n$ (darker colours indicating lower values of $\hu_h^n \cdot \bm{e}_3$ and lighter colours indicating higher values of $\hu_h^n \cdot \bm{e}_3$).  We also provide pseudo-time plots of the energy functionals, similarly to the 2D simulations in the previous subsection. The parameter details are summarised in the Table~\ref{tbl:2}.

\begin{table}[h]
\centering
\begin{tabular}{c|c|c|c|c|c|c}
Figure & Profile & Domain & $W$ & $\Gamma^{\rm tar}$ & $c^{\rm tar}$ & $ k^{\rm tar}$ \\
\hline
\ref{fig:3d4pi_61kicktop_g001_Wid} &\eqref{eq:3dutarget} & $[0,6] \times [-\frac{3}{2}, \frac{3}{2}] \times [0,1]$ & $\Id$ & $\pd_{\rm top} \Omega$ & $0.075$ & - \\[1ex]
\ref{fig:3d4pi_cos1231f1top_g01} & \eqref{eq:3dutarget2}  & $[0,12] \times [-\frac{3}{2}, \frac{3}{2}] \times [0,1]$ & $\bm{e}_3 \otimes \bm{e}_3$ & $\pd_{\rm top} \Omega$ & 1 & 2 \\[1ex]
\ref{fig:3d4pi_cos121kick_g001t_Wid} & \eqref{eq:3dutarget2} & $[0,12] \times [-\frac{3}{2}, \frac{3}{2}] \times [0,1]$ & $\Id$ & $\pd_{\rm top} \Omega$ & $1$ & $2$  \\[1ex]
\ref{fig:3d2pi_linear661_g002twist} &  \eqref{eq:3dutarget1} & $[0,6] \times [-3,3] \times [0,1]$ & $\bm{e}_3 \otimes \bm{e}_3$ & $\pd_{\rm right} \Omega$ & $0.1$ & - 
\end{tabular}
\caption{Parameter details for numerical simulations in 3D.}
\label{tbl:2}
\end{table}

In the first three figures we take $\eps = \frac{1}{4\pi}$ and for $\gamma$ choose either $0.1$ or $0.01$. In each simulation the cost functional decreases monotonically, but the proportions of the two energies (elastic vs interfacial) differ from case to case.

The first simulation is a direct 3D analogue for the computation previously
shown in Figure~\ref{fig:2d8pi_61kick_g001_Wid}, the only difference being that
here we restrict the set $\Gamma^{\rm tar}$ to the upper part of the boundary
$\partial\Omega$. As expected, the observed results are very close to the ones
seen previously in the 2D setting. In particular, the active phase occupies
the lower half of the domain, with a dip towards the right end of the domain.
\begin{figure}
\center
\includegraphics[angle=-0,width=0.3\textwidth]{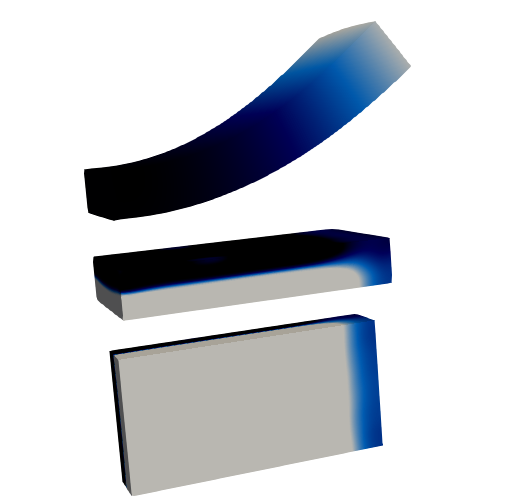}
\includegraphics[angle=-0,width=0.3\textwidth]{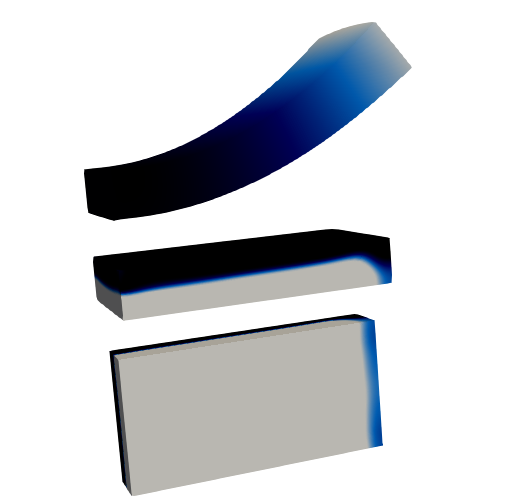}
\includegraphics[angle=-0,width=0.3\textwidth]{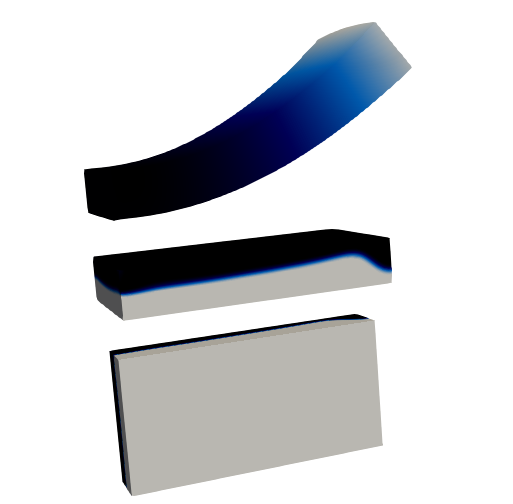} \\[1ex]
\includegraphics[angle=-0,width=0.3\textwidth]{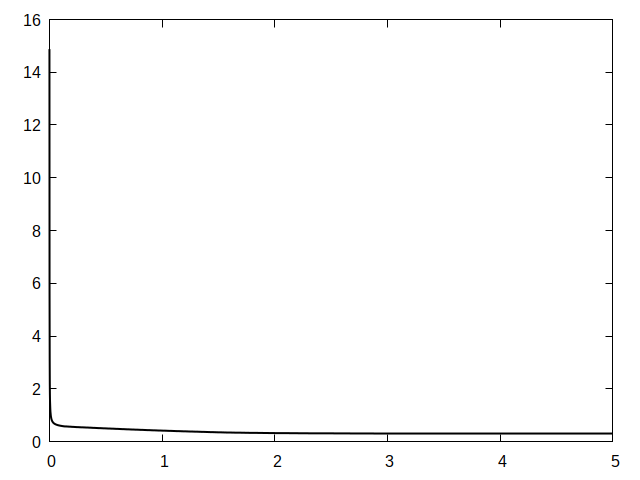}
\includegraphics[angle=-0,width=0.3\textwidth]{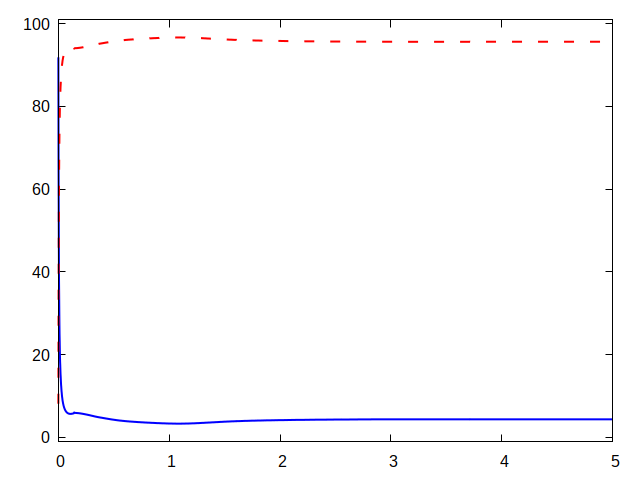}
\includegraphics[angle=-0,width=0.3\textwidth]{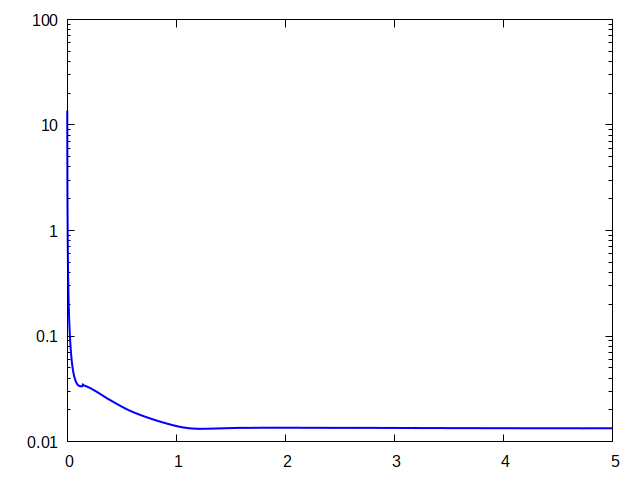}
\caption{
Computation on $\overline\Omega = [0,6] \times [-\tfrac32, \tfrac32] 
\times [0,1]$
for the target shape \eqref{eq:3dutarget} with $c^{\rm tar} = 0.075$ and 
$W = \Id$ and $\Gamma^{\rm tar} = \partial_{\rm top}\Omega$, 
with $\eps=\frac1{4\pi}$ and $\gamma=0.01$.
We display $\vp_h^n$ (side view and bottom view) 
and the displacement $\hu_h^n$ (with colour coding
for $\hu_h^n \cdot \bm{e}_3$) at pseudo-times 
$t=1,\,2,\,5$.
Below we show plots of the cost functional $J^h(\vp_h^n, \hu_h^n)$,
the proportion in it of the elastic energy $\mathcal{E}^{\rm h,tar}(\hu_h^n)$
(solid blue) and the interfacial energy $\gamma \mathcal{E}^h(\varphi_h^n)$ 
(dashed red), as well as of $\log_{10} \mathcal{E}^{\rm h,tar}(\hu_h^n)$.
}
\label{fig:3d4pi_61kicktop_g001_Wid}
\end{figure}%

On more elongated domains we once again observe that relatively little active
material can result in large deformations at the programmed stage. For example, 
the strategic placement of the active component seen in 
Figure~\ref{fig:3d4pi_cos1231f1top_g01} yields a large cosine profile
deformation.
\begin{figure}
\center
\includegraphics[angle=-0,width=0.3\textwidth]{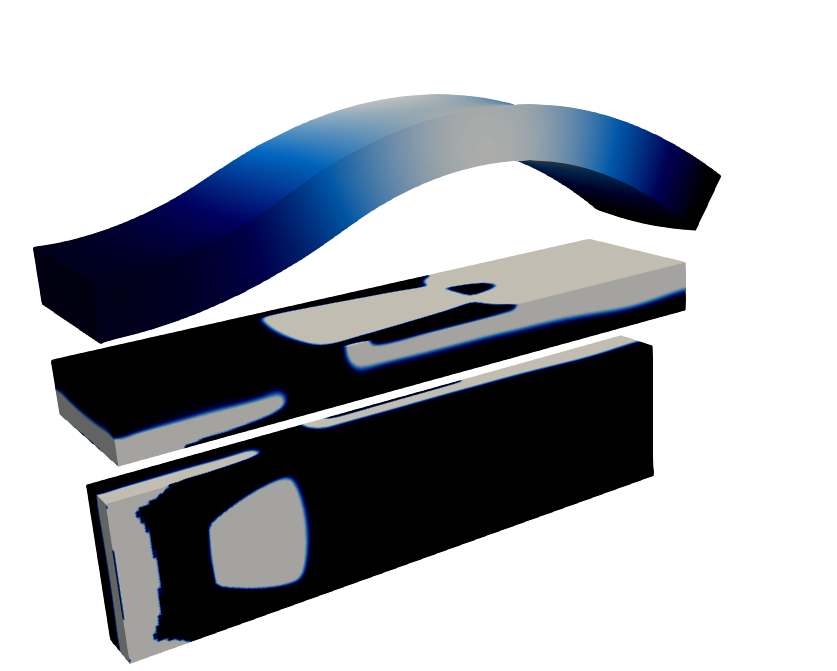}
\includegraphics[angle=-0,width=0.3\textwidth]{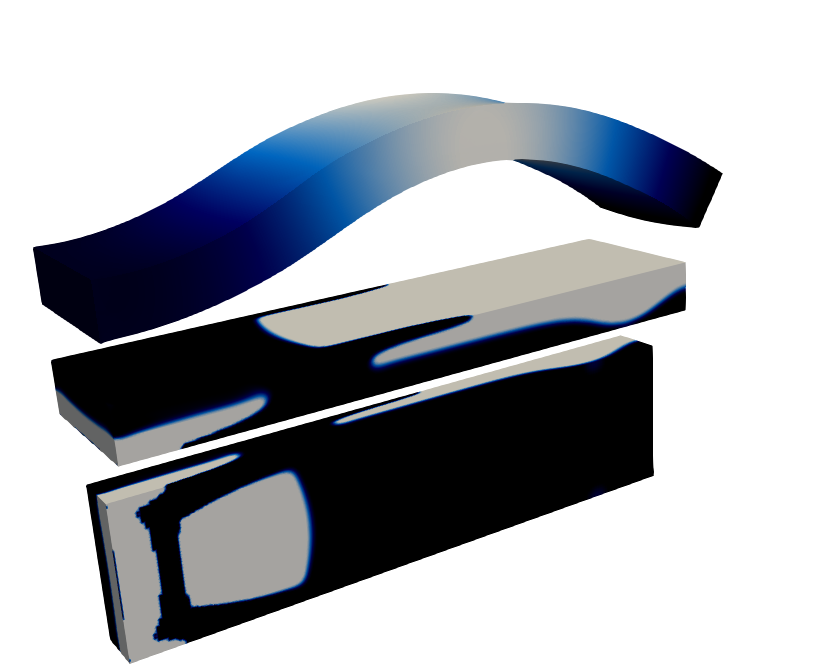}
\includegraphics[angle=-0,width=0.3\textwidth]{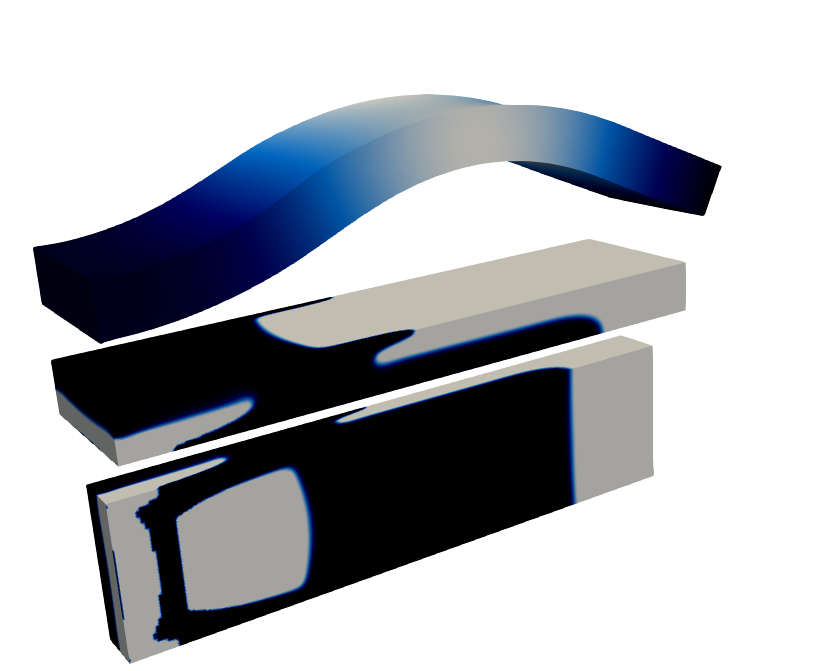} \\[1ex]
\includegraphics[angle=-0,width=0.3\textwidth]{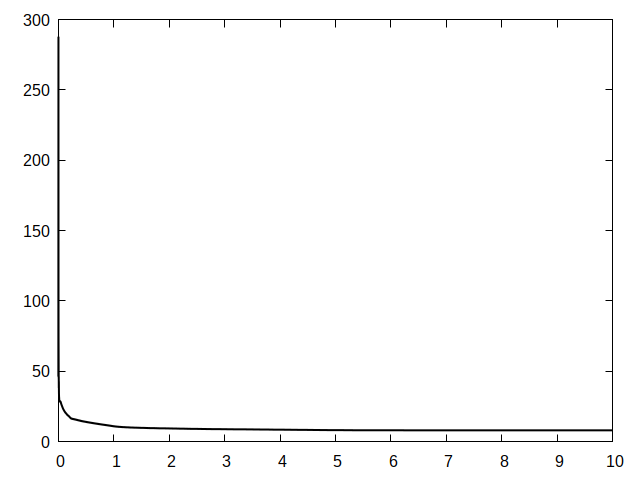}
\includegraphics[angle=-0,width=0.3\textwidth]{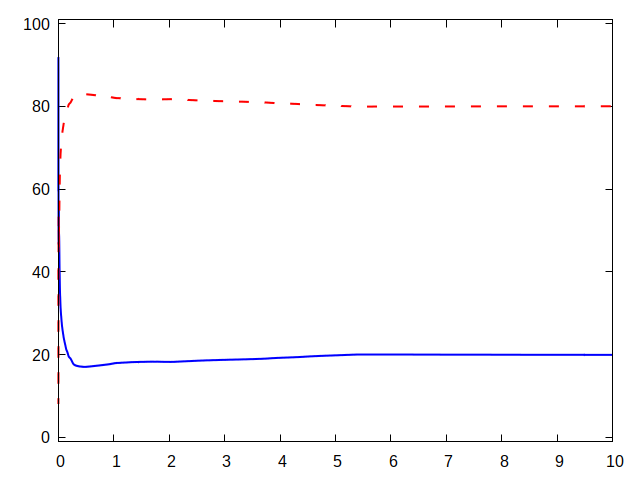}
\includegraphics[angle=-0,width=0.3\textwidth]{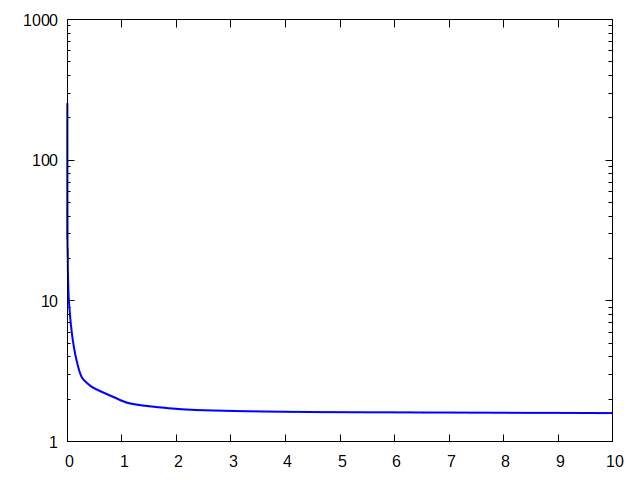}
\caption{%
Computation on 
$\overline\Omega = [0,12] \times [-\tfrac32, \tfrac32] \times [0,1]$
for the target shape \eqref{eq:3dutarget2} with $c^{\rm tar} = 1$,
$k^{\rm tar} = 2$ 
and $W = \bm{e}_3 \otimes \bm{e}_3$ and $\Gamma^{\rm tar} = \partial_{\rm top}\Omega$, 
with $\eps=\frac1{4\pi}$ and $\gamma=0.1$. 
We display $\vp_h^n$ (side view and bottom view) 
and the displacement $\hu_h^n$ (with colour coding
for $\hu_h^n \cdot \bm{e}_3$) at pseudo-times $t=1,\,2,\,10$.
Below we show plots of the cost functional $J^h(\vp_h^n, \hu_h^n)$,
the proportion in it of the elastic energy $\mathcal{E}^{\rm h,tar}(\hu_h^n)$
(solid blue) and the interfacial energy $\gamma \mathcal{E}^h(\varphi_h^n)$ 
(dashed red), as well as of $\log_{10} \mathcal{E}^{\rm h,tar}(\hu_h^n)$.
}
\label{fig:3d4pi_cos1231f1top_g01}
\end{figure}%
It is interesting to note that by simply changing the weighting matrix $W$ in
the target energy functional, we obtain a completely different optimal
distribution of active material. This can be seen in 
Figure~\ref{fig:3d4pi_cos121kick_g001t_Wid}, where the only change to the
previous simulation is $W = \Id$, rather than $W = \bm{e}_3 \otimes \bm{e}_3$.
Now there are just two connected components for the active phase, one at the
lower left part of the domain, and one at the middle of the top of the domain.
Yet the obtained deformations at the programmed stage are very similar.
\begin{figure}
\center
\includegraphics[angle=-0,width=0.3\textwidth]{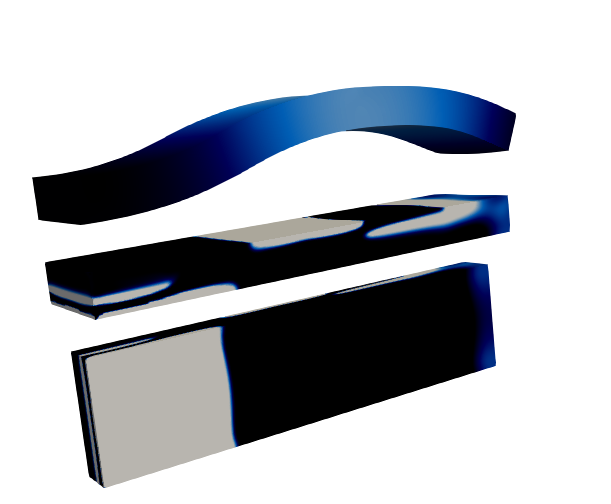}
\includegraphics[angle=-0,width=0.3\textwidth]{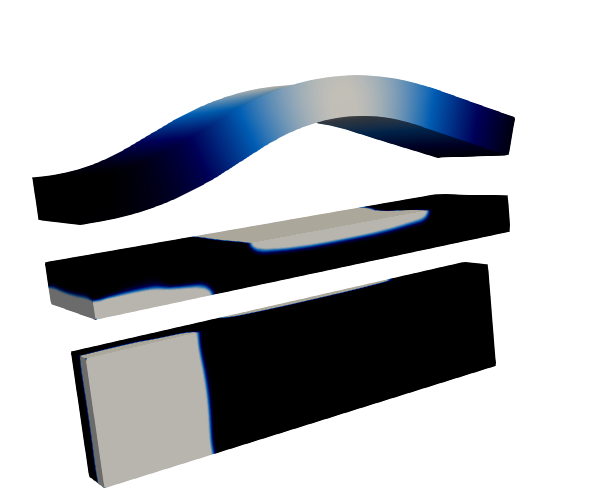}
\includegraphics[angle=-0,width=0.3\textwidth]{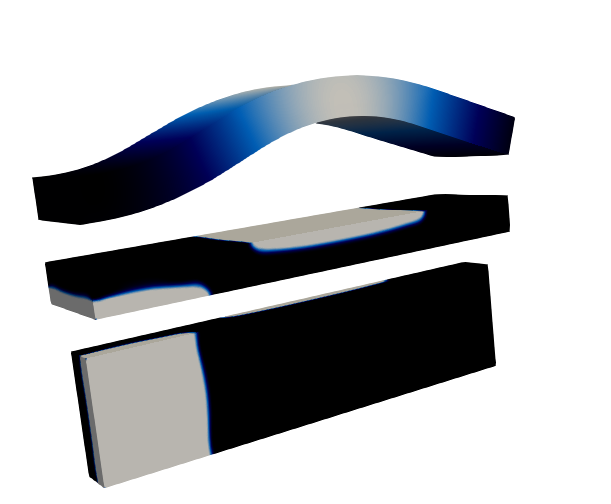} \\[1ex]
\includegraphics[angle=-0,width=0.3\textwidth]{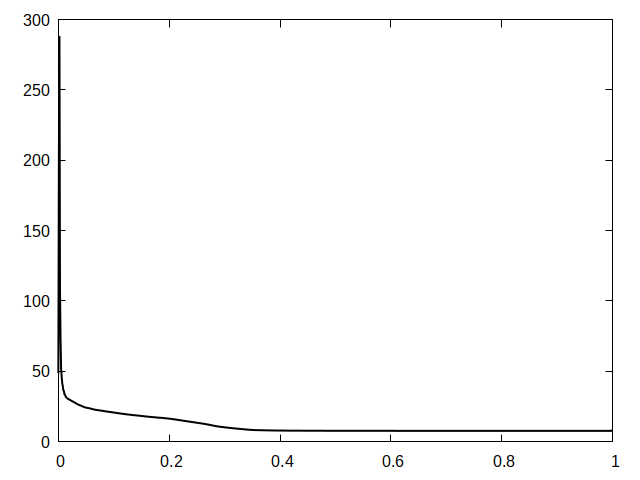}
\includegraphics[angle=-0,width=0.3\textwidth]{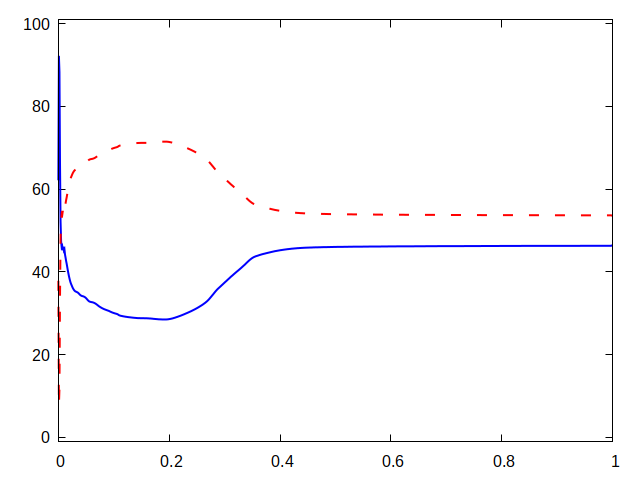}
\includegraphics[angle=-0,width=0.3\textwidth]{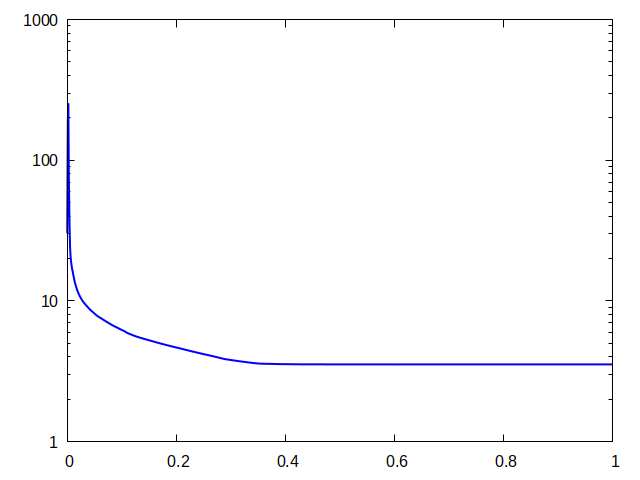}
\caption{
Computation on $\overline\Omega = [0,12] \times [-\tfrac32, \tfrac32] 
\times [0,1]$
for the target shape \eqref{eq:3dutarget2} with $c^{\rm tar} = 1$, $k^{\rm tar}=2$ 
and $W = \Id$ and $\Gamma^{\rm tar} = \partial_{\rm top}\Omega$, 
with $\eps=\frac1{4\pi}$ and $\gamma=0.01$.
We display $\vp_h^n$ (side view and bottom view) 
and the displacement $\hu_h^n$ (with colour coding
for $\hu_h^n \cdot \bm{e}_3$) at pseudo-times 
$t=0.1,\,0.5,\,1$.
Below we show plots of the cost functional $J^h(\vp_h^n, \hu_h^n)$,
the proportion in it of the elastic energy $\mathcal{E}^{\rm h,tar}(\hu_h^n)$
(solid blue) and the interfacial energy $\gamma \mathcal{E}^h(\varphi_h^n)$ 
(dashed red), as well as of $\log_{10} \mathcal{E}^{\rm h,tar}(\hu_h^n)$.
}
\label{fig:3d4pi_cos121kick_g001t_Wid}
\end{figure}%

Our final numerical simulation is for a twisted target shape. In particular,
we use the target function \eqref{eq:3dutarget1} with $c^{\rm tar} = 0.1$ 
on the domain $\overline\Omega = [0,6] \times [-3, 3] \times [0,1]$, with
$W = \bm{e}_3 \otimes \bm{e}_3$ and 
$\Gamma^{\rm tar} = \partial_{\rm right}\Omega$.
In Figure~\ref{fig:3d2pi_linear661_g002twist}
we provide visualisations of the numerical solution $\vp_h^n$ at various pseudo-times, and the corresponding displacements $\hu_h^n$ (here darker colours indicate lower values of $|\hu_h^n|$ and lighter colours indicate higher values of $|\hu_h^n|$). For this computation we take $\eps=\frac{1}{2\pi}$ and $\gamma=0.02$.
At the optimal configuration we observe an elaborate distribution of the active
material, which yields a twisted shape of the component at the programmed
stage.
\begin{figure}
\center
\includegraphics[angle=-0,width=0.3\textwidth]{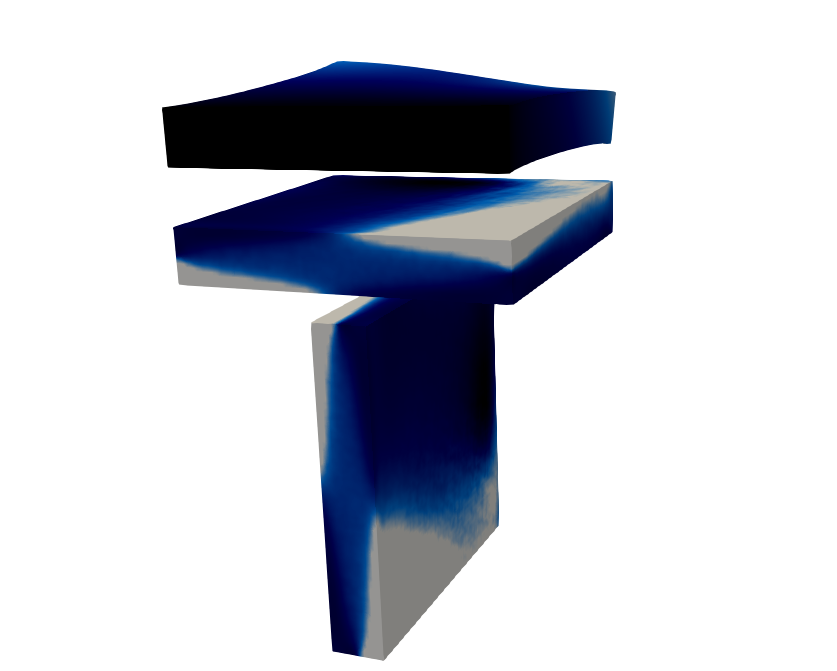}
\includegraphics[angle=-0,width=0.3\textwidth]{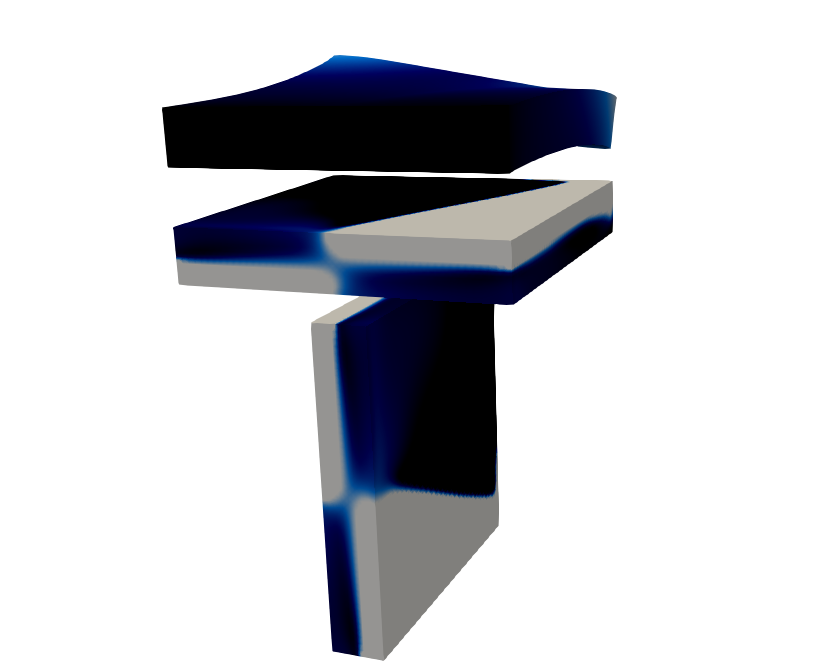}
\includegraphics[angle=-0,width=0.3\textwidth]{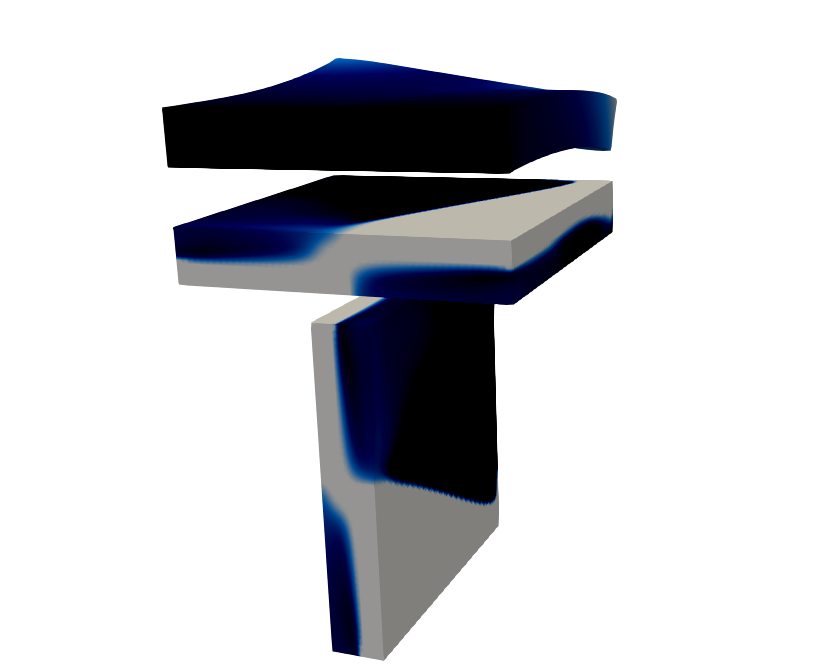}
\includegraphics[angle=-0,width=0.3\textwidth]{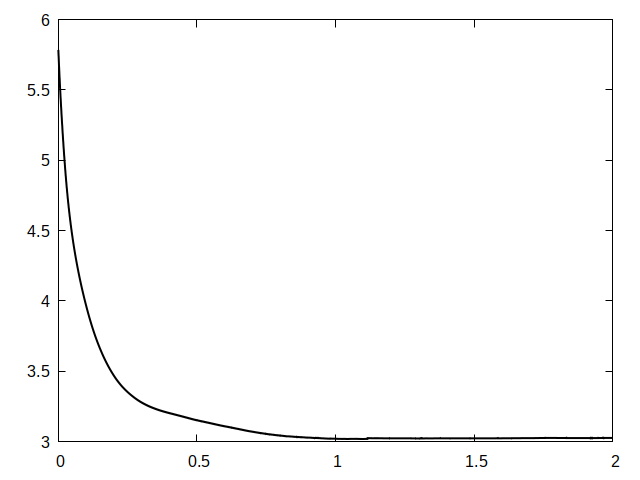}
\includegraphics[angle=-0,width=0.3\textwidth]{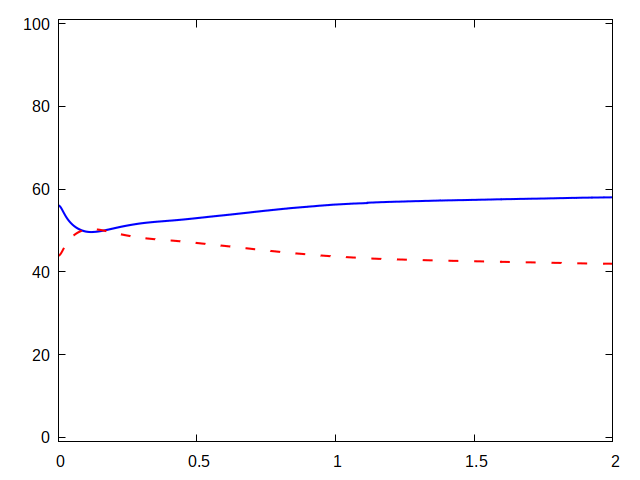}
\includegraphics[angle=-0,width=0.3\textwidth]{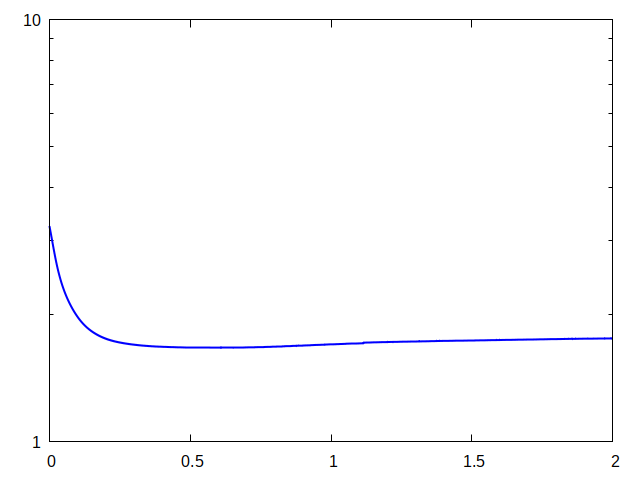}
\includegraphics[angle=-0,width=0.3\textwidth]{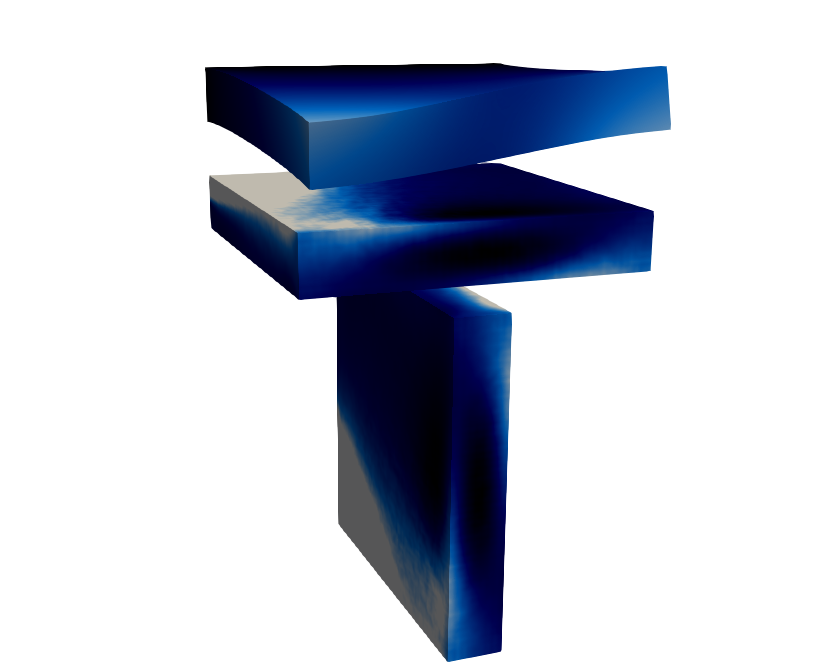}
\includegraphics[angle=-0,width=0.3\textwidth]{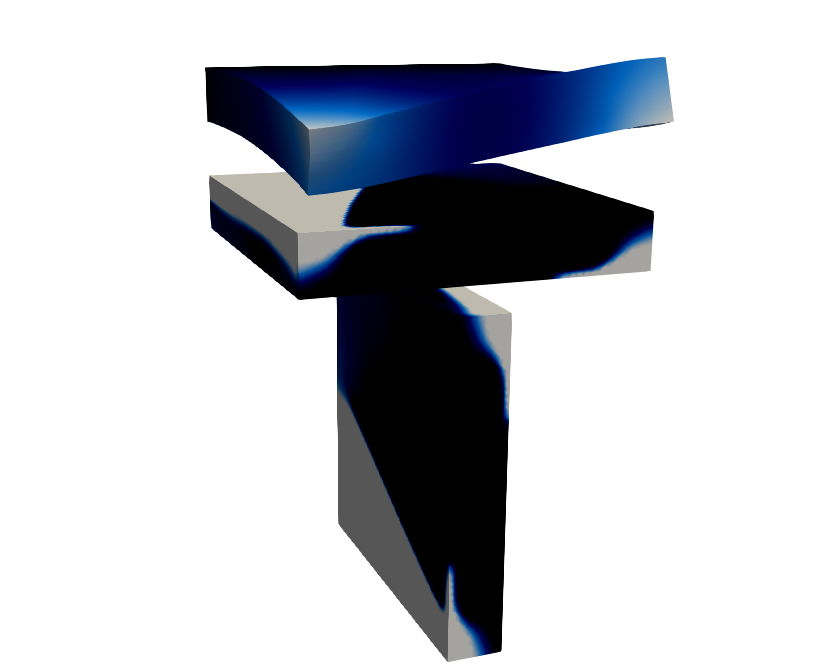}
\includegraphics[angle=-0,width=0.3\textwidth]{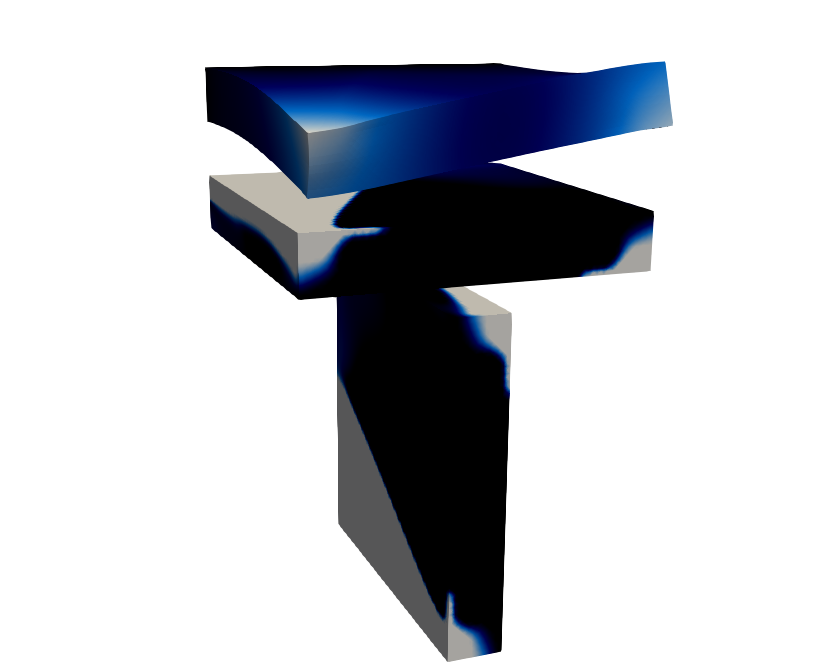}
\caption{%
Computation on 
$\overline\Omega = [0,6] \times [-3, 3] \times [0,1]$
for the target shape \eqref{eq:3dutarget1}(b) with $c^{\rm tar} = 0.1$ 
and $W = \bm{e}_3 \otimes \bm{e}_3$ and 
$\Gamma^{\rm tar} = \partial_{\rm right}\Omega$,
with $\eps=\frac1{2\pi}$ and $\gamma=0.02$. 
We display $\vp_h^n$ and the displacement $\hu_h^n$ (with colour coding
for $|\hu_h^n|$) at pseudo-times $t=0.1,\ 1,\,2$.
In the middle  
we show plots of the cost functional $J^h(\vp_h^n, \hu_h^n)$,
the proportion in it of the elastic energy $\mathcal{E}^{\rm h,tar}(\hu_h^n)$
(solid blue) and the interfacial energy $\gamma \mathcal{E}^h(\varphi_h^n)$ 
(dashed red), as well as of $\log_{10} \mathcal{E}^{\rm h,tar}(\hu_h^n)$.
In the bottom we show a backward view of the first row of plots.
}
\label{fig:3d2pi_linear661_g002twist}
\end{figure}%

\section*{Acknowledgments}
\noindent The authors HG and AS gratefully acknowledge the support by the Graduiertenkolleg 2339 IntComSin of the Deutsche Forschungsgemeinschaft (DFG, German Research Foundation) -- Project-ID 321821685. The work of KFL is supported by the Research Grants Council of the Hong Kong Special Administrative Region, China [Project No.: HKBU 14302218 and HKBU 12300321].

\footnotesize
\bibliographystyle{plain}

\end{document}